\documentclass[12pt]{amsart}

\usepackage{amsfonts, amssymb, amsmath, latexsym,  mathtools, amscd}
\usepackage{hyperref,graphicx}
\usepackage{enumerate}% http://ctan.org/pkg/enumerate
\usepackage[dvipsnames]{xcolor}
\usepackage{url}
\usepackage{wrapfig}
\usepackage{bbm}    %  Used for mathbb 1

\usepackage{vmargin}
\setpapersize{USletter}
\setmargrb{2.5cm}{2cm}{2.5cm}{2.cm} % --- sets all four margins.  Cool.
\newtheorem{Theorem}{Theorem}[section]
\newtheorem{Proposition}[Theorem]{Proposition}
\newtheorem{DefProp}[Theorem]{Definition-Proposition}
\newtheorem{Lemma}[Theorem]{Lemma}
\newtheorem{Corollary}[Theorem]{Corollary}

\theoremstyle{remark}
\newtheorem{remark}[Theorem]{Remark}
\newenvironment{Remark}
{\pushQED{\qed}\remark}
{\popQED\endremark}

\newtheorem{example}[Theorem]{Example}
\newenvironment{Example}
{\pushQED{\qed}\example}
{\popQED\endexample}

\newtheorem{definition}[Theorem]{Definition}
\newenvironment{Definition}
{\pushQED{\qed}\definition}
{\popQED\enddefinition}

\newcommand{\CC}{{\mathbb C}}
\newcommand{\NN}{{\mathbb N}}
\newcommand{\PP}{{\mathbb P}}
\newcommand{\RR}{{\mathbb R}}
\newcommand{\TT}{{\mathbb T}}
\newcommand{\ZZ}{{\mathbb Z}}

\newcommand{\bOne}{\mathbbm{1}}

\newcommand{\calA}{{\mathcal A}}
\newcommand{\calE}{{\mathcal E}}
\newcommand{\calF}{{\mathcal F}}
\newcommand{\calG}{{\mathcal G}}
\newcommand{\calH}{{\mathcal H}}
\newcommand{\calN}{{\mathcal N}}
\newcommand{\calO}{{\mathcal O}}
\newcommand{\calV}{{\mathcal V}}

\newcommand{\Nvert}{N_{{\rm vert}}}
\newcommand{\Ndisc}{N_{{\rm disc}}}
\newcommand{\hGamma}{\widehat{\Gamma}}

\newcommand{\Ad}{{\text{\rm Ad}}}
\newcommand{\area}{{\rm area}}

\newcommand{\bfzero}{{\bf  0}}

\newcommand\codim{{\rm codim}}
\newcommand{\conv}{{\text{\rm conv}}}
\newcommand\cpdeg{{\rm cpdeg}}
\newcommand{\BV}{{\textit{BV}}\hspace{0.5pt}}

\newcommand{\Hom}{{\text{\rm Hom}}}
\newcommand{\ini}{\text{\rm in}}

\newcommand{\MA}{{\text{\rm MA}}}
\newcommand{\sgn}{\text{\rm sgn}}
\newcommand\vol{{\rm vol}}
\newcommand\nvol{{\rm nvol}}
\newcommand\rank{{\rm rank}}
\newcommand{\Sat}{{\text{\rm Sat}}}
\newcommand{\Var}{{\textit{Var}}\hspace{0.5pt}}
\newcommand\wt{{\rm wt}}

\newcommand{\hooklongrightarrow}{\lhook\joinrel\longrightarrow}
\newcommand{\longtwoheadrightarrow}{\relbar\joinrel\twoheadrightarrow}

\newcommand{\defcolor}[1]{{\color{blue}#1}}
\newcommand{\demph}[1]{\defcolor{{\sl #1}}}

\definecolor{TAMU}{RGB}{140,0,0}
\definecolor{myblue}{RGB}{0,0,198}
\definecolor{myred}{RGB}{182,0,0}

\title{The Critical Point Degree of a Periodic Graph}

\author{M.~Faust}
\address{Matthew Faust, Department of Mathematics,
         Michigan State University, East Lansing, MI 48832, USA}
\email{mfaust@msu.edu}
\urladdr{https://mattfaust.github.io/}
\author{J.~Robinson}
\address{Jonah Robinson, Department of Mathematics,
         University of Michigan, Ann Arbor, MI 48109,  USA}
\email{jonahrob@umich.edu}
\urladdr{https://jonah-robinson.github.io/}
\author{F. Sottile}
\address{Frank Sottile, Department of Mathematics,
         Texas A\&M University, College Station, Texas 77843,  USA}
\email{sottile@tamu.edu}
\urladdr{https://franksottile.github.io}
\thanks{Research supported in part by NSF grants DMS-2052519, DMS-2052572, DMS-2201005, and Simons grant TSM-00014009}
\subjclass[2020]{14M25, 47A75, 81Q10, 52B20, 05C50}
\keywords{Bloch Variety,Toric Variety,  Dispersion Polynomial,  Newton Polytope}
\begin{document}

\begin{abstract}
 The critical point degree of a periodic graph operator is the number of critical points of its complex Bloch variety.
 Determining it is a step towards the spectral edges conjecture and more generally understanding Bloch varieties.
 Previous work showed that it is bounded above by the volume of the Newton polytope of the graph, and that the inequality is 
 strict when there are asymptotic critical points.
 We identify contributions from asymptotic critical points that arise from the structure of the graph,
 and show that the critical point degree is bounded above by the difference of the volume of the Newton polytope and these
 contributions.
 These results have implications for nonlinear optimization.
\end{abstract}

\maketitle 

%%%%%%%%%%%%%%%%%%%%%%%%%%%%%%%%%%%%%%%%%%%%%%%%%%%%%%%%%%%%%%%%%%%%%%%%%%%%%%%%%%%%%%%%%%%%%%%%%%%%
\section*{Introduction}
Operators on $\ZZ^d$-periodic graphs are an abstraction of physical models of crystals.
The spectrum of such an operator is a union of intervals in $\RR$, which Floquet theory reveals to be the
image of the Bloch variety of the operator under a coordinate projection.
The Bloch variety may be understood as an algebraic hypersurface in $(S^1)^d\times\RR$ and this coordinate projection to
$\RR$ is a function $\lambda$ whose critical values include the edges of the spectral intervals.

The complexification of the Bloch variety is an algebraic hypersurface in $(\CC^\times)^d\times\CC$ defined by the
dispersion polynomial. 
The \demph{critical point degree} of the operator is the number of complex critical points of the function $\lambda$
on the Bloch variety, counted with multiplicity. 
Results of~\cite{FS} show that the critical point degree is bounded above by the normalized volume $\nvol(A)$ of the
Newton polytope $A$ of the dispersion polynomial, and   
this bound is not attained if the Newton polytope has vertical faces or if the Bloch variety has asymptotic singularities.
This is reviewed in Section~\ref{S:CPDPO}.

We improve that bound.
Corollary~\ref{C:verticalContribution} identifies a contribution \defcolor{$\Nvert$} from vertical faces of $A$.
This is asymptotic behavior as $\lambda$ remains bounded, but the Floquet multipliers become unbounded.
In  Section~\ref{S:initGraph} we identify a structure of the graph (disconnected initial graph) which implies that the
Bloch variety has asymptotic singularities.
Corollary~\ref{C:obliqueContribution} identifies their contribution,  \defcolor{$\Ndisc$}.
For these, both $\lambda$ and the Floquet multipliers become unbounded.
Both contributions $\Nvert$ and $\Ndisc$ arise from structural aspects of the underlying graph,
which we study in Section~\ref{S:initGraph}.
We state a simplified version of our main theorem.\medskip

\noindent{\bf Theorem~\ref{Th:main}.}
{\it For $d=1$, $2$, or $3$, the critical point degree of a Bloch variety of an operator on a $\ZZ^d$-periodic graph
   is at most $\nvol(A)-\Nvert-\Ndisc$.}\medskip

This restriction to $d\leq 3$ is because the singular locus of a complex Bloch variety often has dimension at least $d-3$.
(See Remark~\ref{R:SingularLocus}.)

Studying critical points of the function $\lambda$ on the Bloch variety is a topic of current
interest~\cite{Berk,DKS,FS,FS25,Liu22}. 
One motivation is the \demph{spectral edges conjecture} from mathematical physics, which concerns the structure
of the Bloch variety above the edges of the spectral bands.
It posits that for a sufficiently general operator $H$, the local extrema of the function $\lambda$ on the real Bloch variety
are nondegenerate in that their Hessians are nondegenerate quadratic forms, and they occur in distinct bands.
This is stated more precisely in~\cite[Conj.\ 5.25]{KuchBAMS}, and it also appears
in~\cite{CdV,fk2, KuchBook,Nov81,Nov83}. 
Important notions, such as effective mass in solid state physics, the Liouville property, Green's
function asymptotics, Anderson localization, homogenization, and many other assumed properties in physics, depend upon this
conjecture. 
This conjecture and some recent progress is discussed in~\cite{DKS, FS,FS25}.
Determining the critical point degree for a general operator was important to those results.

This work is related to nonlinear optimization.
The results in~\cite{DKS,FS} were among the first to show that well-known combinatorial bounds (BKK bound~\cite{BKK}
for general sparse systems) for
the number of solutions to certain polynomial optimization problems are attained, despite the corresponding systems of
polynomials being far from general.
This directly inspired~\cite{BSW} and subsequent works~\cite{LT,LMR,LNRW,TT24}, which  have shown that this phenomenon
is quite common.
This is related to recent work~\cite{E24,KKS,Sel25} refining the BKK bound for structured polynomial systems.
The structural understanding of the critical point degree that we derive here may indicate the next steps in this program.

Section~\ref{S:one} provides some background on periodic graph operators and their Bloch varieties.
In Section~\ref{S:initGraph} we study asymptotic properties of the dispersion polynomial and define initial graphs.
Section~\ref{S:toric} develops technical aspects of projective toric varieties, including local properties and singularities,
and we prove Theorem~\ref{Th:main}, using results from the subsequent two sections.
Section~\ref{S:vertical} studies the contribution $\Nvert$ from vertical faces.
Finally, in Section~\ref{S:Structural} we study the contribution $\Ndisc$ of asymptotic singularities that arise when
an initial graph of $\Gamma$ is disconnected.
We end with Example~\ref{Ex:singularity} suggesting further refinements to our analysis of asymptotic critical points
using the structure of the graph. 
This approach of understanding the asymptotic behavior of the Bloch variety/operator has long been used to study
operators on periodic graphs~\cite{Baettig,FLiu,FL-G,FaustKachkovskiy,FLM22,flm23,fk2,GKT,Liu22}. 

We derive another result of interest.
Theorem~\ref{Th:InitialGraphForm} describes restrictions of the
dispersion polynomial to faces of its Newton polytope in terms of initial subgraphs of $\Gamma$.
This refines~\cite[Lem.~4.2]{FLiu} and will be used in~\cite{FLSS}.

%%%%%%%%%%%%%%%%%%%%%%%%%%%%%%%%%%%%%%%%%%%%%%%%%%%%%%%%%%%%%%%%%%%%%%%%%%%%%%%%%%%%%%%%%%%%%%%%%%%%
\section{Algebraic Aspects of Discrete Periodic Operators}\label{S:one}

Let $n$ and $d$ be positive integers.
Let $\ZZ$, $\RR$, $\CC$, $\TT$ be, respectively, the integers, the real numbers, the complex numbers,
and the unit complex numbers. 
Unless otherwise noted, a graph will be simple (no loops or multiple edges),
undirected, and have bounded vertex degree.
A more comprehensive development is found in~\cite[Ch.~4]{BerkKuch} and~\cite{AAPGO}.

%%%%%%%%%%%%%%%%%%%%%%%%%%%%%%%%%%%%%%%%%%%%%%%%%%%%%%%%%%%%%%%%%%%%%%%%%%%%%%%%%%%%%%%%%%%%%%%%%%%%
\subsection{Discrete Periodic Operators}\label{S:DPO}
A $\ZZ^d$-periodic graph \defcolor{$\Gamma$} is a graph equipped with a free cocompact action of $\ZZ^d$.
Cocompactness means there are finitely many orbits of vertices  \defcolor{$\calV(\Gamma)$} and of  edges \defcolor{$\calE(\Gamma)$}.
An edge between vertices $u$ and $v$ is written $(u,v)$; we have $(u,v)=(v,u)$, as $\Gamma$ is undirected,
and when $(u,v)\in\calE(\Gamma)$ we may write $u\defcolor{\sim}v$.
Figure~\ref{F:three_graphs} shows three $\ZZ^2$-periodic graphs.
%%%%%%%%%%%%%%%%%%%%%%%%%%%%%%%%%%%%%%%%%%%%%%%%%%%%%%%%%%%%%%%%%%%%%%%%%%%%%%%%%%%%%%%%%%%%%%%%%%%%
\begin{figure}[htb]
  \centering
  \raisebox{-40pt}{\includegraphics[height=90pt]{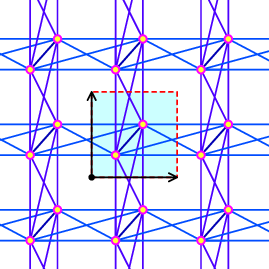}} \quad\qquad
  \raisebox{-45pt}{\includegraphics[height=100pt]{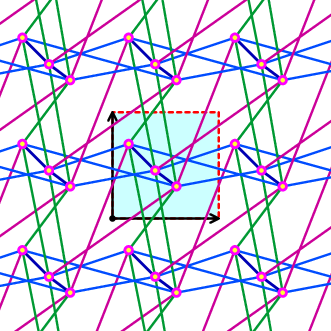}} \quad\qquad
  \raisebox{-30pt}{\includegraphics[height=70pt]{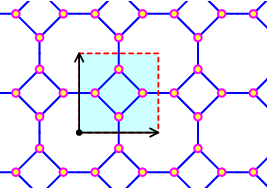}}
 \caption{Three periodic graphs.}\label{F:three_graphs}
\end{figure}
%%%%%%%%%%%%%%%%%%%%%%%%%%%%%%%%%%%%%%%%%%%%%%%%%%%%%%%%%%%%%%%%%%%%%%%%%%%%%%%%%%%%%%%%%%%%%%%%%%%%  
  %
 These have 2, 3, and 4  orbits of $\ZZ^2$ on $\calV(\Gamma)$, respectively.

A periodic graph $\Gamma$ is a discrete model of a crystal whose vertices represent atoms and edges
represent interactions among atoms.
Interaction strengths are modeled by a $\ZZ^d$-periodic function $\defcolor{e}\colon\calE(\Gamma)\to\RR$ called edge weights, and we fix a
$\ZZ^d$-periodic function $\defcolor{V}\colon\calV(\Gamma)\to\RR$, a \demph{potential}.
The tight-binding model~\cite{AshcroftMermin,Kittel} involves the discrete Schr\"odinger operator \defcolor{$H$} acting on
functions $f\colon\calV(\Gamma)\to\CC$.
This is the sum of a multiplication operator and a weighted graph Laplacian.
The function $Hf$ has value at $v\in\calV(\Gamma)$,
\begin{equation}\label{Eq:DPSO}
  (Hf)(v)\ =\  V(v)\cdot f(v)\ +\ \sum_{v\sim u} e_{(v,u)}(f(v)-f(u))\,.
\end{equation}
Let \defcolor{$\ell_2(\Gamma)$} be the Hilbert space of square summable functions $f\colon\calV(\Gamma)\to\CC$.
Then $H$ is a bounded self-adjoint operator on $\ell_2(\Gamma)$ whose spectrum $\sigma(H)$ is a finite union of closed and bounded intervals in
$\RR$, giving the familiar structure of energy bands and band gaps.

%%%%%%%%%%%%%%%%%%%%%%%%%%%%%%%%%%%%%%%%%%%%%%%%%%%%%%%%%%%%%%%%%%%%%%%%%%%%%%%%%%%%%%%%%%%%%%%%%%%%
\begin{Remark}
  For  $v\in\calV(\Gamma)$ we absorb the sum $\sum_{v\sim u} e_{(v,u)}$ into $V(v)$.
  Then~\eqref{Eq:DPSO} becomes
 \begin{equation}\label{Eq:reducedSum}
    (Hf)(v)\ =\  V(v)\cdot f(v)\ -\ \sum_{v\sim u} e_{(v,u)}f(u)\,. \qedhere %eqno{\diamond}
 \end{equation}
\end{Remark}
%%%%%%%%%%%%%%%%%%%%%%%%%%%%%%%%%%%%%%%%%%%%%%%%%%%%%%%%%%%%%%%%%%%%%%%%%%%%%%%%%%%%%%%%%%%%%%%%%%%%

More structure of the spectrum is revealed by Floquet theory~\cite[Ch.~4]{BerkKuch},~\cite[Sect.~1.3]{KorotyaevSabruova14},
or~\cite{KuchBook}. 
As $V$ and $e$ are $\ZZ^d$-periodic, the operator $H$ commutes with the $\ZZ^d$-action.
The points $z=(z_1,\dotsc,z_d)\in\TT^d$ are unitary characters of $\ZZ^d$ under the map
$(z,\alpha)\mapsto \defcolor{z^\alpha}\vcentcolon= z_1^{\alpha_1}\dotsb z_d^{\alpha_d}\in\TT$.
Let us consider the Fourier transform of functions in $\ell_2(\Gamma)$.
Let \defcolor{$L^2(\TT^d)$} be the Hilbert space of square-integrable functions on $\TT^d$.
Writing  \defcolor{$\alpha+v$} for the action of $\alpha\in\ZZ^d$ on a vertex
$v\in\calV(\Gamma)$, the Fourier transform of  $f\colon\calV(\Gamma)\to\CC$ is a function
$\defcolor{\hat{f}}\colon\calV(\Gamma)\to L^2(\TT^d)$ that is quasi-periodic as follows.
For $v\in\calV(\Gamma)$ and $\alpha\in\ZZ^d$,
 \begin{equation}\label{Eq:quasi-periodic}
   \hat{f}(\alpha+v)\ =\ z^\alpha\cdot\hat{f}(v)\ ,
 \end{equation}
as functions of the \demph{Floquet multipliers} $z=(z_1,\dotsc,z_d)\in\TT^d$.
Note that $z=e^{\sqrt{-1}\, k}$, where $k\in\RR^d$ are quasimomenta.

By~\eqref{Eq:quasi-periodic}, a quasiperiodic function $\hat{f}$ is determined by its values at one vertex in each
$\ZZ^d$-orbit on $\calV(\Gamma)$.
Let $\defcolor{W}\subset\calV(\Gamma)$ be such a choice of orbit representatives, called a \demph{fundamental domain}.
These are highlighted for  the graphs in Figure~\ref{F:three_graphs}. % ~\eqref{Eq:Graphs_sample}.
Then the Fourier transform is an isometry $\ell_2(\Gamma)\xrightarrow{\,\sim\,} L^2(\TT^d)^W$, where $L^2(\TT^d)^W$ is
the space of functions from
$W$ to $L^2(\TT^d)$. 

As $H$ commutes with the $\ZZ^d$-action on $\Gamma$, for $\hat{f}\colon W\to L^2(\TT^d)$, the value of $H\hat{f}$ at a
vertex $v\in W$ is (using~\eqref{Eq:DPSO} and~\eqref{Eq:quasi-periodic}),
 \begin{equation}\label{Eq:Floquet_operator}
   \bigl( H \hat{f} \bigr)(v)\ =\  V(v)\cdot \hat{f}(v)\ - \ 
   \sum_{v\sim \alpha+u} e_{(v,\alpha+u)}\, z^\alpha\hat{f}(u)\ .
 \end{equation}
 %

%%%%%%%%%%%%%%%%%%%%%%%%%%%%%%%%%%%%%%%%%%%%%%%%%%%%%%%%%%%%%%%%%%%%%%%%%%%%%%%%%%%%%%%%%%%%%%%%%%%%
\begin{Example}\label{Ex:Hexagon_operator}
  Let  $\Gamma$ be the hexagonal lattice shown in Figure~\ref{F:localHexagon} with 
  a labeling in a neighborhood of its fundamental domain.
 %%%%%%%%%%%%%%%%%%%%%%%%%%%%%%%%%%%%%%%%%%%%%%%%%%%%%%%%%%%%%%%%%%%%%%%%%%%%%%%%%%%%%%%%%%%%%%%%%%%%%
  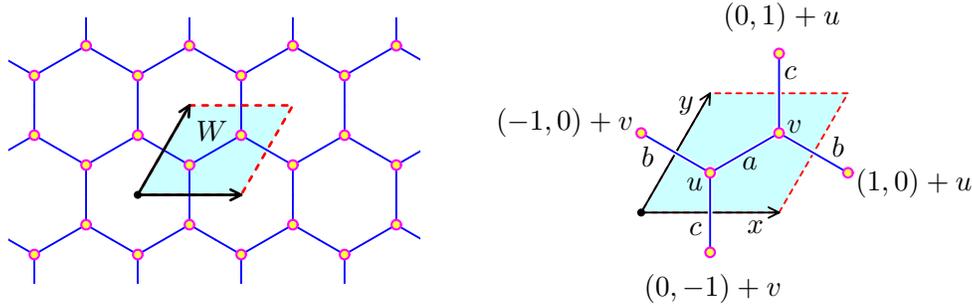
\begin{figure}[htb]
    \centering
   \begin{picture}(157,110)(0,-5)
     \put(0,0){\includegraphics{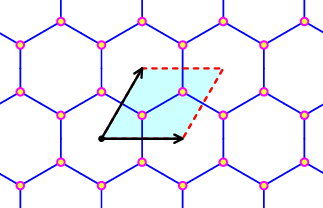}}
     \put(71,54){\small$W$}
   \end{picture}
    \qquad
   \begin{picture}(179,110)(-41,-9)
     \put(0,0){\includegraphics{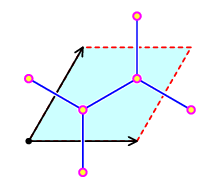}}
     \put(31,31){\small$u$}   \put( 95,31){\small$(1,0)+u$}    \put(45,94){\small$(0,1)+u$}
     \put(69,53){\small$v$}   \put(-41,54){\small$(-1,0)+v$}   \put(15,-9){\small$(0,-1)+v$}
     \put(54,15){\small$x$}
     \put(28,62){\small$y$}
     \put(52,39){\small$a$}
     \put(32,14){\small$c$}   \put(68,73){\small$c$}
     \put(14,40){\small$b$}   \put(86,44){\small$b$}
   \end{picture}
   \caption{The hexagonal lattice and a labeling in a neighborhood of $W$.}
   \label{F:localHexagon}
 \end{figure}
 %%%%%%%%%%%%%%%%%%%%%%%%%%%%%%%%%%%%%%%%%%%%%%%%%%%%%%%%%%%%%%%%%%%%%%%%%%%%%%%%%%%%%%%%%%%%%%%%%%%%%
  Thus, $W=\{u,v\}$ consists of two vertices and there are three (orbits of) edges, with labels $a,b,c$.
  Let $(x,y)\in\TT^2$.
  The operator $H$ is
 \begin{align*}
   (H\hat{f})(u)\  &=\ V(u)\hat{f}(u)\ -\ a\hat{f}(v)\ -\
                          bx^{-1}\hat{f}(v)\ -\ cy^{-1}\hat{f}(v)\ ,\\
   (H\hat{f})(v)\  &=\ V(v)\hat{f}(v)\ -\ a\hat{f}(u)\ -\
                         bx\hat{f}(u)\ \ \ \  -\  cy\hat{f}(u)\ .
 \end{align*}
 Collecting coefficients of $\hat{f}(u)$ and $\hat{f}(v)$, we represent $H$ by the $2\times 2$-matrix,
 \begin{equation}\label{Eq:Hexagon_Matrix}
    H\ =\ H(x,y)\ =\ \left( \begin{matrix}
          V(u)       &-a-b x^{-1}-c y^{-1}\\
         -a-b x-c y  & V(v)
              \end{matrix}\right)\ ,
 \end{equation}
 whose entries are Laurent polynomials in $x,y$.
 For $(x,y)\in\TT^2$, $(\overline{x},\overline{y})=(x^{-1},y^{-1})$.
 Then $H^T=\overline{H}$, so that $H$ is Hermitian.
\end{Example}
%%%%%%%%%%%%%%%%%%%%%%%%%%%%%%%%%%%%%%%%%%%%%%%%%%%%%%%%%%%%%%%%%%%%%%%%%%%%%%%%%%%%%%%%%%%%%%%%%%%%

The \demph{Floquet matrix $H(z)$} is the $W\times W$ matrix of Laurent polynomials in $z$ whose rows and columns are
indexed by elements of $W$, and whose entry in row $v$ and column $u$ is
  \begin{equation}
   \label{Eq:Floquet_Matrix}
  \delta_{v,u} V(v) \ -\ \sum_{v\sim\alpha+u}  e_{(v,\alpha+u)}\, z^\alpha\,.
 \end{equation}
({\it cf}.\ the right hand side of~\eqref{Eq:Floquet_operator}). 
The operator $H$ acts on $L^2(\TT^d)^W$ via multiplication by the Floquet matrix. 
Write $c=(e,V)$ for the edge weights and potential, and \defcolor{$H_c(z)$} for the Floquet matrix 
when needed to indicate the parameters.
As $e$ is $\ZZ^d$-periodic and the edges are undirected, $e_{(v,\alpha+u)}=e_{(-\alpha+v,u)}=e_{(u,-\alpha+v)}$, which
implies the symmetry, $H(z)^T=H(z^{-1})$.

%%%%%%%%%%%%%%%%%%%%%%%%%%%%%%%%%%%%%%%%%%%%%%%%%%%%%%%%%%%%%%%%%%%%%%%%%%%%%%%%%%%%%%%%%%%%%%%%%%%%
\begin{Remark}\label{Rem:Sparse}
  The potential $V(v)$ only occurs in the diagonal entry in position $(v,v)$ .
  The edge parameter $e_{(v,\alpha+u)}=e_{(u,-\alpha+v)}$ only occurs in positions $(v,u)$ and $(u,v)$
  with the term $e_{(v,\alpha+u)}\, z^\alpha$ in position $(v,u)$ and 
  $e_{(v,\alpha+u)}\, z^{-\alpha}$ in position $(u,v)$.  
\end{Remark}  
%%%%%%%%%%%%%%%%%%%%%%%%%%%%%%%%%%%%%%%%%%%%%%%%%%%%%%%%%%%%%%%%%%%%%%%%%%%%%%%%%%%%%%%%%%%%%%%%%%%%

%%%%%%%%%%%%%%%%%%%%%%%%%%%%%%%%%%%%%%%%%%%%%%%%%%%%%%%%%%%%%%%%%%%%%%%%%%%%%%%%%%%%%%%%%%%%%%%%%%%%
\subsection{Bloch and Fermi Varieties}\label{S:BLoch}
The \demph{dispersion polynomial} $\Phi(z,\lambda)=\Phi_c(z,\lambda)$ is the characteristic polynomial of the Floquet
matrix $H(z)$, 
 \begin{equation}\label{Eq:dispersionPolynomial}
   \defcolor{\Phi_{c}(z,\lambda)}\ \vcentcolon=\ 
   \det\left(\lambda {\textit Id}_W\ -\ H_{c}(z)\right)\,.
 \end{equation}
This is a monic polynomial in $\lambda$ of degree $|W|$ whose coefficients are Laurent polynomials in $z$.
The dispersion polynomial of the Floquet matrix $H(x,y)$ of~\eqref{Eq:Hexagon_Matrix} is
 \begin{multline}
   \qquad \lambda^2\ -\lambda\bigl(V(u)+V(v)\bigr)  +  V(u)V(v) - (a^2+b^2+c^2) \label{Eq:DispersionPOlynomiaHexagonalLattice}\\
   -\ ab(x+x^{-1})\ -\ ac(y+y^{-1})\ -\ bc(xy^{-1} + yx^{-1})\ . \qquad
 \end{multline}

The \demph{Bloch variety} $\BV=\BV_{c}\subset\TT^d\times\RR$ is defined by the dispersion polynomial
 \[
 \defcolor{\BV_{c}}\ =\ \Var(\Phi_{c})\ :=\
 \{(z,\lambda)\in \TT^d\times\RR \mid \Phi_{(e,V)}(z,\lambda)\ =\ 0\}\,. 
 \]
(We write \defcolor{$\Var(f)$} for the vanishing set of a function $f$.)
Figure~\ref{F:HexagonalAndBloch} shows three Bloch varieties for the hexagonal lattice from Figure~\ref{F:localHexagon}.
%%%%%%%%%%%%%%%%%%%%%%%%%%%%%%%%%%%%%%%%%%%%%%%%%%%%%%%%%%%%%%%%%%%%%%%%%%%%%%%%%%%%%%%%%%%%%%%%%%%%
\begin{figure}[htb]
  \centering
  \begin{picture}(115,110)
    \put(0,0){\includegraphics[height=110pt]{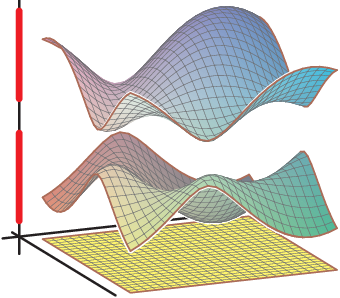}}
  \end{picture}
  \qquad\quad
  \begin{picture}(115,110)
    \put(0,0){\includegraphics[height=110pt]{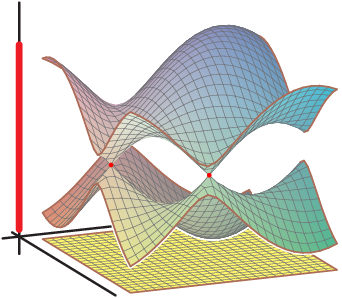}}
  \end{picture}
  \qquad\quad
  \begin{picture}(115,110)
        \put(0,0){\includegraphics[height=110pt]{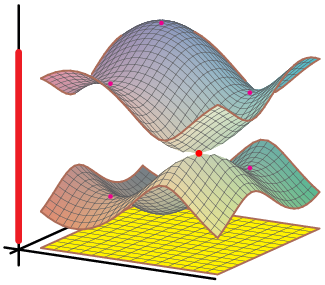}}
  \end{picture}

    \caption{Three  Bloch varieties for the hexagonal lattice}
  \label{F:HexagonalAndBloch}
\end{figure}
%%%%%%%%%%%%%%%%%%%%%%%%%%%%%%%%%%%%%%%%%%%%%%%%%%%%%%%%%%%%%%%%%%%%%%%%%%%%%%%%%%%%%%%%%%%%%%%%%%%%
The parameters $(a,b,c,u,v)$ of each are $(1,1,1,1,0)$, $(1,1,1,0,0)$, and $(5,3,2,0,0)$, respectively.
These are displayed over a square representing the fundamental domain $[-\frac{\pi}{2},\frac{3\pi}{2}]^2$ of
$\TT^2=\RR^2/(2\pi\ZZ)^2$, and the spectra $\sigma(H)$ are the thickened intervals along the vertical axes.
Each Bloch variety is a 2-sheeted cover of $\TT^2$.

For a given $z\in\TT^d$, the points $\lambda$ such that $(z,\lambda)\in\BV$ are the eigenvalues of the matrix $H(z)$.
Since $z^{-1}=\overline{z}$ (the complex conjugate of $z$) and $H(z)^T=H(z^{-1})$, the Floquet matrix $H(z)$ is Hermitian.
Consequently, it has $|W|$ real eigenvalues, which implies that the Bloch variety
$\BV\subset\TT^d\times\RR$ is a $|W|$-sheeted cover of $\TT^d$ (counting multiplicities of eigenvalues).
Figure~\ref{F:more_BV} shows Bloch varieties corresponding to the graphs of Figure~\ref{F:three_graphs}. %~\eqref{Eq:Graphs_sample}.
%%%%%%%%%%%%%%%%%%%%%%%%%%%%%%%%%%%%%%%%%%%%%%%%%%%%%%%%%%%%%%%%%%%%%%%%%%%%%%%%%%%%%%%%%%%%%%%%%%%%
\begin{figure}[htb]
  \centering
   \raisebox{-50pt}{\includegraphics[height=100pt]{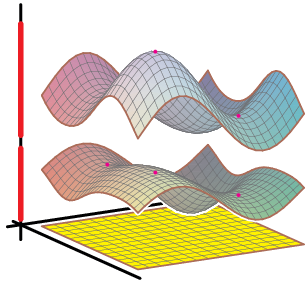}}\qquad        
   \raisebox{-55pt}{\includegraphics[height=140pt]{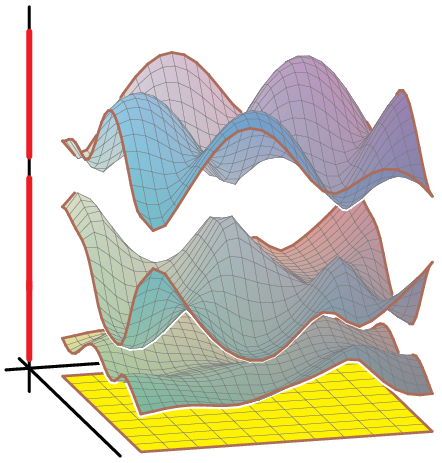}}\qquad 
   \raisebox{-60pt}{\includegraphics[height=130pt]{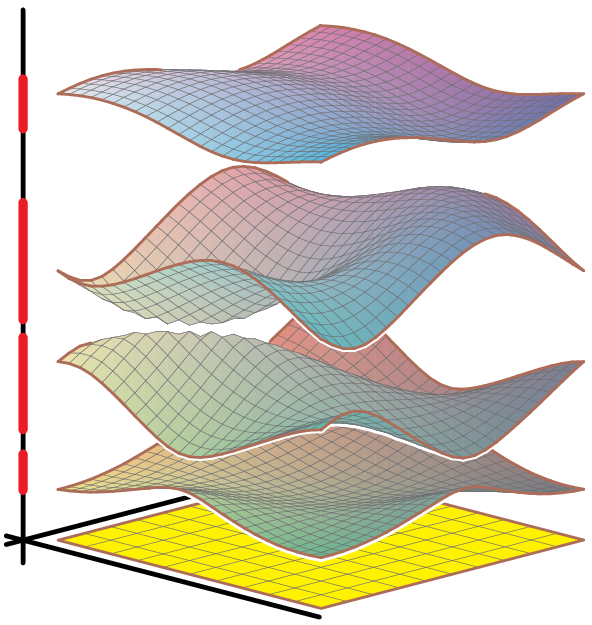}}          

    \caption{Three more Bloch varieties}
  \label{F:more_BV}
\end{figure}
%%%%%%%%%%%%%%%%%%%%%%%%%%%%%%%%%%%%%%%%%%%%%%%%%%%%%%%%%%%%%%%%%%%%%%%%%%%%%%%%%%%%%%%%%%%%%%%%%%%%
The $j$th sheet is the graph of the $j$th \demph{band function} $\lambda_j\colon\TT^d\to\RR$, whose image is a
spectral band of $H$.
Rather than treating these individually, we instead (see~\cite{DKS,FS,AAPGO}) consider the restriction of the coordinate function
$\lambda$ to the Bloch variety, which is a function $\lambda\colon \BV\to\RR$.
Then the spectrum of $H$ is the image of the Bloch variety under this function $\lambda$.
The spectra in Figure~\ref{F:more_BV} are the thickened intervals along the vertical axes.

The endpoints of the intervals are the \demph{spectral edges} of $H$.
These are images of local extrema of the function $\lambda$ on the Bloch variety.
The \demph{spectral edges conjecture}~\cite[Conj.\ 5.25]{KuchBAMS} posits, among other things,
that the Hessian of $\lambda$ is a nondegenerate quadratic form at each extremum.
This is implied by the critical points conjecture~\cite{FS}: all critical points of $\lambda$ are isolated and multiplicity 1 (and consequently nonsingular).
By Lemma~5.1 of~\cite{FS}, a nonsingular critical point has a nondegenerate Hessian.

%%%%%%%%%%%%%%%%%%%%%%%%%%%%%%%%%%%%%%%%%%%%%%%%%%%%%%%%%%%%%%%%%%%%%%%%%%%%%%%%%%%%%%%%%%%%%%%%%%%%
\begin{Remark}\label{R:SingularLocus}
  The singular locus of the determinant hypersurface has codimension 4 in $\textrm{Mat}_{n\times n}$.
  Thus, we expect that the singular locus of the Bloch variety has dimension at least $d{+}1-4=d{-}3$.
  When $d\geq 4$, this is positive, giving infinitely many critical points.
  As singular points are critical points, counting critical points often becomes moot when $d\geq 4$.
  The paper~\cite{FS25} concerns cases when $d>3$ and yet the critical points are discrete.
\end{Remark}
%%%%%%%%%%%%%%%%%%%%%%%%%%%%%%%%%%%%%%%%%%%%%%%%%%%%%%%%%%%%%%%%%%%%%%%%%%%%%%%%%%%%%%%%%%%%%%%%%%%%

%%%%%%%%%%%%%%%%%%%%%%%%%%%%%%%%%%%%%%%%%%%%%%%%%%%%%%%%%%%%%%%%%%%%%%%%%%%%%%%%%%%%%%%%%%%%%%%%%%%%

If we write the dispersion polynomial $\Phi(z,\lambda)$~\eqref{Eq:dispersionPolynomial} as a linear combination of monomials,
 \begin{equation}\label{Eq:DispersionPolynomial}
   \Phi(z,\lambda)\ =\ \sum_{(\alpha,j)\in\ZZ^d\times\NN} c_{(\alpha,j)} z^\alpha \lambda^j\ ,
 \end{equation}
the set \defcolor{$\calA(\Phi)$} of exponent vectors $(\alpha,j)$ with non-zero coefficients $c_{(\alpha,j)}$ is its \demph{support}.
For general parameters $V,e$, the origin $\bfzero$ lies in $\calA(\Phi)$ as the product of the potentials occurs in $\Phi$.
The dispersion polynomial for the hexagonal lattice has support the columns of the matrix
\[
   \left(\begin{matrix}  0&0& 0& 1& 1& 0&-1&-1& 0\\
                         0&0& 1& 0&-1&-1& 0& 1& 0\\
                         2&1& 0& 0& 0& 0& 0& 0& 0\end{matrix}\right)\ .
\]

The convex hull $\defcolor{\calN(\Phi)}=\conv(\calA(\Phi))$ of the support is the \demph{Newton polytope} of $\Phi$.
Figure~\ref{F:moreNewtonPolytopes} shows
Newton polytopes for dispersion polynomials from the hexagonal lattice and the graphs
of Figure~\ref{F:three_graphs}. % ~\eqref{Eq:Graphs_sample}.
All are symmetric with respect to the vertical axis as $\Phi(z,\lambda)=\Phi(z^{-1},\lambda)$,
which follows from $H(z)^T=H(z^{-1})$.
%%%%%%%%%%%%%%%%%%%%%%%%%%%%%%%%%%%%%%%%%%%%%%%%%%%%%%%%%%%%%%%%%%%%%%%%%%%%%%%%%%%%%%%%%%%%%%%%%%%%
\begin{figure}[htb] \centering
   \raisebox{-50pt}{\includegraphics[height=70pt]{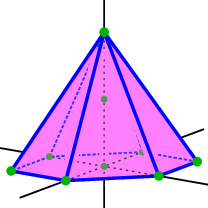}} \quad 
   \raisebox{-52.45pt}{\includegraphics[height=70pt]{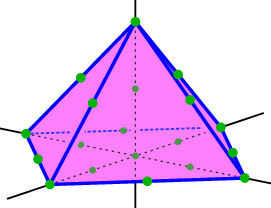}} \quad
   \raisebox{-59.8pt}{\includegraphics[height=105pt]{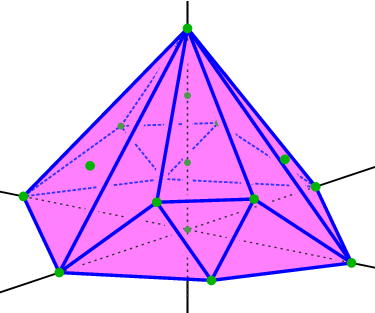}}\quad
   \raisebox{-50pt}{\includegraphics[height=112pt]{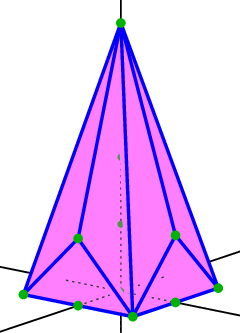}} 
\caption{Some Newton polytopes}\label{F:moreNewtonPolytopes}
\end{figure}  
%

%%%%%%%%%%%%%%%%%%%%%%%%%%%%%%%%%%%%%%%%%%%%%%%%%%%%%%%%%%%%%%%%%%%%%%%%%%%%%%%%%%%%%%%%%%%%%%%%%%%%
\subsection{Critical Points of Discrete Periodic Operators}\label{S:CPDPO}

As the Bloch variety is defined by a polynomial $\Phi$ with real coordinates, it is natural to consider its
complexification. 
This is the set of points $(z,\lambda)$ in $(\CC^\times)^d\times\CC$ such that $\Phi(z,\lambda)=0$.
This complexified Bloch variety is the Zariski closure of the real Bloch variety defined in Section~\ref{S:BLoch}.

A goal of~\cite{DKS,FS} was to give bounds for the number of critical points and use that to prove the critical points
conjecture for some periodic graphs.
Our goal is to refine those bounds.
We begin with  equations that define the set of critical points of
$\lambda$ on the Bloch variety.

%%%%%%%%%%%%%%%%%%%%%%%%%%%%%%%%%%%%%%%%%%%%%%%%%%%%%%%%%%%%%%%%%%%%%%%%%%%%%%%%%%%%%%%%%%%%%%%%%%%%
\begin{Proposition}[Prop.~2.1 of~\cite{FS}]
  A point $(z,\lambda)\in(\CC^\times)^d\times \CC$ is a critical point of the function $\lambda$ on the Bloch variety 
  if and only if it is a solution to the system of equations
 \begin{equation}\label{Eq:CPE}
   \Phi(z,\lambda)
   \ =\ z_1\frac{\partial \Phi(z,\lambda)}{\partial z_1}
   \ =\ z_2\frac{\partial \Phi(z,\lambda)}{\partial z_2}
   \ =\ \ \dotsb\ 
   \ =\ z_d\frac{\partial \Phi(z,\lambda)}{\partial z_d}
   \ =\ 0\,.
 \end{equation}
\end{Proposition}
%%%%%%%%%%%%%%%%%%%%%%%%%%%%%%%%%%%%%%%%%%%%%%%%%%%%%%%%%%%%%%%%%%%%%%%%%%%%%%%%%%%%%%%%%%%%%%%%%%%%

Write \defcolor{$\Psi$} for this system of polynomial equations~\eqref{Eq:CPE}, called the \demph{critical point equations}.
The number of critical points (counted with multiplicity, see Proposition~\ref{P:Degree_Section}) on a
Bloch variety is its \demph{critical point degree}.
This is 0 if there are infinitely many critical points.
Standard arguments in algebraic geometry imply that there is a (Zariski) dense open subset $U$ of the space of parameters
for operators on a graph $\Gamma$ with the following property:
All operators with parameters from $U$ have the same critical point degree, and this common degree is maximal among all operators on
$\Gamma$.
This maximum is called the \demph{critical point degree of\/ $\Gamma$}, written \defcolor{$\cpdeg(\Gamma)$}.
The following is a consequence of~\cite[Thm.\ 16]{DKS} and~\cite[Lem.\ 5.1]{FS}
(by the same arguments as~\cite[Thm.\ 5.2]{FS}).

%%%%%%%%%%%%%%%%%%%%%%%%%%%%%%%%%%%%%%%%%%%%%%%%%%%%%%%%%%%%%%%%%%%%%%%%%%%%%%%%%%%%%%%%%%%%%%%%%%%%
\begin{Theorem}\label{Th:SEC_CPD}
  If one operator on a periodic graph $\Gamma$ has $\cpdeg(\Gamma)$ distinct critical points,
  then the spectral edges conjecture holds for $\Gamma$. 
\end{Theorem}
%%%%%%%%%%%%%%%%%%%%%%%%%%%%%%%%%%%%%%%%%%%%%%%%%%%%%%%%%%%%%%%%%%%%%%%%%%%%%%%%%%%%%%%%%%%%%%%%%%%%

There is a well-known lower bound for the critical point degree (when it is nonzero).
As $\Phi(z,\lambda)=\Phi(z^{-1},\lambda)$, we have
 \begin{equation}\label{Eq:partial_reciprocal}
   \frac{\partial\Phi}{\partial z_i}(z,\lambda)\ =\ 
   -\frac{1}{z_i^2}\frac{\partial\Phi}{\partial z_i}(z^{-1},\lambda)
   \qquad\mbox{for } i=1,\dotsc,d\,.
 \end{equation}
A fixed point of the involution $z\mapsto z^{-1}$ is a 2-torsion point ($z^2=1$) and is called a
\demph{corner point}.
There are $2^d$ corner points, $\{\pm 1\}^d$.
By~\eqref{Eq:partial_reciprocal}, if $z$ is a corner point, then
\[
   \frac{\partial\Phi}{\partial z_i}(z,\lambda)\ =\ 0\,,   \qquad\mbox{for } i=1,\dotsc,d\,.
\]
This implies the following well-known fact.

%%%%%%%%%%%%%%%%%%%%%%%%%%%%%%%%%%%%%%%%%%%%%%%%%%%%%%%%%%%%%%%%%%%%%%%%%%%%%%%%%%%%%%%%%%%%%%%%%%%%
\begin{Proposition}
  Any point $(z,\lambda)\in \BV$ with $z$ a corner point is a critical point.
\end{Proposition}
%%%%%%%%%%%%%%%%%%%%%%%%%%%%%%%%%%%%%%%%%%%%%%%%%%%%%%%%%%%%%%%%%%%%%%%%%%%%%%%%%%%%%%%%%%%%%%%%%%%%
  
%%%%%%%%%%%%%%%%%%%%%%%%%%%%%%%%%%%%%%%%%%%%%%%%%%%%%%%%%%%%%%%%%%%%%%%%%%%%%%%%%%%%%%%%%%%%%%%%%%%%
\begin{Corollary}\label{C:Corner_CPdeg}
  Let $\Gamma$ be a periodic graph with fundamental domain $W$.
  If an operator on $\Gamma$ has finitely many critical points, then its critical point degree is at least $2^d|W|$.
\end{Corollary}
%%%%%%%%%%%%%%%%%%%%%%%%%%%%%%%%%%%%%%%%%%%%%%%%%%%%%%%%%%%%%%%%%%%%%%%%%%%%%%%%%%%%%%%%%%%%%%%%%%%%
\begin{proof}
  For a corner point $z_0$, there are $m$ points $(z_0,\lambda)$ on the Bloch variety, counted with multiplicity.
  This gives $2^d|W|$ such critical points on the Bloch variety.
\end{proof}
%%%%%%%%%%%%%%%%%%%%%%%%%%%%%%%%%%%%%%%%%%%%%%%%%%%%%%%%%%%%%%%%%%%%%%%%%%%%%%%%%%%%%%%%%%%%%%%%%%%%

The main results in~\cite{FS} concern bounds for the critical point
degree.  Given a polytope $A$ of dimension $n$ in $\RR^n$ with vertices in the integer lattice $\ZZ^n$ (a lattice
polytope), its \demph{normalized volume} is $\defcolor{\nvol(A)}\vcentcolon=n!\vol(A)$, where \defcolor{$\vol(A)$} is the
Euclidean volume of $A$, normalized so that a fundamental domain of $\ZZ^n$ has volume one.
The normalized volume is always an integer.

%%%%%%%%%%%%%%%%%%%%%%%%%%%%%%%%%%%%%%%%%%%%%%%%%%%%%%%%%%%%%%%%%%%%%%%%%%%%%%%%%%%%%%%%%%%%%%%%%%%%
%
%  This is OK even if \ZZ\calA \neq \ZZ^{d+1}
%
\begin{Proposition}[Cor.~2.5~\cite{FS}]
  \label{P:CPD_Bound}
  The critical point degree of a discrete periodic operator with dispersion polynomial $\Phi$ is at most 
  $\nvol(\calN(\Phi))$.
\end{Proposition}
%%%%%%%%%%%%%%%%%%%%%%%%%%%%%%%%%%%%%%%%%%%%%%%%%%%%%%%%%%%%%%%%%%%%%%%%%%%%%%%%%%%%%%%%%%%%%%%%%%%%

A second result of~\cite{FS} gives conditions implying that the bound is attained.
This is in terms of the asymptotic behavior of the Bloch variety.
As we  explain in Section~\ref{S:toric}, this asymptotic behavior is controlled by the Newton polytope  $\calN(\Phi)$.
Dot product with an integer vector $\eta\in\ZZ^{d+1}$ is a linear form on $\RR^{d+1}$. 
The subset of $\calN(\Phi)$ on which this linear form is minimized is the face of $\calN(\Phi)$ \demph{exposed} by $\eta$.
Every face of $\calN(\Phi)$ is exposed by some  vector $\eta\in\ZZ^{d+1}$.
The dimension of the linear span of the vectors exposing a face $F$ is its \demph{codimension},
$\defcolor{\codim(F)}\vcentcolon=d{+}1{-}\dim(F)$.
A face $F$ of $\calN(\Phi)$ has a normalized volume, $\nvol(F)=(\dim(F))!\, \vol(F)$, where $\vol(F)$ is normalized with respect
to the intersection of $\ZZ^{d+1}$ with the affine span of $F$ (this intersection is isomorphic to $\ZZ^{\dim(F)}$).

Restricting the sum~\eqref{Eq:DispersionPolynomial} to the face $F$ exposed by $\eta$ gives the \demph{initial  form}  of $\Phi$,
 \[
   \defcolor{\ini_\eta\Phi}\ \vcentcolon=\
   \sum_{(\alpha,j)\in \calA(\Phi)\cap F} c_{(\alpha,j)} z^\alpha \lambda^j\, .
 \]
 This describes the asymptotic behavior of the Bloch variety in the logarithmic direction $\eta$, and is studied in
 Section~\ref{S:initial_matrix}.
%We explain that. 
A vector $(\eta,a)\in\ZZ^{d+1}$ gives an action of $\CC^\times$ on $(\CC^\times)^d\times\CC$, 
\[
   t\in\CC^\times\,,\ (z,\lambda)\in(\CC^\times)^d\times\CC\
   \longmapsto\ t.(z,\lambda)\ \vcentcolon=\ (t^{\eta_1}z_1,t^{\eta_2}z_2,\dotsc,t^{\eta_d}z_d\,,\, t^a\lambda)\,.
\]
We have the following result about initial forms.
   
%%%%%%%%%%%%%%%%%%%%%%%%%%%%%%%%%%%%%%%%%%%%%%%%%%%%%%%%%%%%%%%%%%%%%%%%%%%%%%%%%%%%%%%%%%%%%%%%%%%% 
\begin{Lemma}\label{Lem:quasihomogeneous}   
  Let $(\eta,a)\in\ZZ^{d+1}$ and $r$ be the minimal value it takes on $\calN(\Phi)$.
  We have the following.
  \begin{enumerate}[$(i)$]
    \item $\ini_{(\eta,a)}\Phi(t.(z,\lambda)) = t^r\, \ini_{(\eta,a)}\Phi(z,\lambda)$.\vspace{1pt}
    \item $t^{-r}\Phi(t.(z,\lambda)) =  \ini_{(\eta,a)}\Phi(z,\lambda) + \mbox{higher order terms in $t$}$.
    \item ${\displaystyle  r\, \ini_{(\eta,a)}\Phi\ =\ \sum_{i=1}^d \eta_i z_i \frac{\partial \ini_{(\eta,a)} \Phi}{\partial z_i}\ +\
              a \lambda \frac{\partial \ini_{(\eta,a)} \Phi}{\partial \lambda} }$.
    \item For each $i=1,\dotsc,d$, $\ini_{(\eta,a)} z_i\frac{\partial }{\partial z_i}\Phi=z_i\frac{\partial }{\partial z_i}\ini_{(\eta,a)} \Phi$.
      %
      %  This holds, as we define the initial form to be the restriction to F.  This is subtle.
      %
  \end{enumerate}
\end{Lemma}  
%%%%%%%%%%%%%%%%%%%%%%%%%%%%%%%%%%%%%%%%%%%%%%%%%%%%%%%%%%%%%%%%%%%%%%%%%%%%%%%%%%%%%%%%%%%%%%%%%%%% 
\begin{proof}   
  Statements $(i)$ and $(ii)$ follow from the calculation
  \[
  (t.(z,\lambda))^{(\alpha,j)}\ =\ t^{\eta\cdot\alpha + aj} z^\alpha \lambda^j\,, 
  \]
  and $(iii)$ holds as $z_i\frac{\partial}{\partial z_i}z^\alpha = \alpha_i z^\alpha$ and
  $\lambda\frac{\partial}{\partial\lambda}\lambda^j = j \lambda^j$.
  This implies that monomials are eigenvectors for the operators
  $z_i\frac{\partial}{\partial z_i}$ and $\lambda\frac{\partial}{\partial\lambda}$, which implies $(iv)$.
\end{proof}  
%%%%%%%%%%%%%%%%%%%%%%%%%%%%%%%%%%%%%%%%%%%%%%%%%%%%%%%%%%%%%%%%%%%%%%%%%%%%%%%%%%%%%%%%%%%%%%%%%%%% 
   
As monomials are eigenvectors for the operators $z_i\frac{\partial}{\partial z_i}$, each polynomial in~\eqref{Eq:CPE}
has support a subset of $\calN(\Phi)$.
Thus, we may replace each polynomial $f$  in the critical point equations~\eqref{Eq:CPE} by its initial form $\ini_{(\eta,a)} f$,
giving its \demph{initial subsystem}, \defcolor{$\ini_{(\eta,a)}\Psi$}.
By Lemma~\ref{Lem:quasihomogeneous} $(iv)$, the initial subsystem $\ini_{(\eta,a)}\Psi$ is the system of critical point
equations for the initial form $\ini_{(\eta,a)}\Phi$ of $\Phi$.

We distinguish three types of faces of $\calN(\Phi)$.
The \demph{base} is the face exposed by the vector $(\bfzero,1)$ parallel to the last
coordinate axis.
The corresponding initial form of $\Phi$ is $\det H(z)$ and is obtained by setting $\lambda=0$.
A face $F$ is \demph{vertical} if it contains a vector parallel to $(\bfzero,1)$.
Any face containing a vertical face is vertical.
Vectors exposing a vertical face are horizontal, having the form $(\eta,0)$.
The fourth Newton polytope of Figure~\ref{F:moreNewtonPolytopes}
has four vertical facets, each of which is a triangle of normalized volume (area) two.
%%%%%    Observe that $\calN(\Phi)$ is vertical.
Lemma~\ref{Lem:quasihomogeneous} implies the following.

%%%%%%%%%%%%%%%%%%%%%%%%%%%%%%%%%%%%%%%%%%%%%%%%%%%%%%%%%%%%%%%%%%%%%%%%%%%%%%%%%%%%%%%%%%%%%%%%%%%%
\begin{Corollary}\label{Cor:VerticalSubSystems}
  Suppose that $(\eta,0)$ exposes a vertical face $F$ of $\calN(\Phi)$.
  Then the initial subsystem $\ini_{(\eta,0)}\Psi$ of the critical point equations~\eqref{Eq:CPE} satisfies
  $\codim(F)$-many independent linear equations given by the vectors that expose $F$.
\end{Corollary}
%%%%%%%%%%%%%%%%%%%%%%%%%%%%%%%%%%%%%%%%%%%%%%%%%%%%%%%%%%%%%%%%%%%%%%%%%%%%%%%%%%%%%%%%%%%%%%%%%%%%
\begin{proof}
  Vectors exposing $F$ have the form $(\zeta,0)$, so that Lemma~\ref{Lem:quasihomogeneous} $(iii)$ becomes
  \begin{equation}\label{Eq:vertical_form}
      r\cdot \ini_{(\zeta,0)}\Phi\ =\ \sum_{i=1}^d \zeta_i z_i \frac{\partial \ini_{(\zeta,0)}\Phi}{\partial z_i}\,, 
  \end{equation}
  a linear equation on $\ini_{(\zeta,0)}\Psi$.    
\end{proof}
%%%%%%%%%%%%%%%%%%%%%%%%%%%%%%%%%%%%%%%%%%%%%%%%%%%%%%%%%%%%%%%%%%%%%%%%%%%%%%%%%%%%%%%%%%%%%%%%%%%%

An \demph{oblique face} is one that is neither vertical nor the base.
We collect some results about the critical point degree and initial subsystems of $\Psi$.

%%%%%%%%%%%%%%%%%%%%%%%%%%%%%%%%%%%%%%%%%%%%%%%%%%%%%%%%%%%%%%%%%%%%%%%%%%%%%%%%%%%%%%%%%%%%%%%%%%%%
\begin{Proposition} \label{P:SolutionsAtInfinity}
  \mbox{\ }
  \begin{enumerate}[$(i)$]
   \item 
     The bound on the critical point degree of Proposition~\ref{P:CPD_Bound} is attained if and only if for every
     vector $\eta$ that is not parallel to $(\bfzero,1)$, the initial  subsystem $\ini_\eta\Psi$ of the critical point
     equations has no solutions. 

   \item 
     If $\eta$ exposes an oblique face, then any solution to the facial subsystem $\ini_\eta\Psi$ is a singular point of the initial
     Bloch variety $\Var(\ini_\eta\Phi)$. 

   \item 
     If $\eta$ exposes a vertical face, then  $\ini_\eta\Psi$ has a solution.
  \end{enumerate}
\end{Proposition}
%%%%%%%%%%%%%%%%%%%%%%%%%%%%%%%%%%%%%%%%%%%%%%%%%%%%%%%%%%%%%%%%%%%%%%%%%%%%%%%%%%%%%%%%%%%%%%%%%%%%

Statement $(i)$ is~\cite[Thm.\ 3.8]{FS}, while $(ii)$ is~\cite[Lem.\ 3.7]{FS}.
The third statement follows from Corollary~\ref{Cor:VerticalSubSystems} and Lemma~\ref{Lem:Vertical}.
While statement $(iii)$ was not stated in~\cite{FS}, its main results excluded Newton polytopes
with vertical faces.
Our main result, Theorem~\ref{Th:main}, is a refinement of Proposition~\ref{P:SolutionsAtInfinity}.
These results all involve studying the critical point equations on a compactification of
the Bloch variety in a projective toric variety.
The next two sections develop the algebraic combinatorics and algebraic geometry needed for this.

%%%%%%%%%%%%%%%%%%%%%%%%%%%%%%%%%%%%%%%%%%%%%%%%%%%%%%%%%%%%%%%%%%%%%%%%%%%%%%%%%%%%%%%%%%%%%%%%%%%%
%
\section{The Initial Graph of a periodic operator}\label{S:initGraph}

We establish a purely combinatorial result about initial forms of a dispersion polynomial. 
This enables the definition of an initial graph; a fundamental concept for Sections~\ref{S:vertical} and \ref{S:Structural}.
Part of this appeared in~\cite{FLiu}, and it is key to~\cite{FLSS}.

%%%%%%%%%%%%%%%%%%%%%%%%%%%%%%%%%%%%%%%%%%%%%%%%%%%%%%%%%%%%%%%%%%%%%%%%%%%%%%%%%%%%%%%%%%%%%%%%%%%%
\subsection{Cancellation-free determinants}\label{S:cancellation-free}

Let $M=(f_{i,j})_{i,j=1}^n$ be a matrix of polynomials $f_{i,j}(y)$ in some (set of) indeterminates $y$.
Recall the formula for the determinant of a matrix,
\[
   \det M\ =\ 
   \sum_{w\in S_n}  \sgn(w) f_{1,w(1)}(y)\, f_{2,w(2)}(y)\, \dotsb\, f_{n,w(n)}(y)\ =\
   \sum_{w\in S_n}  \sgn(w) M_w\ .
\]
This is the sum over all permutations $w\in S_n$, $\sgn(w)\in\{\pm 1\}$ is the sign of $w$, and \defcolor{$M_w$} is
the corresponding product of entries. 
We say that \defcolor{$\det M$} is \demph{cancellation-free} if, for all $w\in S_n$ such that $M_w$ is a non-zero
polynomial, its support is a subset of the support of $\det M$.

The \demph{generic Floquet matrix} \defcolor{$H(z,e,V)$} for a graph $\Gamma$ is the Floquet
matrix~\eqref{Eq:Floquet_Matrix} where the edge parameters $e(u,v)$ and the potentials $V(v)$ are indeterminates.
The \demph{generic dispersion polynomial} for $\Gamma$ is
\[
   \defcolor{\Phi(z,\lambda,e,V)}\ \vcentcolon=\ \det(\lambda  {\textit Id}_n\ -\ H(z,e,V))\,,
\]
where $H(z,e,V)$ is the generic Floquet matrix for $\Gamma$, and we identify $W$ with $[n]$.

%%%%%%%%%%%%%%%%%%%%%%%%%%%%%%%%%%%%%%%%%%%%%%%%%%%%%%%%%%%%%%%%%%%%%%%%%%%%%%%%%%%%%%%%%%%%%%%%%%%%
\begin{Proposition}\label{P:GFMcancellation-free}
  The generic dispersion polynomial is cancellation-free.
\end{Proposition}
%%%%%%%%%%%%%%%%%%%%%%%%%%%%%%%%%%%%%%%%%%%%%%%%%%%%%%%%%%%%%%%%%%%%%%%%%%%%%%%%%%%%%%%%%%%%%%%%%%%%

To minimize discussion of signs, replace all potentials $V(v)$ by their negatives, so that every entry in the generic
characteristic matrix $\defcolor{M(z,\lambda,e,V)}\vcentcolon= \lambda  {\textit Id}_n - H(z,e,V)$ is a positive
sum of monomials in $z,\lambda,e,V$ ({\it cf}.~\eqref{Eq:Floquet_Matrix}).
Then there is no cancellation in the expansion of any product of entries of $M(z,\lambda,e,V)$, nor in specializing $z$ to be
$\defcolor{\bOne}\vcentcolon=(1,\dotsc,1)\in\TT^d$.

We show that the determinant of a generic symmetric matrix is cancellation-free, use that to show that
$\det M(\bOne,\lambda,e,V)$ is cancellation-free, and finally deduce Proposition~\ref{P:GFMcancellation-free}.

Suppose that $M=(x_{i,j})_{i,j=1}^n$ is a generic symmetric matrix in that the $x_{i,j}$ are indeterminates and
$x_{i,j}=x_{j,i}$.
The expression for its determinant
\[
  \det M\ =\
  \sum_{w\in S_n}  \sgn(w) x_{1,w(1)} x_{2,w(2)} \dotsb x_{n,w(n)} =\
  \sum_{w\in S_n}  \sgn(w) M_w\ ,
\]
may have repeated monomial terms.
Collecting like terms gives a well-known, compact expression.
We derive it for completeness.

For $w\in S_n$, let \defcolor{$|w|$} be the graph on $\defcolor{[n]}=\{1,\dotsc,n\}$ with one edge $i\sim w(i)$
for each $i\in[n]$.
Every vertex of $|w|$ has degree two.
Fixed points of $w$ are loops and 2-cycles correspond to parallel edges.
Such a graph is a \demph{cycle cover of $[n]$}.
There are five cycle covers of $[3]$:
\[
 \begin{picture}(56,52)
   \put(0,0){\includegraphics{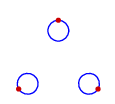}}
   \put(1,2){$3$} \put(25.2,44){$1$} \put(50,2){$2$}
 \end{picture}
 \qquad\quad
 \begin{picture}(56,52)
   \put(0,0){\includegraphics{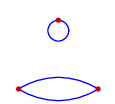}}
   \put(1,3){$3$} \put(25.2,44){$1$} \put(50,3){$2$}
 \end{picture}
 \qquad\quad
 \begin{picture}(56,52)
   \put(0,0){\includegraphics{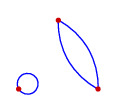}}
   \put(1,2){$3$} \put(25.2,44){$1$} \put(50,3){$2$}
 \end{picture}
 \qquad\quad
 \begin{picture}(56,52)
   \put(0,0){\includegraphics{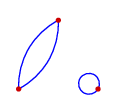}}
   \put(1,3){$3$} \put(25.2,44){$1$} \put(50,2){$2$}
 \end{picture}
 \qquad\quad
 \begin{picture}(56,52)
   \put(0,0){\includegraphics{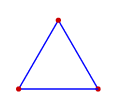}}
   \put(1,2){$3$} \put(25.2,44){$1$} \put(50,2){$2$}
 \end{picture}
\]
Note that the triangle occurs twice, once for each 3-cycle in $S_3$.

Let $c$ be a cycle cover of $[n]$ and let $m_i$ be its number of components of size $i$; this
is the number of $i$-cycles in $w$ when $|w|=c$.
Define \defcolor{$\sgn(c)$} to be $(-1)^m$, where $m$ is the number of components of even size in $c$ and
set $\defcolor{n_c}\vcentcolon = 2^{m_3+m_4+\dotsb}$.
We also set \defcolor{$M_c$} to be the product of $x_{i,j}$, where $i\sim j$ is an edge of $c$.

%%%%%%%%%%%%%%%%%%%%%%%%%%%%%%%%%%%%%%%%%%%%%%%%%%%%%%%%%%%%%%%%%%%%%%%%%%%%%%%%%%%%%%%%%%%%%%%%%%%%
\begin{Proposition}\label{P:Det_of_generic_symmetric_matrix}
  For any cycle cover $c$ of\/ $[n]$, 
  \begin{equation}\label{Eq:n_c} 
     n_c\ =\ \#\{w\in S_n \colon |w|=c\}\,.
  \end{equation}
  For $w\in S_n$, $\sgn(w)=\sgn(|w|)$.
  If $M$ is a symmetric matrix, then $M_w=M_{|w|}$, and 
  \[
  \det M\ =\ \sum_c \sgn(c)\, n_c\, M_c\,,
  \]
  the sum over all cycle covers of $[n]$.
\end{Proposition}
%%%%%%%%%%%%%%%%%%%%%%%%%%%%%%%%%%%%%%%%%%%%%%%%%%%%%%%%%%%%%%%%%%%%%%%%%%%%%%%%%%%%%%%%%%%%%%%%%%%%
\begin{proof}
  For a cycle $\zeta = (i_1,i_2,\dotsc,i_r)$, set
  $\defcolor{x_\zeta}\vcentcolon= x_{i_1,i_2} x_{i_2,i_3}\dotsb x_{i_r,i_1}$.
  As $x_{i,j}=x_{j,i}$, we have that $x_\zeta=x_{\zeta^{-1}}$.
  Every permutation $w\in S_n$ is uniquely a product of disjoint cycles
  \[
     w\ =\ \zeta_1\cdot \zeta_2 \dotsb \zeta_k\,.
  \]
  Note that $M_w=x_{\zeta_1} x_{\zeta_2}\dotsb x_{\zeta_k}$.
  If $u=\zeta_1^{\pm}\zeta_2^{\pm}\dotsb\zeta_k^{\pm}$ (any choice of $\pm$), then $|u|=|w|$, $\sgn(u)=\sgn(w)$, and $M_u=M_w$.
  Moreover, if  $M_u=M_w$, then $u$ has this form.

  Since $\zeta=\zeta^{-1}$ if and only if the cycle has length less than 3, we deduce~\eqref{Eq:n_c}.
  The remaining statements follow from the observations that $\sgn(w)$ and $M_w$ only depend upon $|w|$.     
\end{proof}
%%%%%%%%%%%%%%%%%%%%%%%%%%%%%%%%%%%%%%%%%%%%%%%%%%%%%%%%%%%%%%%%%%%%%%%%%%%%%%%%%%%%%%%%%%%%%%%%%%%%

Let $M(\bOne, \lambda,e,V)$ be the specialization of the generic characteristic matrix at the point
$z=\bOne$.
Note that $M$ is symmetric.
Specializing the entries of the generic symmetric matrix to those of $M(\bOne, \lambda,e,V)$, we have the
formula of Proposition~\ref{P:Det_of_generic_symmetric_matrix},
 \begin{equation}\label{Eq:formula}
   \det M(\bOne, \lambda,e,V)\ =\ \sum_c \sgn(c)\, n_c \, M_c\,,
 \end{equation}
the sum over all cycle covers of $[n]$.
(Here, $M_c=M_c(\bOne, \lambda,e,V)$.)

%%%%%%%%%%%%%%%%%%%%%%%%%%%%%%%%%%%%%%%%%%%%%%%%%%%%%%%%%%%%%%%%%%%%%%%%%%%%%%%%%%%%%%%%%%%%%%%%%%%%
\begin{Lemma}\label{L:monomialsDetermine}
  For any cycle cover $c$, if $M_c\neq 0$, then any monomial in $M_c$ determines $c$.
\end{Lemma}
%%%%%%%%%%%%%%%%%%%%%%%%%%%%%%%%%%%%%%%%%%%%%%%%%%%%%%%%%%%%%%%%%%%%%%%%%%%%%%%%%%%%%%%%%%%%%%%%%%%%
\begin{proof}
  Suppose that $w\in S_n$ has $|w|=c$, and the polynomial $M_c=M_w$ is nonzero.
  Writing $M(\bOne, \lambda,e,V)=(f_{i,j}(\lambda,e,V))_{i,j=1}^n$, we have
  \[
     M_w\ =\ f_{1,w(1)}(\lambda,e,V)\, f_{2,w(2)}(\lambda,e,V)\, \dotsb\, f_{n,w(n)}(\lambda,e,V)\, .
  \]
  By Remark~\ref{Rem:Sparse}, each term in $f_{i,j}(\lambda,e,V)$ (except $\lambda$ when $i=j$) determines the
  unordered pair $\{i,j\}$.
  Thus, each monomial in $M_w$ determines all unordered pairs $\{i,w(i)\}$ when $i\neq w(i)$.
  But this determines $c=|w|$.
\end{proof}
%%%%%%%%%%%%%%%%%%%%%%%%%%%%%%%%%%%%%%%%%%%%%%%%%%%%%%%%%%%%%%%%%%%%%%%%%%%%%%%%%%%%%%%%%%%%%%%%%%%%

%%%%%%%%%%%%%%%%%%%%%%%%%%%%%%%%%%%%%%%%%%%%%%%%%%%%%%%%%%%%%%%%%%%%%%%%%%%%%%%%%%%%%%%%%%%%%%%%%%%%
\begin{Corollary}\label{C:Cancellation-free}
  $\det M(\bOne,\lambda,e,V)$ is cancellation-free.
\end{Corollary}
%%%%%%%%%%%%%%%%%%%%%%%%%%%%%%%%%%%%%%%%%%%%%%%%%%%%%%%%%%%%%%%%%%%%%%%%%%%%%%%%%%%%%%%%%%%%%%%%%%%%
\begin{proof}
  Every entry of $M(\bOne,\lambda,e,V)$ has positive coefficients, so there are no cancellations in the expansion of any
  non-zero product $M_w$.
  The result follows from the formula~\eqref{Eq:formula} and Proposition~\ref{P:Det_of_generic_symmetric_matrix} as
  $M_w=M_{|w|}$.
\end{proof}
%%%%%%%%%%%%%%%%%%%%%%%%%%%%%%%%%%%%%%%%%%%%%%%%%%%%%%%%%%%%%%%%%%%%%%%%%%%%%%%%%%%%%%%%%%%%%%%%%%%%

%%%%%%%%%%%%%%%%%%%%%%%%%%%%%%%%%%%%%%%%%%%%%%%%%%%%%%%%%%%%%%%%%%%%%%%%%%%%%%%%%%%%%%%%%%%%%%%%%%%%
\begin{proof}[Proof of Proposition~{\protect \ref{P:GFMcancellation-free}}]
  Each term of $M(z,\lambda,e,V)_w$ has underlying monomial $\alpha(\lambda,e,V) z^{a_w}$ for some monomial
  $\alpha(\lambda,e,V)$
  of $M(\bOne,\lambda,e,V)_w$ and exponent $a_w\in\ZZ^d$.
  By Lemma~\ref{L:monomialsDetermine}, the monomial $\alpha(\lambda,e,V)$ only occurs in $M(\bOne,\lambda,e,V)_u$ for
  $|u|=|w|$.
  Thus, the terms in $\det M(z,\lambda,e,V)$ involving the monomial $\alpha(\lambda,e,V)$ have the form
  \begin{equation}\label{Eq:SameSign}
     \sgn(w) \, \alpha(\lambda,e,V) \sum_{|u|=|w|} z^{a_u}\,.\qedhere
  \end{equation}
\end{proof}  
%%%%%%%%%%%%%%%%%%%%%%%%%%%%%%%%%%%%%%%%%%%%%%%%%%%%%%%%%%%%%%%%%%%%%%%%%%%%%%%%%%%%%%%%%%%%%%%%%%%%

By~\eqref{Eq:SameSign} something stronger holds.
Any monomial in $M(z,\lambda,e,V)_w$ has the same sign in any other $M(z,\lambda,e,V)_u$, and these are the same sign as
that monomial in $\det M(z,\lambda,e,V)$.

%%%%%%%%%%%%%%%%%%%%%%%%%%%%%%%%%%%%%%%%%%%%%%%%%%%%%%%%%%%%%%%%%%%%%%%%%%%%%%%%%%%%%%%%%%%%%%%%%%%%
%
\subsection{Initial matrix}\label{S:initial_matrix}
We continue with the notation of Section~\ref{S:cancellation-free}.
The \demph{Newton polytope \defcolor{$\calN(\Gamma)$} of the graph} $\Gamma$ is the Newton polytope of the generic
dispersion polynomial $\Phi$, as a polynomial in $z,\lambda$.
This was defined in~\cite[Sect.~4.1]{FS}, where it was shown that the Newton polytope of the dispersion polynomial of any
operator on $\Gamma$ is a subset of $\calN(\Gamma)$.

Let $\eta\in\ZZ^{d+1}$ and \defcolor{$F$} be the face of $\calN(\Gamma)$ that it exposes.
For $w\in S_n$, let $M_w$ be the corresponding term in $\det M$.
Then \demph{$w$ contributes to} $\ini_\eta\Phi$ if $M_w$ has terms whose monomials lie
along the face $F$.
We define the \demph{initial matrix} $\defcolor{\ini_\eta M}= \ini_\eta M(z,\lambda,e,V)$ by its entries
 \begin{equation}\label{Eq:InitialMatrix}
  \bigl(\ini_\eta M\bigr)_{i,j}\ \vcentcolon=\  \left\{
     \begin{array}{rcl} \ini_\eta (M_{i,j}) & &\mbox{if $j=w(i)$ for some $w\in S_n$ contributing to $\ini_\eta\Phi$}\\
       0 &&\mbox{otherwise} \end{array}\right..
 \end{equation}
%
%%%%%%%%%%%%%%%%%%%%%%%%%%%%%%%%%%%%%%%%%%%%%%%%%%%%%%%%%%%%%%%%%%%%%%%%%%%%%%%%%%%%%%%%%%%%%%%%%%%%
\begin{Remark}\label{Rem:initialMatrix}
  Note that the initial matrix $\ini_\eta M$ is {\it not}  necessarily the matrix of
  $\eta$-initial forms of
  entries of $M$, but rather a submatrix  of it.
  This is illustrated in Example~\ref{Ex:hexPlus}.
\end{Remark}
%%%%%%%%%%%%%%%%%%%%%%%%%%%%%%%%%%%%%%%%%%%%%%%%%%%%%%%%%%%%%%%%%%%%%%%%%%%%%%%%%%%%%%%%%%%%%%%%%%%%

%%%%%%%%%%%%%%%%%%%%%%%%%%%%%%%%%%%%%%%%%%%%%%%%%%%%%%%%%%%%%%%%%%%%%%%%%%%%%%%%%%%%%%%%%%%%%%%%%%%%
\begin{Lemma}\label{Lem:initialMatrixOK}
  The initial matrix $\ini_\eta M$ only depends on the face of $\calN(\Gamma)$ exposed by $\eta$, and 
  \begin{equation}\label{Eq:initialDeterminant}
     \det \ini_\eta M(z,\lambda,e,V)\ =\ \ini_\eta \Phi\,.
  \end{equation}
\end{Lemma}
%%%%%%%%%%%%%%%%%%%%%%%%%%%%%%%%%%%%%%%%%%%%%%%%%%%%%%%%%%%%%%%%%%%%%%%%%%%%%%%%%%%%%%%%%%%%%%%%%%%%
\begin{proof}
  Let $\eta\in\ZZ^{d+1}$ expose the face $F$ of $\calN(\Phi)$ and let $\defcolor{N}\vcentcolon=\ini_\eta M(z,\lambda,e,V)$
  be the initial matrix.
  We first prove~\eqref{Eq:initialDeterminant} and then deduce its independence from $\eta$.

  For a nonzero polynomial $f$ in $z,\lambda$, let \defcolor{$\eta(f)$} be the minimum value that the linear form $\eta$
  takes on the exponents of $f$ and $+\infty$ if $f=0$.
  Then $\ini_\eta f$ is the sum of terms whose exponent vectors $\alpha$ minimize $\eta$, in that $\eta\cdot\alpha=\eta(f)$.
  Note how these behave under multiplication,
  $\eta(f\cdot g)= \eta(f)+\eta(g)$ and $\ini_\eta(f\cdot g) = \ini_\eta f \cdot \ini_\eta g $.
  Consequently, for $w\in S_n$,
  \[
  \ini_\eta(M_w)\ =\ \prod_{i=1}^n \ini_\eta( M_{i,w(i)})
  \qquad\mbox{and}\qquad
  \eta(M_w)\ =\ \sum_{i=1}^n \eta(M_{i,w(i)})\,.
  \]
  As $\det M=\Phi$ is cancellation-free and $\eta(\Phi)$ is the minimum of $\eta$ on $\calN(\Phi)$,
  and each entry $N_{i,j}$ of the initial matrix $N=\ini_\eta M$ is either zero or is equal to $\ini_\eta(M_{i,j})$, we have 
  \begin{equation}\label{Eq:cineq}
  \eta(\Phi)\ \leq\ \eta(M_w)\ \leq\ \eta(N_w)\,,
  \end{equation}
  with equality if and only if $w$ contributes to $\ini_\eta\Phi$.
  In that case, observe that
  \[
     \ini_\eta( M_{i,w(i)})\ =\ N_{i,w(i)}\ =\ \ini_\eta(N_{i,w(i)}) \qquad\mbox{for }i\in [n]\,.
  \]

  We prove~\eqref{Eq:initialDeterminant} by showing that  for $\tau\in S_n$ with $N_\tau\neq 0$, we have
  \[
     \eta(\Phi)\ =\ \sum_{i=1}^n \eta(N_{i,\tau(i)})\ =\ \eta(N_\tau)\,.
  \]

  Suppose that $\tau\in S_n$ and $N_\tau\neq 0$.
  In particular, for each $i\in[n]$, $0\neq N_{i,\tau(i)}=\ini_\eta(M_{i,\tau(i)})$.
  By the definition of $N=\ini_\eta M$, for each $i=1,\dotsc,n$, there is a permutation $w_i$ such that $w_i(i)=\tau(i)$
  and $w_i$ contributes to $\ini_\eta\Phi$.
  
  For each permutation $w\in S_n$, let \defcolor{$\Pi_w$} be the corresponding permutation matrix and set
  \[
     \Pi\ \vcentcolon=\ \sum_{i=1}^n \Pi_{w_i}\ -\ \Pi_\tau\,,
  \]
  which is a matrix with nonnegative integer entries, and whose row and column sums are each $n{-}1$.
  By Birkhoff's Theorem~\cite{B46} (a consequence of Hall's Theorem on matchings), there are permutations
  $\rho_1,\dotsc,\rho_{n-1}$ such that
  \begin{equation}\label{Eq:n-1}
      \Pi\ =\ \sum_{i=1}^{n-1}\Pi_{\rho_i}\,.
  \end{equation}

  Consider the additive map \defcolor{$\widehat{\eta}$} on the set of nonnegative matrices $X=(x_{i,j})_{i,j=1}^n$ defined by
  \[
  \widehat{\eta}\ \colon\ X\ \longmapsto\ \sum_{i,j} x_{i,j} \cdot \eta( N_{i,j} )\ \in\ \RR_{\geq 0}\cup\{\infty\}\,.
  \]
  Then $\widehat{\eta}(\Pi_w)=\eta(N_w)$ for any permutation $w\in S_n$.
  This, together with the definition of $\eta$, the inequality~\eqref{Eq:cineq}, and expressions for $\Pi$ give
  \begin{multline*}
          n\cdot \eta(\Phi)\
    \leq\ \eta( N_\tau)\ +\ \sum_{i=1}^{n-1}\eta(N_{\rho_i})\
     =\  \widehat{\eta}(\Pi_\tau)\ +\ \sum_{i=1}^{n-1}\widehat{\eta}(\Pi_{\rho_i})\ 
     =\  \widehat{\eta}(\Pi_\tau)\ +\ \widehat{\eta}\Bigl(\sum_{i=1}^{n-1}\Pi_{\rho_i}\Bigr)\\
     =\  \widehat{\eta}( \Pi_\tau + \Pi)\
     =\  \widehat{\eta}\Bigr( \sum_{i=1}^n \Pi_{w_i}\Bigl)\ 
    =\  \sum_{i=1}^n \widehat{\eta}(\Pi_{w_i})\ 
     =\  \sum_{i=1}^{n}\eta(N_{w_i})\ =\ n\cdot \eta(\Phi)\,.
  \end{multline*}
  This implies in particular that $\eta(\Phi)$ is equal to $\eta( N_\tau)$ and completes the proof
  of~\eqref{Eq:initialDeterminant}.
  By~\eqref{Eq:initialDeterminant}, $\det\ini_\eta M(z,\lambda,e,V)$ depends on the face $F$ of $\calN(\Gamma)$ exposed
  by $\eta$.
  By Remark~\ref{Rem:initialMatrix}, the initial matrix $\ini_\eta M(z,\lambda,e,V)$ also only
  depends on $F$.

  Finally, note that the initial matrix $\ini_\eta M$ depends only on the face exposed by $\eta$ and is otherwise independent
  of $\eta \in \ZZ^{d+1}$.
  This follows from~\eqref{Eq:InitialMatrix}, as the only way $\ini_\eta M$ and  $\ini_{\eta'} M$ could differ is if two nonzero
  entries differ, but this contradicts Corollary~\ref{C:Cancellation-free}.  
\end{proof}
%%%%%%%%%%%%%%%%%%%%%%%%%%%%%%%%%%%%%%%%%%%%%%%%%%%%%%%%%%%%%%%%%%%%%%%%%%%%%%%%%%%%%%%%%%%%%%%%%%%%

%%%%%%%%%%%%%%%%%%%%%%%%%%%%%%%%%%%%%%%%%%%%%%%%%%%%%%%%%%%%%%%%%%%%%%%%%%%%%%%%%%%%%%%%%%%%%%%%%%%%
%
\subsection{The initial graph}\label{SS:initial_graph}

The arguments of Sections~\ref{S:cancellation-free} and~\ref{S:initial_matrix}, while ostensibly algebraic, 
have a purely combinatorial interpretation in terms of the graph $\Gamma$.
Thus, $\ini_\eta\Phi$ and ultimately the correction terms $\Nvert$ and $\Ndisc$ of Theorem~\ref{Th:main} are purely combinatorial and reflect
structural properties of  $\Gamma$.

We  construct a labeled, directed multigraph  $\hGamma$ from $\Gamma$ whose vertices are $W$ and whose (weighted) adjacency matrix 
is $M(z,\lambda,e,V)$.
The dispersion polynomial $\Phi$ is a cancellation-free sum of monomials indexed by {\it directed} cycle covers of $\hGamma$.
For $\eta\in\ZZ^{d+1}$, we define a subgraph $\ini_\eta\hGamma$ of $\hGamma$ whose adjacency matrix is $\ini_\eta M$.
Structural properties of this initial graph  $\ini_\eta\hGamma$ determine the correction terms
$\Nvert$ and $\Ndisc$ of Theorem~\ref{Th:main}.

%%%%%%%%%%%%%%%%%%%%%%%%%%%%%%%%%%%%%%%%%%%%%%%%%%%%%%%%%%%%%%%%%%%%%%%%%%%%%%%%%%%%%%%%%%%%%%%%%%%%
\begin{Definition}\label{De:hatGamma}
  Given a $\ZZ^d$-periodic labeled graph $\Gamma$ with fundamental domain $W$, let \defcolor{$\hGamma$} be the graph
  with vertex set $W$, and labeled, directed edges as follows.
  \begin{enumerate}[$(i)$]
    
    \item For $v\in W$, $\hGamma$ has two loops at $v$, one labeled $\lambda$ and the other labeled $-V(v)$.
   
    \item For each $u,v\in W$ and $\alpha\in\ZZ^d$, if $v\sim \alpha{+}u$ is an edge of $\Gamma$, then there is a
      directed edge $v\leftarrow u$ in $\hGamma$ with label $e_{(v,\alpha{+}u)}z^\alpha$.

  \end{enumerate}
  Observe that $\hGamma$ has loops and parallel edges.
  Each vertex $v\in W$ has degree $4+2 d_v$ in $\hGamma$, where $d_v$ is the degree of $v$ in $\Gamma$.
\end{Definition}
%%%%%%%%%%%%%%%%%%%%%%%%%%%%%%%%%%%%%%%%%%%%%%%%%%%%%%%%%%%%%%%%%%%%%%%%%%%%%%%%%%%%%%%%%%%%%%%%%%%%
\begin{Example}\label{Ex:Hexagon_quotient}
  For the hexagonal lattice $\Gamma$ of Figure~\ref{F:localHexagon}, $\hGamma$ is the following graph.
\[
  \raisebox{-47pt}{\begin{picture}(276,103)(-138,-50)
   \put(-113,-50){\includegraphics{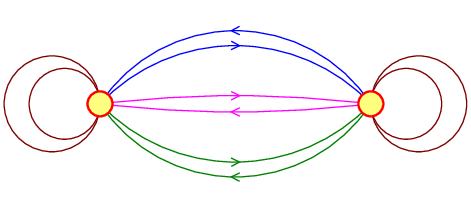}}
   \put( -68, -2.3){$u$}       \put( 63,  -2.3){$v$}
   \put( -95, -2.3){$\lambda$} \put( 88,  -2.3){$\lambda$}
   \put(-105,  28){$-V(u)$}    \put( 75,  28){$-V(v)$}
   \put( -10,  7){$a$}         \put(-10, -11){$a$}
   \put(   7, 39){$cy^{-1}$}    \put(  7,  19){$cy$}
   \put(   7,-46){$bx^{-1}$}    \put(  7, -23){$bx$}
  \end{picture}}
  \qedhere
\]
\end{Example}
%%%%%%%%%%%%%%%%%%%%%%%%%%%%%%%%%%%%%%%%%%%%%%%%%%%%%%%%%%%%%%%%%%%%%%%%%%%%%%%%%%%%%%%%%%%%%%%%%%%%
\begin{Remark}
  The graph $\hGamma$ is  independent of the choice of $W$.
  Replacing $W$ by the set of $\ZZ^d$-orbits of vertices in Definition~\ref{De:hatGamma} gives the same graph,
  {\it mutatis mutandis}.
  In fact, $\hGamma$ is the quotient by $\ZZ^d$ of a directed version of $\Gamma$ and $\Gamma$ may be recovered from
  $\hGamma$ by removing the loops labeled $\lambda$, reversing signs, and then taking the underlying undirected
  graph of an appropriate cover in the sense of Sunada~\cite{Sunada}.
  (We leave the details to the reader.)
\end{Remark}  
%%%%%%%%%%%%%%%%%%%%%%%%%%%%%%%%%%%%%%%%%%%%%%%%%%%%%%%%%%%%%%%%%%%%%%%%%%%%%%%%%%%%%%%%%%%%%%%%%%%%
\begin{Definition}
  The \demph{adjacency matrix} of a finite, directed, labeled multigraph $G$ with vertex set $\calV$ is the
  $\calV\times\calV$ matrix
  \defcolor{$\Ad_G$} whose entry in position $(v,u)\in\calV\times\calV$ is
  \[
      \bigl( \Ad_G\bigr)_{v,u}\ =\ \sum_{v\leftarrow u} e_{v\leftarrow u}\,,
  \]
   where $e_{v\leftarrow u}$ is the label of the edge $v\leftarrow u$ in $G$.
\end{Definition}
%%%%%%%%%%%%%%%%%%%%%%%%%%%%%%%%%%%%%%%%%%%%%%%%%%%%%%%%%%%%%%%%%%%%%%%%%%%%%%%%%%%%%%%%%%%%%%%%%%%%

The following observation is immediate from these  definitions.

%%%%%%%%%%%%%%%%%%%%%%%%%%%%%%%%%%%%%%%%%%%%%%%%%%%%%%%%%%%%%%%%%%%%%%%%%%%%%%%%%%%%%%%%%%%%%%%%%%%%
\begin{Proposition}
  Let $\Gamma$ be a $\ZZ^d$-periodic graph with Floquet matrix $H(z)$.
  Then
  \[
     \Ad_{\hGamma}\ =\ \lambda I_W\ -\ H(z)\,,
  \]
  the characteristic matrix of $\Gamma$.
\end{Proposition}
%%%%%%%%%%%%%%%%%%%%%%%%%%%%%%%%%%%%%%%%%%%%%%%%%%%%%%%%%%%%%%%%%%%%%%%%%%%%%%%%%%%%%%%%%%%%%%%%%%%%

Thus, each labeled edge of $\hGamma$ corresponds to a monomial in an entry of the characteristic matrix $M(z,\lambda,e,V)$.
A monomial term in the expansion of $\det(M(z,\lambda,e,V))$
corresponds to a permutation and
a choice of monomials, one from each of the appropriate entries of $M(z,\lambda,e,V)$.
This collection of edges corresponds to a (directed) cycle cover $\zeta$ of the graph $\hGamma$ and the monomial is the product
\defcolor{$\wt(\zeta)$} of the labels of the directed edges of $\zeta$.

%%%%%%%%%%%%%%%%%%%%%%%%%%%%%%%%%%%%%%%%%%%%%%%%%%%%%%%%%%%%%%%%%%%%%%%%%%%%%%%%%%%%%%%%%%%%%%%%%%%%
\begin{Example}
  The graph $\hGamma$ of Example~\ref{Ex:Hexagon_quotient} has 13 cycle covers.
  We display four of them, together with their monomials.
  \[
  \begin{picture}(107,67)(3,-8)
    \put(0,0){\includegraphics{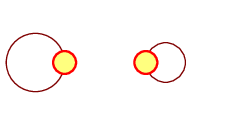}}
    \put(27.8,27.5){$u$}    \put(67.5,27.5){$v$}
    \put(3,47){$-V(u)$}   \put(77,42){$\lambda$}
    \put(28,-8){$-\lambda V(u)$}
  \end{picture}\ \ 
  \begin{picture}(90,67)(10,-8)
    \put(0,0){\includegraphics{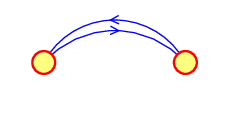}}
    \put(17.8,27.5){$u$}    \put(86,27.5){$v$}
    \put(36,35){$cy$}  \put(66,51){$cy^{-1}$}
    \put(51,-8){$c^2$}
  \end{picture}\ \ 
  \begin{picture}(110,67)(0,-8)
    \put(0,0){\includegraphics{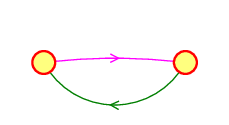}}
    \put(17.8,27.5){$u$}    \put(86,27.5){$v$}
    \put(60,34){$a$}  \put(39,15){$bx^{-1}$}
    \put(43,-8){$abx^{-1}$}
  \end{picture}\ \ 
  \begin{picture}(90,67)(10,-8)
    \put(0,0){\includegraphics{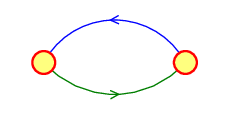}}
    \put(17.8,27.5){$u$}    \put(86,27.5){$v$}
    \put(37,20){$bx$}  \put(66,51){$cy^{-1}$}
    \put(42,-8){$bcxy^{-1}$}
  \end{picture}
  \]
  These 13 cycle covers, with their monomials, correspond to the 13 terms
  in~\eqref{Eq:DispersionPOlynomiaHexagonalLattice}.
\end{Example}
%%%%%%%%%%%%%%%%%%%%%%%%%%%%%%%%%%%%%%%%%%%%%%%%%%%%%%%%%%%%%%%%%%%%%%%%%%%%%%%%%%%%%%%%%%%%%%%%%%%%

The following observation is immediate from these  definitions.

%%%%%%%%%%%%%%%%%%%%%%%%%%%%%%%%%%%%%%%%%%%%%%%%%%%%%%%%%%%%%%%%%%%%%%%%%%%%%%%%%%%%%%%%%%%%%%%%%%%%
\begin{Proposition}\label{P:fine_Determinant}
  Let $\Gamma$ be a $\ZZ^d$-periodic graph with Floquet matrix $H(z)$.
  Then we have
  \[
     \Phi\ =\ \det(\lambda I_W-H(z))\ =\ \det M(z,\lambda,e,V)\ =\ \sum \wt(\zeta)\ ,
  \]
  the sum over all cycle covers $\zeta$ of\/ $\hGamma$.  
\end{Proposition}
%%%%%%%%%%%%%%%%%%%%%%%%%%%%%%%%%%%%%%%%%%%%%%%%%%%%%%%%%%%%%%%%%%%%%%%%%%%%%%%%%%%%%%%%%%%%%%%%%%%%

Let $\eta\in\ZZ^{d+1}$.
For a cycle cover $\zeta$ of $\hGamma$, write \defcolor{$\eta(\zeta)$} for the integer $\eta(\wt(\zeta))$
and let $\eta(\hGamma)$ be the minimum of $\eta(\zeta)$ over all cycle covers $\zeta$ of $\hGamma$.

%%%%%%%%%%%%%%%%%%%%%%%%%%%%%%%%%%%%%%%%%%%%%%%%%%%%%%%%%%%%%%%%%%%%%%%%%%%%%%%%%%%%%%%%%%%%%%%%%%%%
\begin{Definition}\label{De:initial_Graph}
  For $\eta\in\ZZ^{d+1}$, the \demph{initial graph} \defcolor{$\ini_\eta\hGamma$} is a directed subgraph of $\hGamma$ with vertex set $W$.
  A labeled edge $v\leftarrow u$ of $\hGamma$ is a labeled edge of $\ini_\eta\hGamma$ if and only if there is a cycle cover
  $\zeta$ of $\hGamma$ that contains the edge $v\leftarrow u$ and we have $\eta(\zeta)=\eta(\hGamma)$.
  Every edge of $\zeta$ lies in $\ini_\eta\hGamma$, so that $\zeta$ is a cycle cover of $\ini_\eta\hGamma$ and every edge of
  $\ini_\eta\hGamma$ occurs in some cycle cover.
\end{Definition}
%%%%%%%%%%%%%%%%%%%%%%%%%%%%%%%%%%%%%%%%%%%%%%%%%%%%%%%%%%%%%%%%%%%%%%%%%%%%%%%%%%%%%%%%%%%%%%%%%%%%

A consequence of these definitions, Lemma~\ref{Lem:initialMatrixOK}, and Proposition~\ref{P:fine_Determinant} is the
following.

%%%%%%%%%%%%%%%%%%%%%%%%%%%%%%%%%%%%%%%%%%%%%%%%%%%%%%%%%%%%%%%%%%%%%%%%%%%%%%%%%%%%%%%%%%%%%%%%%%%%
\begin{Theorem}\label{Th:InitialGraphForm}
  Let $\Gamma$ be a $\ZZ^d$-periodic graph with characteristic matrix $M(z,\lambda,e,V)$ and dispersion polynomial $\Phi$.
  For any $\eta\in\ZZ^{d+1}$, we have
  \[
    \ini_\eta M\ =\ \Ad_{\ini_\eta\hGamma}
    \qquad\mbox{and}\qquad
      \ini_\eta \Phi\ =\ \sum \wt(\zeta)\,,
  \]
  the sum over all cycle covers  $\zeta$  of the initial graph $\ini_\eta\hGamma$.      
\end{Theorem}  
%%%%%%%%%%%%%%%%%%%%%%%%%%%%%%%%%%%%%%%%%%%%%%%%%%%%%%%%%%%%%%%%%%%%%%%%%%%%%%%%%%%%%%%%%%%%%%%%%%%% 

Vertical faces of $\calN(\Phi)$ are detected by initial graphs.

%%%%%%%%%%%%%%%%%%%%%%%%%%%%%%%%%%%%%%%%%%%%%%%%%%%%%%%%%%%%%%%%%%%%%%%%%%%%%%%%%%%%%%%%%%%%%%%%%%%%
\begin{Theorem}\label{L:vertical}
  A vector $\eta$ exposes a vertical face of  $\calN(\Phi)$ if and only if
  the initial graph $\ini_{\eta}\hGamma$ has a vertex $v$  with loops labeled $\lambda$ and $-V(v)$.
  Such a vector $\eta$ is horizontal.
\end{Theorem}
%%%%%%%%%%%%%%%%%%%%%%%%%%%%%%%%%%%%%%%%%%%%%%%%%%%%%%%%%%%%%%%%%%%%%%%%%%%%%%%%%%%%%%%%%%%%%%%%%%%%
\begin{proof}
  For the forward implication, note that  the only monomials in the matrix $M$ whose exponents do not lie 
  in $\ZZ^d\times\{0\}$ are the indeterminates $\lambda$ of each diagonal entry.
  As $\eta$ exposes a vertical face, it is horizontal.
  Consequently, if the face exposed by $\eta$ is vertical, $\ini_{\eta}\hGamma$ has a loop
  at some vertex \defcolor{$v$} labeled $\lambda$.
  Let $\zeta$ be a cycle cover of $\ini_{\eta}\hGamma$ containing this loop and let $\zeta'$ be the union of the other cycles in $\zeta$.
  Let $\zeta''$ be $\zeta'$ together with the loop at $v$ labeled $-V(v)$; This is a cycle cover of  $\hGamma$.
  Since $\eta$ is horizontal, $\eta(\zeta)=\eta(\zeta')=\eta(\zeta'')$, which implies that the loop at $v$ labeled $-V(v)$ also
  lies in $\ini_{\eta}\hGamma$.

  Now suppose that $v$ is a vertex of $\ini_{\eta}\hGamma$ with two loops labeled $\lambda$ and $-V(v)$.
  The $\eta$ is necessarily horizontal.
  By the same arguments as in the previous paragraph, then there are cycle covers $\zeta$
  and $\zeta''$ of $\ini_{\eta}\hGamma$ that differ only in these two loops at $v$.
  Writing $\zeta'$ for the (common) rest of $\zeta, \zeta''$, we have
  \[
  \wt(\zeta)+\wt(\zeta'')\ =\ (\lambda -V(v))\wt(\zeta')\,.
  \]
  By the first factor, the face of $\calN(\Phi)$ exposed by $\eta$ contains the vertical vector $(\bfzero,1)$.
\end{proof}  
%%%%%%%%%%%%%%%%%%%%%%%%%%%%%%%%%%%%%%%%%%%%%%%%%%%%%%%%%%%%%%%%%%%%%%%%%%%%%%%%%%%%%%%%%%%%%%%%%%%%

%%%%%%%%%%%%%%%%%%%%%%%%%%%%%%%%%%%%%%%%%%%%%%%%%%%%%%%%%%%%%%%%%%%%%%%%%%%%%%%%%%%%%%%%%%%%%%%%%%%%
\begin{Example}\label{Ex:meet_singularHouse}
  The periodic graph $\Gamma$ on the left in Figure~\ref{F:Vertical} has four
%%%%%%%%%%%%%%%%%%%%%%%%%%%%%%%%%%%%%%%%%%%%%%%%%%%%%%%%%%%%%%%%%%%%%%%%%%%%%%%%%%%%%%%%%%%%%%%%%%%%
\begin{figure}[htb]
  \centering
   \includegraphics[height=110pt]{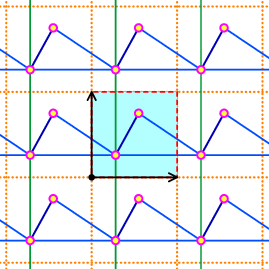}
    \qquad
   \begin{picture}(157,110)(-26,0)
    \put(  0,  0){\includegraphics{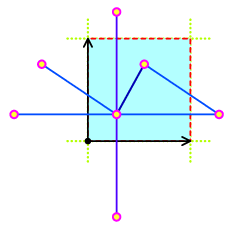}}
    \put( 36, 30){{\color{white}\circle*{9}}}
    \put( 47, 47){\small$u$}      \put( 67, 83){\small$v$}  

    \put(-31, 44){\small$(-1,0){+}u$}  \put(98, 63){\small$(1,0){+}u$}
    \put(61,102){\small$(0,1){+}u$} \put( 61,3){\small$(0,-1){+}u$}
    \put(-17, 86){\small$(-1,0){+}v$}

    \put(65, 61){\small$a$}
    \put(28, 45){\small$b$}   \put(75, 45){\small$b$}
    \put(49, 75){\small$c$}   \put(49, 24){\small$c$}
    \put(31, 73){\small$d$}   \put(82, 72){\small$d$}   
   \end{picture}
   \quad
   \begin{picture}(100,120)
     \put(0,0){\includegraphics[height=110pt]{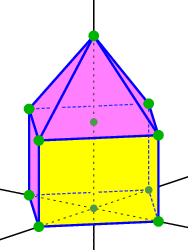}}
     \put( 55,38){\small$G$}     \put( 60,77){\small$A$}
     \put(5,0){$x$} \put(70,0){$y$} \put(47,100){$\lambda$}
   \end{picture}\vspace{5pt}
  \caption{A $\ZZ^2$-periodic graph, its edge labels, and Newton polytope.}
   \label{F:Vertical}
\end{figure}
%%%%%%%%%%%%%%%%%%%%%%%%%%%%%%%%%%%%%%%%%%%%%%%%%%%%%%%%%%%%%%%%%%%%%%%%%%%%%%%%%%%%%%%%%%%%%%%%%%%%
(orbits of) edges and two vertices in its fundamental domain.
Let $u,v$ be the vertices and $a,b,c,d$ the edge labels as indicated
in the middle diagram in Figure~\ref{F:Vertical}.

The characteristic matrix $M=\lambda I_2-H(x,y)$  is
\[
  M\ \vcentcolon=\ 
  \left(
  \begin{array}{cc}
       \lambda - V(u) +\underline{bx}+bx^{-1}+\underline{cy}+cy^{-1} & \underline{a}+dx^{-1} \\
         a + \underline{dx}                                        & \underline{\lambda -V(v)}   
  \end{array}
   \right).
\]
The dispersion polynomial $\Phi(x,y,\lambda)=\det(\lambda I - H)$ is
 \begin{multline} \label{Eq:singHouseDispersonPolynomial}
  \lambda^2
   +\lambda( bx +cy - V(u) - V(v) +bx^{-1} + cy^{-1}) \\
   -bV(v)x -adx - cV(v)y  -a^2-d^2 + V(u)V(v) -(bV(v)+ad)x^{-1}  - cV(v)y^{-1}\,.\qquad
 \end{multline}
Its Newton polytope $A$ is shown on the right in Figure~\ref{F:Vertical}, and here is the graph $\hGamma$.
\[
\begin{picture}(273,110)(-147,-54)
   \put(-125,-55){\includegraphics{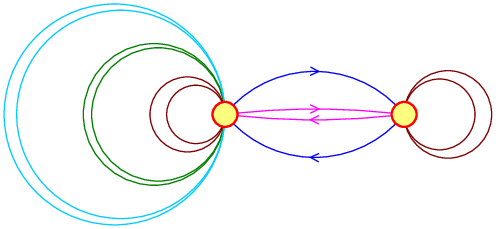}}
   \put( -19, -2.6){$u$}       \put(66.5,-2.6){$v$}
   \put( -42, -2.6){$\lambda$} \put(  92,-2.6){$\lambda$}
   \put( -19,  43){$-V(u)$}
   \thicklines
   \put( -11,39  ){{\color{White}\line(-1,-1){20}}}   \put( -11,39.5){{\color{White}\line(-1,-1){20}}}
   \put( -11,40  ){{\color{White}\line(-1,-1){20}}}
   \put( -11,40.5){{\color{White}\line(-1,-1){20}}}   \put( -11,41  ){{\color{White}\line(-1,-1){20}}}
   \put( -11,40  ){{\color{White}\vector(-1,-1){19}}}
   \thinlines
   \put( -11,40){{\color{Maroon}\vector(-1,-1){20.5}}}
   \thinlines
                               \put(  88,  24){$-V(v)$}
   \put(  31,  6){$a$}         \put(  31, -10){$a$}
   \put(-135, -14){$bx$}       \put(-112, -14){$bx^{-1}$}
   \put( -95,  13){$cy$}       \put( -76,   8){$cy^{-1}$}
   \put(  41,-28){$dx^{-1}$}    \put(  41,  20.5){$dx$}
\end{picture}
\]
The face $G$ of $A$ is exposed by $\eta=(-1,-1,0)$.
We show $\ini_\eta M$ and the initial graph $\ini_\eta\hGamma$.
 \begin{equation}\label{Eq:inGgraph}
  \left(
  \begin{array}{cc}
       bx+cy & a \\
         dx  & \lambda -V(v)   
  \end{array}\right)\qquad\quad
  \raisebox{-27pt}{\begin{picture}(220,60)(-12,0)
   \put(0,0){\includegraphics{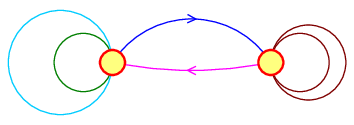}}
   \put(50.5,  27){$u$}       \put(127, 27){$v$}
   \put( -9,  28){$bx$}      \put( 13,  28){$cy$}  
   \put(102,  18){$a$}       \put(102,  52){$dx$}
   \put(149,  27){$\lambda$} \put(150,  49){$-V(v)$}
  \end{picture}} 
 \end{equation}
Then $\ini_\eta\Phi = (\lambda-V(v))(bx+cy) -adx= \Ad_{\ini_\eta\hGamma}$,
whose support is the face $G$ of the polytope in Figure~\ref{F:Vertical}.
\end{Example}
%%%%%%%%%%%%%%%%%%%%%%%%%%%%%%%%%%%%%%%%%%%%%%%%%%%%%%%%%%%%%%%%%%%%%%%%%%%%%%%%%%%%%%%%%%%%%%%%%%%% 

A connected component $G$ of $\ini_\eta\hGamma$ is \demph{monomial} if every cycle cover of $G$ gives the same monomial in $(z,\lambda)$.
Continuing Example~\ref{Ex:meet_singularHouse}, if $\eta=(-1,0,1)$, then $\ini_\eta A=(1,0,0)$ is the vertex marked `$x$'. 
We show both $\ini_\eta M$ and $\ini_\eta\hGamma$, which is monomial.
\[
  \left(
  \begin{array}{cc}
       bx & a \\
         dx  & V(v)   
  \end{array}\right)\qquad\quad
  \raisebox{-22pt}{\begin{picture}(155,51)(-9,0)
   \put(0,0){\includegraphics{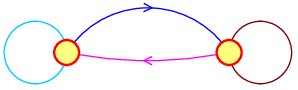}}
   \put(28.5,15.5){$u$}       \put(107,15.5){$v$}
   \put( -9,  23){$bx$}      
   \put( 80,   7){$a$}       \put(80,  41){$dx$}
   \put(110,  37){$-V(v)$}
  \end{picture}} 
\]

%%%%%%%%%%%%%%%%%%%%%%%%%%%%%%%%%%%%%%%%%%%%%%%%%%%%%%%%%%%%%%%%%%%%%%%%%%%%%%%%%%%%%%%%%%%%%%%%%%%%
\begin{Definition}\label{Def:asymptoticallyDisconnected}
  If $\ini_\eta\hGamma$ has at least two connected components that are not monomial, then  $\Gamma$
  is \demph{asymptotically disconnected} in the direction $\eta$.
  As with $\ini_\eta M$, this depends upon the face of $\calN(\Phi)$ exposed by $\eta$.
\end{Definition}  
%%%%%%%%%%%%%%%%%%%%%%%%%%%%%%%%%%%%%%%%%%%%%%%%%%%%%%%%%%%%%%%%%%%%%%%%%%%%%%%%%%%%%%%%%%%%%%%%%%%%

We state an easy consequence of Definition~\ref{Def:asymptoticallyDisconnected}

%%%%%%%%%%%%%%%%%%%%%%%%%%%%%%%%%%%%%%%%%%%%%%%%%%%%%%%%%%%%%%%%%%%%%%%%%%%%%%%%%%%%%%%%%%%%%%%%%%%%
\begin{Lemma}\label{Lem:asymptoticallydisconnected}
  Suppose that $\Gamma$ is asymptotically disconnected in the direction of $\eta$.
  Then $\ini_\eta\Phi$ has a factorization of the form $\gamma g_1\dotsb g_r$, where $\gamma$ arises from the monomial components
  of $\ini_\eta\Gamma$ and each $g_i$ from a non-monomial component.
\end{Lemma}
%%%%%%%%%%%%%%%%%%%%%%%%%%%%%%%%%%%%%%%%%%%%%%%%%%%%%%%%%%%%%%%%%%%%%%%%%%%%%%%%%%%%%%%%%%%%%%%%%%%%

In Section~\ref{S:Structural}, we explore the consequences of asymptotic disconnectedness and determine its
contribution to the correction term $\Ndisc$ of Theorem~\ref{Th:main}.

%%%%%%%%%%%%%%%%%%%%%%%%%%%%%%%%%%%%%%%%%%%%%%%%%%%%%%%%%%%%%%%%%%%%%%%%%%%%%%%%%%%%%%%%%%%%%%%%%%%%
%
\section{Toric varieties} \label{S:toric}
%
%%%%%%%%%%%%%%%%%%%%%%%%%%%%%%%%%%%%%%%%%%%%%%%%%%%%%%%%%%%%%%%%%%%%%%%%%%%%%%%%%%%%%%%%%%%%%%%%%%%%
The asymptotic behavior of an algebraic hypersurface $\Var(\Phi)$ is reflected in the geometry of the Newton polytope
$\calN(\Phi)$ and encapsulated by a compactification of $\Var(\Phi)$ in an associated projective toric variety.
For critical points in Bloch varieties, the volume of the Newton polytope is an upper bound with correction terms arising
from asymptotic critical points.
We formulate and prove our main result, Theorem~\ref{Th:main}, based on results of
Section~\ref{S:vertical} and~\ref{S:Structural}.

We give local descriptions of the Bloch variety and critical point equations needed for those results. 
This includes a discussion of singular solutions in Theorem~\ref{Th:multiplicity_at_infinity}.
A treatment of projective toric varieties is in~\cite[Ch.~5]{GKZ} with further results found
in~\cite[Ch.\ 2]{CLS},~\cite{Fu93},~\cite[Ch.~3]{IHP},
  and~\cite{Ibadan}.
For projective space,~\cite[Ch.\ 8]{CLO} suffices.

%%%%%%%%%%%%%%%%%%%%%%%%%%%%%%%%%%%%%%%%%%%%%%%%%%%%%%%%%%%%%%%%%%%%%%%%%%%%%%%%%%%%%%%%%%%%%%%%%%%%
\subsection{Toric varieties and critical points}
To simplify notation, the indeterminates in this section will be $t_1,\dotsc,t_n$.
Let $\calA\subset\ZZ^n$ be a finite set of integer vectors, which are exponent vectors for monomials in $t$.
We will use $\calA$ as an index set.
For example, \defcolor{$\CC^\calA$} is the set of functions from $\calA$ to $\CC$, a finite-dimensional $\CC$-vector space.
It has coordinates $(x_a\mid a\in\calA)$.
Write \defcolor{$\PP^\calA$} for the corresponding projective space which has homogeneous coordinates $[x_a\mid a\in\calA]$.
This equals $(\CC^\calA\smallsetminus\{0\})/\Delta\CC^\times$, the quotient of $\CC^\calA\smallsetminus\{0\}$ by the
action of scalar multiplication given by the scalar matrices \defcolor{$\Delta\CC^\times$}.
It is a compact complex manifold of dimension $|\calA|{-}1$.
The set of points $x\in\PP^\calA$ with no coordinate vanishing forms its \demph{dense torus},
$(\CC^\times)^\calA/\Delta\CC^\times \simeq (\CC^\times)^{|\calA|-1}$.

Let us consider the map $\varphi_\calA\colon (\CC^\times)^n\to\PP^\calA$ given by 
 \begin{equation}\label{Eq:varphi_A}
   \varphi_\calA\ \colon\ t\ \longmapsto\
          [ t^a  \mid  a \in\calA]\,.
  \end{equation}
We record a few facts about $\varphi_\calA$ that follow from these definitions.
Let $\defcolor{\ZZ\calA}\subset\ZZ^{n}$ be the free abelian subgroup generated by the differences
$\{a-b\mid a,b\in\calA\}$.
Its rank equals the dimension of the affine span of $\calA$.
Its saturation \defcolor{$\Sat(\calA)$} is $\defcolor{\RR\calA}\cap\ZZ^n$, the integer vectors in the $\RR$-span of $\ZZ\calA$,
which is a free abelian group of the same rank as $\ZZ\calA$.

%%%%%%%%%%%%%%%%%%%%%%%%%%%%%%%%%%%%%%%%%%%%%%%%%%%%%%%%%%%%%%%%%%%%%%%%%%%%%%%%%%%%%%%%%%%%%%%%%%%%
\begin{Proposition} \label{P:varphi_calA_facts}
 The map $\varphi_\calA$ is a group homomorphism from $(\CC^\times)^n$ to the dense torus
 in $\PP^\calA$.
 Its image \defcolor{$\calO_\calA$} is a torus of dimension equal to the rank of $\ZZ\calA$.
 Its kernel is
 \[
   \ker(\calA)\ =\ \{ t\in(\CC^\times)^n\mid t^\alpha = 1\ \mbox{\ for all\ }\ \alpha\in\ZZ\calA\}\,.
 \]
 We have the exact sequence
 \[
   1\ \longrightarrow\ \ker(\calA)_0\
   \longrightarrow\  \ker(\calA)\
   \longrightarrow\  \ker(\calA)_{\mbox{\footnotesize\rm tor}}\
   \longrightarrow\  1\,,
 \]
 where $\ker(\calA)_0$ is the connected component of the identity in $\ker(\calA)$,
  \[
     \defcolor{\ker(\calA)_0}\ \vcentcolon=\ \{t\in(\CC^\times)^n\mid t^\beta = 1\ \mbox{\ for all\ }\ \beta\in\Sat(\calA)\}\,,
  \]
  which is isomorphic to $(\CC^\times)^{n-\rank(\ZZ\calA)}$.
  The component group $\ker(\calA)_{\mbox{\footnotesize\rm tor}}$ is
  \begin{eqnarray*}
    \defcolor{\ker(\calA)_{\mbox{\footnotesize\rm tor}}} &\vcentcolon=&
      \{ z\in(\CC^\times)^n/\ker(\calA)_0 \mid z^\alpha = 1\ \mbox{\ for all\ }\ \alpha\in\ZZ\calA\}\\
      &\simeq& \Hom(\Sat(\calA)/\ZZ\calA\,,\ \CC^\times)\,,
  \end{eqnarray*}
  which is a finite group of order $[\Sat(\calA)\colon\ZZ\calA]$.
  The induced map  $\varphi_\calA\colon (\CC^\times)^n/\ker(\calA)_0\to\calO_\calA$ is a covering space of degree
  $[\Sat(\calA)\colon\ZZ\calA]$.
\end{Proposition}
%%%%%%%%%%%%%%%%%%%%%%%%%%%%%%%%%%%%%%%%%%%%%%%%%%%%%%%%%%%%%%%%%%%%%%%%%%%%%%%%%%%%%%%%%%%%%%%%%%%%

Let us write $\defcolor{X_\calA}$ for the closure of the image of $\varphi_\calA$.
This is the \demph{projective toric variety} associated to $\calA$.
A reason for these definitions is that they interchange nonlinearities of
equations in the following sense.
A linear combination of monomials, 
\[
   f\ =\ \sum_{a \in \calA}  c_a t^a \qquad\ c_a\in\CC\,,
\]
is a polynomial with \demph{support}  $\calA$.
Define $\defcolor{\Lambda_f}\vcentcolon= \Var(\sum c_a x_a)$, the hyperplane in $\PP^\calA$ whose defining
linear form has coefficients from $f$.
Then $\Var(f)=\varphi_\calA^{-1}(X_\calA\cap\Lambda_f)$.
The image under $\varphi_\calA$ of the nonlinear hypersurface $\Var(f)$ is a hyperplane section of
$\varphi_\calA((\CC^\times)^n)$.

Let us specialize these definitions to the objects in this paper.
Let  $\Phi$ be a dispersion polynomial for an operator on a periodic graph, let $\defcolor{\calA}$ be its support, and
let $\defcolor{A}\vcentcolon=\calN(\Phi)=\conv(\calA)$ be its Newton polytope. 
As $\ZZ^d$ acts cocompactly on $\Gamma$,  $\Sat(\calA)=\ZZ^{d+1}$ and $\ZZ\calA$ has full rank $d{+}1$.
(In fact they are equal in all examples that we have considered.)

We will write \defcolor{$t$} for the indeterminates $(z,\lambda)$ on $(\CC^\times)^d\times\CC$.
As $\lambda$ only appears with a positive exponent, the map $\varphi_\calA$~\eqref{Eq:varphi_A} extends to 
$(\CC^\times)^d\times\CC$, 
 \begin{equation}\label{Eq:newVarphi_A}
   \varphi_\calA\ \colon\ (\CC^\times)^d\times\CC\ni t\ \longmapsto\
          [ t^a  \mid  a \in\calA]\in\PP^\calA\,.
 \end{equation}
We obtain the projective toric variety $X_\calA$ as the closure of the image of $\varphi_\calA$.
Write $\defcolor{X^\circ_\calA}$ for the image $\varphi_\calA((\CC^\times)^d\times\CC)$ and
$\defcolor{\partial X_\calA}\vcentcolon= X_\calA\smallsetminus X^\circ_\calA$ for its complement.
Then
 \begin{equation}\label{Eq:Boundary_Decomp}
   X_\calA\ =\ X_\calA^\circ\; \coprod\; \partial X_\calA\,.
 \end{equation}

The \demph{compactified Bloch variety} is $\defcolor{\overline{\BV}}\vcentcolon= X_\calA\cap\Lambda_\Phi$, where
$\Lambda_\Phi$ is the hyperplane corresponding to the dispersion polynomial $\Phi$.
It is the projective closure of $\varphi_\calA(\BV)$.
We have
\[
   \varphi_\calA(\BV\,)\ =\ \varphi_\calA((\CC^\times)^d\times\CC)) \cap \Lambda_\Phi \ =\ \overline{\BV}\cap X^\circ_\calA\ .
\]
Each polynomial $f\in\Psi$ in the critical point equations~\eqref{Eq:CPE} has support a subset of $\calA$,
and therefore corresponds to a hyperplane $\Lambda_f$.
Let \defcolor{$L_\Psi$} be the intersection of these hyperplanes.
Then the critical points of the Bloch variety are $\varphi^{-1}(X^\circ_\calA\cap L_\Psi)$.
We have the inequality,
 \begin{equation}\label{Eq:ineq}
    \# \bigl( X^\circ_\calA\cap L_\Psi \bigr)\ \leq\    \# \bigl( X_\calA \cap L_\Psi\bigr)\,,
 \end{equation}
 where we count isolated points with their multiplicities.

%%%%%%%%%%%%%%%%%%%%%%%%%%%%%%%%%%%%%%%%%%%%%%%%%%%%%%%%%%%%%%%%%%%%%%%%%%%%%%%%%%%%%%%%%%%%%%%%%%%%
\begin{Remark}\label{Rem:intMult}
We recall some basics about these intersection multiplicities.
There are many sources; for example, this may be extracted from Sections 1.6 and 2.1 of~\cite{Fu96b} with full
justifications found in the books~\cite{EH3264,Fu84a}.

Let $\defcolor{Y}\subset\PP^n$ be a projective variety, and let $\defcolor{L}\subset\PP^n$ be a linear subspace of dimension
$n{-}\dim(Y)$, the \demph{codimension} of $Y$.
We expect that $Y\cap L$ is zero-dimensional, in which case we call $Y\cap L$ a \demph{linear section} of $Y$
and say that the intersection is \demph{proper}.
Bertini's Theorem asserts that there is a Zariski-open subset in the Grassmannian of linear subspaces of $\PP^n$ of
dimension $n-\dim(Y)$ such that the intersection $Y\cap L$ is transverse (and thus also zero-dimensional).
Moreover, the number of points in such a transverse linear section is independent of $L$ and is called the \demph{degree} of $Y$,
\defcolor{$\deg(Y)$}.
This independence from $L$ extends to all linear sections, if we count points with appropriate multiplicities.

Let $y\in Y\cap L$.
The local \demph{intersection multiplicity} \defcolor{$m(y;Y,L)$} is the dimension (as a complex vector space) of the local
ring of the intersection $Y\cap L$ supported at the point $y$.
This may be computed in any affine neighborhood of $y$ using local equations for $Y$ and $L$.
It has the following dynamic definition.
If $L(s)$ for $s\in [-\epsilon,\epsilon]$ is a continuous family of linear spaces with $\lim_{s\to 0}L(s)=L$ such that for $s\neq 0$, 
$Y\cap L(s)$ is transverse, then $m(y;Y,L)$ is the number of points in $Y\cap L(s)$ that converge to $y$ as $s\to 0$.
\end{Remark}
%%%%%%%%%%%%%%%%%%%%%%%%%%%%%%%%%%%%%%%%%%%%%%%%%%%%%%%%%%%%%%%%%%%%%%%%%%%%%%%%%%%%%%%%%%%%%%%%%%%%

We record the fundamental global property of these local multiplicities.

%%%%%%%%%%%%%%%%%%%%%%%%%%%%%%%%%%%%%%%%%%%%%%%%%%%%%%%%%%%%%%%%%%%%%%%%%%%%%%%%%%%%%%%%%%%%%%%%%%%%
\begin{Proposition}\label{P:Degree_Section}
  Suppose that $Y\cap L$ is a linear section of a projective variety $Y$.
  Then
  \begin{equation}\label{Eq:degree_multiplicities}
      \deg(Y)\ =\ \sum_{y\in Y\cap L} m(y;Y,L)\,.
  \end{equation}
\end{Proposition}
%%%%%%%%%%%%%%%%%%%%%%%%%%%%%%%%%%%%%%%%%%%%%%%%%%%%%%%%%%%%%%%%%%%%%%%%%%%%%%%%%%%%%%%%%%%%%%%%%%%%

We use this to refine Proposition~\ref{P:SolutionsAtInfinity}, and deduce a step towards Theorem~\ref{Th:main}.

%%%%%%%%%%%%%%%%%%%%%%%%%%%%%%%%%%%%%%%%%%%%%%%%%%%%%%%%%%%%%%%%%%%%%%%%%%%%%%%%%%%%%%%%%%%%%%%%%%%%
\begin{Corollary}\label{Cor:CPDeg}
  Let $\Gamma$ be a periodic graph, $\calA$ the support of its generic dispersion polynomial, and $A=\conv(\calA)$ its
  Newton polytope.
  Suppose that $\Phi(z,\lambda)$ is the dispersion polynomial of a general operator on $\Gamma$ with critical point
  equations $\Psi$.
  If the intersection $X_\calA\cap L_\Psi$ of the toric variety $X_\calA$ with the linear space $L_\Psi$ of\/ $\PP^\calA$
  corresponding to $\Psi$
  is a linear section, then
  \begin{equation}\label{Eq:CPDeg}
    \cpdeg(\Gamma)\ =\ \nvol(A)\ -\
    [\ZZ^{d+1}\colon \ZZ\calA]\sum_{x\in\partial X_\calA\cap L_\Psi} m(x; X_\calA,L_\Psi)\ .
  \end{equation}
\end{Corollary}  
%%%%%%%%%%%%%%%%%%%%%%%%%%%%%%%%%%%%%%%%%%%%%%%%%%%%%%%%%%%%%%%%%%%%%%%%%%%%%%%%%%%%%%%%%%%%%%%%%%%%
\begin{proof}
  The decomposition $X_\calA= X^\circ_{\calA} \coprod \partial X_\calA$~\eqref{Eq:Boundary_Decomp}
  and Proposition~\ref{P:Degree_Section} imply that
  \[
  \deg X_\calA\ =\ \sum_{x\in  X^\circ_\calA\cap L_\Psi} m(x; X_\calA,L_\Psi)\ +\
      \sum_{x\in\partial X_\calA\cap L_\Psi} m(x; X_\calA,L_\Psi)\ .
  \]
  By Kushnirenko's Theorem~\cite{Kushnirenko}, we have $\nvol(A)=[\ZZ^{d+1}\colon \ZZ\calA]\cdot\deg(X_\calA)$,
  and by Proposition~\ref{P:varphi_calA_facts}, the map $\varphi_\calA\colon(\CC^\times)^d\times\CC\to X^\circ_\calA$
  has degree $[\ZZ^{d+1}\colon \ZZ\calA]$, so that
  \[
    [\ZZ^{d+1}\colon \ZZ\calA]\sum_{x\in X^\circ_\calA\cap L_\Psi} m(x; X_\calA,L_\Psi)
  \]
  counts the critical points of the Bloch variety of $\Phi(z,\lambda)$.
  As $\Phi(z,\lambda)$ is general, this is the critical point degree of $\Gamma$.
  This implies the formula~\eqref{Eq:CPDeg}.
\end{proof}
%%%%%%%%%%%%%%%%%%%%%%%%%%%%%%%%%%%%%%%%%%%%%%%%%%%%%%%%%%%%%%%%%%%%%%%%%%%%%%%%%%%%%%%%%%%%%%%%%%%%

%%%%%%%%%%%%%%%%%%%%%%%%%%%%%%%%%%%%%%%%%%%%%%%%%%%%%%%%%%%%%%%%%%%%%%%%%%%%%%%%%%%%%%%%%%%%%%%%%%%%
\begin{Remark}\label{Rem:foreshadowing}
  We study the structure of the toric variety $X_\calA$ in the remainder of this section, including
  points in a linear section  $\partial X_\calA\cap L_\Psi$ and local multiplicities $m(x;X_\calA,L_\Psi)$.
  The remaining two sections are devoted to understanding two types of  contributions to the sum in~\eqref{Eq:CPDeg}.
  In Section~\ref{S:vertical}, Corollary~\ref{C:verticalContribution} gives a contribution $\Nvert$ from vertical faces of the
  Newton polytope $A$, and in Section~\ref{S:Structural},  we identify a contribution $\Ndisc$
  coming from oblique faces whose initial  graph is disconnected.
  By Theorem~\ref{L:vertical} and Lemma~\ref{Lem:asymptoticallydisconnected} the contributions arise from structural
  properties of the graph $\Gamma$.  
  By Corollaries~\ref{C:verticalContribution} and~\ref{C:obliqueContribution} we have the inequality,
  \[
     \sum_{x\in\partial X_\calA\cap L_\Psi} m(x; X_\calA,L_\Psi)\ \geq\ \Nvert\ +\ \Ndisc\,.  \qedhere
  \]
\end{Remark}
%%%%%%%%%%%%%%%%%%%%%%%%%%%%%%%%%%%%%%%%%%%%%%%%%%%%%%%%%%%%%%%%%%%%%%%%%%%%%%%%%%%%%%%%%%%%%%%%%%%%

We deduce our main theorem from this.

%%%%%%%%%%%%%%%%%%%%%%%%%%%%%%%%%%%%%%%%%%%%%%%%%%%%%%%%%%%%%%%%%%%%%%%%%%%%%%%%%%%%%%%%%%%%%%%%%%%%
\begin{Theorem}\label{Th:main}
  Let $\Gamma$ be a $\ZZ^d$-periodic graph with $d\leq 3$.
  Its critical point degree satisfies the inequality,
  \[
    \cpdeg(\Gamma)\ \leq\ \nvol(A)\ -\ [\ZZ^{d+1}\colon \ZZ\calA] (\Nvert\ +\ \Ndisc)\,.
  \]
\end{Theorem}
%%%%%%%%%%%%%%%%%%%%%%%%%%%%%%%%%%%%%%%%%%%%%%%%%%%%%%%%%%%%%%%%%%%%%%%%%%%%%%%%%%%%%%%%%%%%%%%%%%%%
The version stated in the Introduction is when $\ZZ\calA=\ZZ^{d+1}$.
%%%%%%%%%%%%%%%%%%%%%%%%%%%%%%%%%%%%%%%%%%%%%%%%%%%%%%%%%%%%%%%%%%%%%%%%%%%%%%%%%%%%%%%%%%%%%%%%%%%%
\begin{proof}
  This is a consequence of Corollary~\ref{Cor:CPDeg} and the inequality of Remark~\ref{Rem:foreshadowing}.
\end{proof}

%%%%%%%%%%%%%%%%%%%%%%%%%%%%%%%%%%%%%%%%%%%%%%%%%%%%%%%%%%%%%%%%%%%%%%%%%%%%%%%%%%%%%%%%%%%%%%%%%%%%
%
%
%
\subsection{Structure of $X_\calA$}\label{SS:Facial}
The inequality~\eqref{Eq:ineq} is strict when $\partial X_\calA \cap L_\Psi \neq \emptyset$.
Points in $\partial X_\calA \cap L_\Psi$ are \demph{asymptotic critical points}.
Proposition~\ref{P:SolutionsAtInfinity} concerns the existence of asymptotic critical points.
Our results involve counting them with their multiplicities.

To study the asymptotic critical points, we will need local coordinates for $X_\calA$ near them, expressions
for the critical point equations in these local coordinates, and a determination of the multiplicities of points in
$\partial X_\calA$.

The torus $(\CC^\times)^{d+1}$ acts on $\PP^\calA$ through the map $\varphi_\calA$~\eqref{Eq:newVarphi_A}.
As $X_\calA$ is the closure of an orbit, $(\CC^\times)^{d+1}$ acts on $X_\calA$.
As it is a toric variety, $X_\calA$ is a disjoint union of finitely many orbits of $(\CC^\times)^{d+1}$.
There is one orbit for each face $F$ of the polytope $A$.
Each orbit has an affine open neighborhood in $X_\calA$ which is itself a toric variety, and which has a canonical
ring of polynomial functions---its coordinate ring.
Each such coordinate ring is an easily-described subalgebra of the ring $\CC[t^\pm]$ of Laurent polynomials in $t$.

The face $A$ corresponds to the orbit $\varphi_\calA((\CC^\times)^{d+1})$, and 
$\varphi_\calA((\CC^\times)^d\times\{0\})$ is the orbit for the base  ({\it i.e.}\ $\lambda=0$).
These two orbits form the image $\varphi_\calA((\CC^\times)^d\times\CC)$ of $\varphi_\calA$, which is $X^\circ_\calA$.
The boundary $\partial X_\calA$ is the union of the remaining orbits.

Let \defcolor{$F$} be a proper face of $A$ that is not its base.
Write $\defcolor{\calF}\vcentcolon= F\cap\calA$ for the points of $\calA$ that lie along $F$.
As $A=\conv(\calA)$, we have  $F=\conv(\calF)$, the convex hull of $\calF$.
We define two subsets of the projective space $\PP^\calA$ corresponding to $F$.
Let $\defcolor{\PP^\calF}\subset\PP^\calA$ be the subspace spanned by the coordinates indexed by
$\calF$ so that $\PP^\calF=\{x\in\PP^\calA\mid x_a=0\mbox{ for }a\not\in \calF\}$, and set
$\defcolor{U_\calF}\vcentcolon=\{x\in\PP^\calA\mid x_b\neq 0\mbox{ for }b\in \calF\}$.
Then $\PP^\calF$ is a projective subspace of $\PP^\calA$ and $U_\calF$ is an affine open subset of $\PP^\calA$.

The set $\calF$ gives a map analogous to $\varphi_\calA$,
 \begin{equation}\label{Eq:calF}
   \varphi_\calF\ \colon\ (\CC^\times)^{d+1}\ \ni\ t\ \longmapsto\
           [ t^a \mid a\in \calF]\ \in\ \PP^\calF\ \subset\ \PP^\calA\,.
 \end{equation}
Write \defcolor{$\calO_\calF$} for its image $\varphi_\calF((\CC^\times)^{d+1})$, which is the orbit corresponding to the
face $\calF$. 
Its closure is the toric variety $\defcolor{X_\calF}$, which is a subvariety of $X_\calA$.
We summarize some consequences of these definitions and Proposition~\ref{P:varphi_calA_facts}, which describe the
large-scale structure of $X_\calA$ as a variety with an action by $(\CC^\times)^{d+1}$.

%%%%%%%%%%%%%%%%%%%%%%%%%%%%%%%%%%%%%%%%%%%%%%%%%%%%%%%%%%%%%%%%%%%%%%%%%%%%%%%%%%%%%%%%%%%%%%%%%%%%
\begin{DefProp}\label{DP:toricSets}
  Let $F$ be a non-base proper face of $A$.
  Define $\defcolor{V_\calF}\vcentcolon= U_\calF\cap X_\calA$.
  Then  $V_\calF$ is an affine toric variety and an open neighborhood in $X_\calA$ of the orbit $\calO_\calF$, and 
  \begin{enumerate}[$(i)$]
   \item $\PP^\calF\cap X_\calA = X_\calF$, $\calO_\calF=X_\calF\cap V_\calF$, and  $\calO_\calF\simeq(\CC^\times)^{\dim(F)}$.
   \item
     The sets $\calO_\calF$, $X_\calF$, and $V_\calF$ are all $(\CC^\times)^{d+1}$-stable subsets of $X_\calA$ with
     $\calO_\calF$ an orbit. 
   \item
     $X_\calA$ is the disjoint union of $X^\circ_\calA$ and the orbits $\calO_\calF$ for $F$ a non-base proper face of $A$.
 \end{enumerate}
\end{DefProp}
%%%%%%%%%%%%%%%%%%%%%%%%%%%%%%%%%%%%%%%%%%%%%%%%%%%%%%%%%%%%%%%%%%%%%%%%%%%%%%%%%%%%%%%%%%%%%%%%%%%%

By $(iii)$, each asymptotic critical point lies on a unique orbit $\calO_\calF$, and may be studied in the neighborhood
$V_\calF$ of $\calO_\calF$. 
We describe the coordinate ring of $V_\calF$ as a subalgebra of $\CC[t^\pm]$, the coordinate ring of $(\CC^\times)^{d+1}$.
This will include the coordinate ring of $\calO_\calF$, and we explain what the critical point equations~\eqref{Eq:CPE}
become in both $V_\calF$ and $\calO_\calF$.

The monomials in $\CC[t^\pm]$ are identified with their exponent vectors---elements of $\ZZ^{d+1}$.
Subalgebras $R$ of $\CC[t^\pm]$ generated by monomials correspond to \demph{affine monoids}---subsets $M$ of $\ZZ^{d+1}$ that contain
$0$ and are closed under addition.
Here, $M=\{a\in\ZZ^{d+1}\mid t^a\in R\}$ and $R\simeq\CC[M]$ is the monoid algebra.
We describe the monoids for the coordinate rings of $V_\calF$ and $\calO_\calF$ of
Definition-Proposition~\ref{DP:toricSets}. 
Define
  \begin{equation}\label{Eq:M_calF}
   \defcolor{M_\calF}\ \vcentcolon=\ \NN\{a-b\mid a\in \calA,\ b\in \calF\}\,,
 \end{equation}
the affine monoid generated by the differences $a-b$ of integer vectors for $a\in\calA$ and $b\in\calF$.
Let $\defcolor{\ZZ\calF}\vcentcolon= M_\calF \cap -M_\calF$, which is the maximal subgroup of $M_\calF$.
It is also the integer span of differences of elements of $\calF$,
\[
    \ZZ\calF\ =\ \NN\{a-b\mid a, b\in \calF\}\ =\ \ZZ\{a-b\mid a, b\in \calF\}\,.
\]

%%%%%%%%%%%%%%%%%%%%%%%%%%%%%%%%%%%%%%%%%%%%%%%%%%%%%%%%%%%%%%%%%%%%%%%%%%%%%%%%%%%%%%%%%%%%%%%%%%%%
\begin{Proposition}\label{P:coordinateRings}
  The coordinate ring of $V_\calF$ is the monoid algebra $\CC[M_\calF]$ and the coordinate ring of 
  $\calO_\calF$ is $\CC[\ZZ\calF]$.
\end{Proposition}
%%%%%%%%%%%%%%%%%%%%%%%%%%%%%%%%%%%%%%%%%%%%%%%%%%%%%%%%%%%%%%%%%%%%%%%%%%%%%%%%%%%%%%%%%%%%%%%%%%%%
\begin{proof}
  On the affine open set $U_\calF$ no coordinate $x_b$ for $b\in\calF$ vanishes, thus
  $\{ x_a/x_b\mid a\in\calA \mbox{ and } b\in\calF\}$ are regular functions on $U_\calF$, and they generate its
  coordinate ring.
  Restricting them to $V_\calF$ gives $\{ t^{a-b}\mid a\in\calA \mbox{ and } b\in\calF\}$, which generates $M_\calF$,
  showing that $\CC[M_\calF]$ is the coordinate ring of $V_\calF$.

  Similarly, on the dense torus of $\PP^\calF$,  $\{ x_a/x_b\mid a,b\in\calF\}$ are regular functions that generate its
  coordinate ring.
  Their restrictions $\{ t^{a-b}\mid a,b\in\calF\}$ to $\calO_\calF$ generate $\ZZ\calF$, showing that $\CC[\ZZ\calF]$ is
  the coordinate ring of $\calO_\calF$.
\end{proof}
%%%%%%%%%%%%%%%%%%%%%%%%%%%%%%%%%%%%%%%%%%%%%%%%%%%%%%%%%%%%%%%%%%%%%%%%%%%%%%%%%%%%%%%%%%%%%%%%%%%%

%%%%%%%%%%%%%%%%%%%%%%%%%%%%%%%%%%%%%%%%%%%%%%%%%%%%%%%%%%%%%%%%%%%%%%%%%%%%%%%%%%%%%%%%%%%%%%%%%%%%
\begin{Example}\label{Ex:singularHouse}
In Figure~\ref{F:singularHouse}, we label some  faces of the polytope $A$ of Figure~\ref{F:Vertical}.
%%%%%%%%%%%%%%%%%%%%%%%%%%%%%%%%%%%%%%%%%%%%%%%%%%%%%%%%%%%%%%%%%%%%%%%%%%%%%%%%%%%%%%%%%%%%%%%%%%%%
\begin{figure}[htb]
\centering
  \begin{picture}(165,124)(-35,-4)
    \put(  0,0){\includegraphics{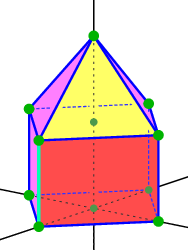}}
    \put(-35,50){\small$(1,0,1)$}
    \put(  4,-3){\small$(1,0,0)$}  \put(70, 1){\small$(0,1,0)$} 
    \put( 80,52){\small$(0,1,1)$} \put(49,101){\small$(0,0,2)$}
    \put( 75,90){$A$}
    \thicklines
      \put(0,54){{\color{White}\vector(1,0){14.8}}}
      \put(0,53.5){{\color{White}\vector(1,0){15.3}}}
      \put(0,53){{\color{White}\vector(1,0){15.8}}}
      \put(0,52.5){{\color{White}\vector(1,0){15.3}}}
      \put(0,52){{\color{White}\vector(1,0){14.8}}}
      \put(0,37){{\color{White}\vector(1,0){17}}}
      \put(0,36.5){{\color{White}\vector(1,0){17.5}}}
      \put(0,36){{\color{White}\vector(1,0){17.8}}}
      \put(0,35.5){{\color{White}\vector(1,0){17.5}}}
      \put(0,35){{\color{White}\vector(1,0){17}}}
    \thinlines
    \put( 1,53){\vector(1,0){14}}  
    \put(-20,33){\small$E$}\put(-9,36){\vector(1,0){26}}
    \put( 55,38){\small$G$}  \put( 48,73){\small$H$}
  \end{picture}
  \caption{A polytope $A$ with three indicated faces.}\label{F:singularHouse}
\end{figure}
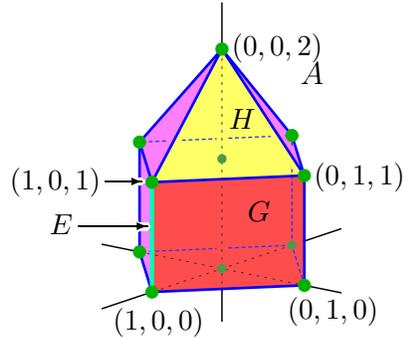
%%%%%%%%%%%%%%%%%%%%%%%%%%%%%%%%%%%%%%%%%%%%%%%%%%%%%%%%%%%%%%%%%%%%%%%%%%%%%%%%%%%%%%%%%%%%%%%%%%%%
Let $E$ be the indicated face (the vertical edge containing $\calE\vcentcolon=\{(1,0,0), (1,0,1)\}$).
Translating $A$ by each endpoint of $E$ moves $E$ to a line segment along the vertical axis containing the
origin, showing that $\ZZ\calE$ is the set of integer points along the vertical axis.
We also have that 
\[
    M_\calE\ =\ \{ (x,y,\lambda)\in\ZZ^3 \mid -x-y\geq 0\mbox{ and } -x+y\geq 0\}\,.
\]
Note that $M_\calE = \ZZ\calE \oplus \tau_\calE$, where
  \begin{equation}\label{Eq:tau_F}
     \defcolor{\tau_\calE}\ \vcentcolon=\ \{(x,y,0)\in\ZZ^3\mid-x-y\geq 0\mbox{ and } -x+y\geq 0\}\,.
  \end{equation}

Now consider $G$, the facet exposed by $(-1,-1,0)$ and set $\calG\vcentcolon= G\cap\calA$.
Then
\[
   M_\calG\ =\ \{ (x,y,\lambda) \mid -x-y\geq 0 \}
   \qquad\mbox{and}\qquad
   \ZZ\calG\ =\ \{(x,y,\lambda)\mid x+y=0\}\ \simeq\ \ZZ^2\,,
\]
and we have $M_\calG = \ZZ\calG\oplus \NN\cdot(-1,0,0) \simeq \ZZ^2\oplus \NN$.

Similarly, the facet $H$ is exposed by $(-1,-1,-1)$. 
Setting $\calH\vcentcolon=H\cap\calA$, we have 
\[
   M_\calH\ =\ \{ (x,y,\lambda) \mid -x-y-\lambda\geq 0 \}
   \quad\mbox{and}\quad
   \ZZ\calH\ =\ \{(x,y,\lambda)\mid x+y+\lambda=0\}\ \simeq\ \ZZ^2\,,
\]
and we have $M_\calH = \ZZ\calH\oplus \NN\cdot(0,0,-1) \simeq \ZZ^2\oplus \NN$.
\end{Example}
%%%%%%%%%%%%%%%%%%%%%%%%%%%%%%%%%%%%%%%%%%%%%%%%%%%%%%%%%%%%%%%%%%%%%%%%%%%%%%%%%%%%%%%%%%%%%%%%%%%%

In all three cases considered in Example~\ref{Ex:singularHouse}, we have that
\[
    M_\calF\ =\ \ZZ\calF \oplus \tau_\calF\,.
\]
A direct sum decomposition does not hold in general~\cite[\S~2.2]{CLS}.

An \demph{ideal} of a monoid $M$ is a subset  $I\subset M$ such that $M+I\subset I$.
It is \demph{prime} when its complement $I^c\vcentcolon= M\smallsetminus I$ is a submonoid.
If $I\subset M$ is prime, then
$\defcolor{\langle I\rangle}\vcentcolon=\langle t^a\mid a\in I\rangle$ is a prime ideal of $\CC[M]$ with $\CC[M]/\langle I\rangle$
canonically isomorphic to the monoid algebra $\CC[I^c]$.

%%%%%%%%%%%%%%%%%%%%%%%%%%%%%%%%%%%%%%%%%%%%%%%%%%%%%%%%%%%%%%%%%%%%%%%%%%%%%%%%%%%%%%%%%%%%%%%%%%%%
\begin{Lemma}\label{L:CoordinateRings}
  For any face $F$ of $P$, we have canonical maps of monoid algebras,
  \begin{equation}\label{Eq:MonoidAlgebraMap}
    \CC[\ZZ\calF]\   \stackrel{\iota}{\hooklongrightarrow}\
    \CC[M_\calF]\ \stackrel{p}{\longtwoheadrightarrow}\ \CC[\ZZ\calF]
  \end{equation}
  which induce canonical maps of affine toric varieties,
  \[
    \calO_\calF\   \stackrel{p^*}{\hooklongrightarrow}\
    V_\calF\ \stackrel{\iota^*}{\longtwoheadrightarrow}\ \calO_\calF\,.
  \]  
\end{Lemma}
%%%%%%%%%%%%%%%%%%%%%%%%%%%%%%%%%%%%%%%%%%%%%%%%%%%%%%%%%%%%%%%%%%%%%%%%%%%%%%%%%%%%%%%%%%%%%%%%%%%%
\begin{proof}
  Let $I\vcentcolon=M_\calF\smallsetminus \ZZ\calF$.
  To see that $I$ is an ideal, suppose that there are $a\in I$ and $b\in M_\calF$ such that
  $a+b\not\in I$, so that $a+b\in\ZZ\calF$.
  As $\ZZ\calF$ is a group, $-(a+b)\in\ZZ\calF$, and thus $-a=-(a+b)+b$ of $\ZZ^{d+1}$ lies in $M_\calF$.
  But then  $a\in M_\calF\cap -M_\calF=\ZZ\calF$, a contradiction.
  As $M_\calF\smallsetminus I=\ZZ\calF$ is a submonoid of $M_\calF$, $I$ is prime.
  The map $\iota$ of~\eqref{Eq:MonoidAlgebraMap} is induced by the inclusion $\ZZ\calF\subset M_\calF$ and the map $p$ is the quotient
  map by the ideal $\langle I\rangle$.
\end{proof}
%%%%%%%%%%%%%%%%%%%%%%%%%%%%%%%%%%%%%%%%%%%%%%%%%%%%%%%%%%%%%%%%%%%%%%%%%%%%%%%%%%%%%%%%%%%%%%%%%%%%

Let $f$ be a polynomial with support $\calA$ and let $F$ be a face of $A$ exposed by $\eta$.
Recall that the initial form $\ini_\eta f$ of $f$ is the sum of its terms whose exponent vectors lie in $F$.

%%%%%%%%%%%%%%%%%%%%%%%%%%%%%%%%%%%%%%%%%%%%%%%%%%%%%%%%%%%%%%%%%%%%%%%%%%%%%%%%%%%%%%%%%%%%%%%%%%%%
\begin{Lemma}\label{L:polysInSubrings}
  Let $f$ be a polynomial with support $\calA$ and $B=\Var(f)\subset(\CC^\times)^{d+1}$ be the subvariety of the
  torus it defines.
  Write $\overline{B}$ for the closure of $\varphi(B)$ in the projective toric variety $X_\calA$.
  Let $F$ be  face of $A$ exposed by $\eta$ and $b\in\calF$ a monomial on $F$, we have
  \begin{enumerate}[$(i)$]
    \item $t^{-b}f\in\CC[M_\calF]$ and $\Var(t^{-b}f)\subset V_\calF$ equals $\overline{B}\cap V_\calF$.
    \item $t^{-b}\ini_\eta f=p(t^{-b}f)\in\CC[\ZZ\calF]$ and
      $\Var(t^{-b}\ini_\eta f)\subset\calO_\calF$ equals $\overline{B}\cap\calO_\calF$, where
      $p\colon \CC[M_\calF]\twoheadrightarrow\CC[\ZZ\calF]$ is the map of~\eqref{Eq:MonoidAlgebraMap}.
  \end{enumerate}
\end{Lemma}
%%%%%%%%%%%%%%%%%%%%%%%%%%%%%%%%%%%%%%%%%%%%%%%%%%%%%%%%%%%%%%%%%%%%%%%%%%%%%%%%%%%%%%%%%%%%%%%%%%%%
\begin{proof}
  The principal ideal $\langle f\rangle$ generated by $f$ in $\CC[\ZZ^{d+1}]$ defines $B=\Var(f)$ in $(\CC^\times)^{d+1}$.
  Its closure $\overline{B}\cap V_\calF$ in $V_\calF$ is defined by $\langle f\rangle\cap \CC[M_\calF]$.

  Let $\alpha t^c$ with $\alpha\neq 0$ and $c\in\calF$ be a term of $f$.
  Then $t^{-c}f$ has a constant term $\alpha$ and $t^{-c}f\in\CC[M_\calF]$ by the definition of $M_\calF$, as $f$ has support $\calA$
  and $c\in\calF$.
  If $a\in M_\calF$, $t^a\in\CC[M_\calF]$ so that $t^at^{-c}f\in\CC[M_\calF]$.
  If $t^at^{-c}f\in\CC[M_\calF]$ for some
  $a\in \ZZ^{d{+}1}$,
%
%  This might need to be replaced by the bit 2 lines below.
%
%  $a\in \ZZ\calA$,
  then $\alpha t^a$ is a term of $t^at^{-c}f$ so that $a\in M_\calF$.
  This shows that $\langle f\rangle\cap \CC[M_\calF]$ is the principal ideal of $\CC[M_\calF]$ generated by $t^{-c}f$.

  If $b\in\calF$, then $t^{-b}f=t^{c-b}t^{-c}f$, which generates the same ideal in $\CC[M_\calF]$ as $t^{-c}f$, as
  $t^{b-c}$ is invertible in $\CC[M_\calF]$.
  This proves $(i)$.
  For $(ii)$, apply the map $p$ to $t^{-b}f$.  
\end{proof}
%%%%%%%%%%%%%%%%%%%%%%%%%%%%%%%%%%%%%%%%%%%%%%%%%%%%%%%%%%%%%%%%%%%%%%%%%%%%%%%%%%%%%%%%%%%%%%%%%%%%

We relate this to the critical point equations.
Recall that $\Phi$ is the dispersion polynomial~\eqref{Eq:dispersionPolynomial} with support $\calA$ and
Newton polytope $A$ and that $\Psi$ is the critical point equations~\eqref{Eq:CPE}.
For $b\in\ZZ^{d+1}$, let $\defcolor{t^{-b}\Psi}\vcentcolon=\{ t^{-b}f \mid f\in\Psi\}$, and the same for any
initial system $\ini_\eta\Psi$.

Given a face $F$ of $A$ exposed by $\eta$, we call $\defcolor{\ini_\eta\BV}\vcentcolon=\overline{\BV}\cap\calO_\calF$ the
\demph{initial Bloch variety}.

%%%%%%%%%%%%%%%%%%%%%%%%%%%%%%%%%%%%%%%%%%%%%%%%%%%%%%%%%%%%%%%%%%%%%%%%%%%%%%%%%%%%%%%%%%%%%%%%%%%%
\begin{Corollary}\label{C:Facial_Systems}
  Let $F$ be a face of $A$ exposed by a vector $\eta$.
  The intersection of the compactified Bloch variety with $V_\calF$ is defined by $t^{-b}\Phi$ for any $b\in\calF$.
  Its component along the orbit $\calO_\calF$, the initial Bloch variety, is defined by $t^{-b}\ini_\eta\Phi$.
  The set of critical points in $V_\calF$ is defined by $t^{-b}\Psi$ for any $b\in \calF$ and the asymptotic critical
  points lying along $\calO_\calF$ are defined by $t^{-b}\ini_\eta\Psi$ for any $b\in\calF$.
\end{Corollary}
%%%%%%%%%%%%%%%%%%%%%%%%%%%%%%%%%%%%%%%%%%%%%%%%%%%%%%%%%%%%%%%%%%%%%%%%%%%%%%%%%%%%%%%%%%%%%%%%%%%%

Corollary~\ref{C:Facial_Systems} describes the equations for asymptotic critical points.
We turn our attention to the multiplicities of such points.

%%%%%%%%%%%%%%%%%%%%%%%%%%%%%%%%%%%%%%%%%%%%%%%%%%%%%%%%%%%%%%%%%%%%%%%%%%%%%%%%%%%%%%%%%%%%%%%%%%%%
\begin{Example}\label{Ex:TV_singularHouse}
  Consider affine toric varieties associated to faces from Example~\ref{Ex:singularHouse}.
  Let $E$ be the edge in Figure~\ref{F:singularHouse} and $\tau_\calE$ the cone~\eqref{Eq:tau_F}.
  Figure~\ref{F:doubleCone} 
  %%%%%%%%%%%%%%%%%%%%%%%%%%%%%%%%%%%%%%%%%%%%%%%%%%%%%%%%%%%%%%%%%%%%%%%%%%%%%%%%%%%%%%%%%%%%%%%%%%%%
  \begin{figure}[htb]
    \centering
    \begin{picture}(114,102)(-12,0)
      \put(0,0){\includegraphics{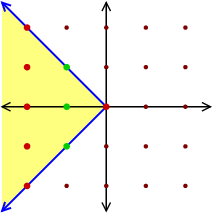}}
      \put(-12,70){$\tau_\calE$}  \put(54,54){$\bfzero$}
    \end{picture}
      \qquad
    \begin{picture}(140,102)
        \put(0,0){\includegraphics[height=102pt]{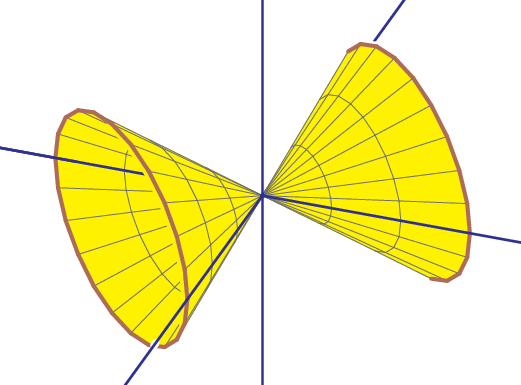}}
        \put(100,20){$Y_\calE$}
        \put(26,3){$a$} \put(1,65){$c$} \put(72,92){$b$}
    \end{picture}

    \caption{Cone $\tau_\calE$ and associated affine toric variety $Y_\calE$.}
    \label{F:doubleCone}
  \end{figure}
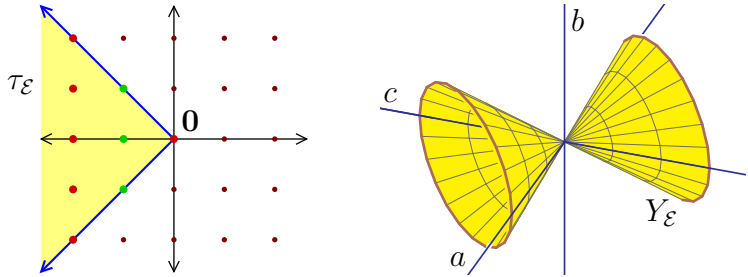
  %%%%%%%%%%%%%%%%%%%%%%%%%%%%%%%%%%%%%%%%%%%%%%%%%%%%%%%%%%%%%%%%%%%%%%%%%%%%%%%%%%%%%%%%%%%%%%%%%%%%
  shows $\tau_\calE$ and the corresponding affine toric variety $Y_\calE$,
  which is the closure of the image of $(\CC^\times)^2$ under the map 
  \[
     (\CC^\times)^2\ \ni\ (x,y)\ \longmapsto\ (x^{-1}y\,,\ x^{-1}\,,\ x^{-1}y^{-1})\ \in\ \CC^3\,.
  \]
  Writing $(a,b,c)$ for the coordinates of $\CC^3$, we have that
  $Y_\calE=\Var(ac-b^2)$, which is a quadratic cone in $\CC^3$.
  Figure~\ref{F:doubleCone} shows the real points of $Y_\calE$.
  This has an isolated singularity of multiplicity two at the origin.
  Indeed,
  \[
    \bigr(\CC[a,b,c]/\langle ac-b^2\rangle\bigl)/\langle a,c \rangle\ \simeq\
    \CC[b]/\langle b^2 \rangle\,,
  \]
  which has dimension 2 as a complex vector space.
  As $\ZZ\calF\simeq\ZZ$, we have that 
  $V_\calE\simeq \CC^\times\times Y_\calE$ is singular along $\CC^\times\times\{(0,0,0)\}$ of constant
  multiplicity two.

  For the facets $G$ and $H$, we have
  \[
    \ZZ^3\ =\ \ZZ(\calG)\oplus \ZZ\cdot(-1,0,0)\
           =\ \ZZ(\calH)\oplus \ZZ\cdot(0,0,-1)\,,
  \]
  and thus $V_\calG\simeq V_\calH\simeq(\CC^\times)^2\times\CC$.
  Both affine toric varieties are smooth. 
\end{Example}
%%%%%%%%%%%%%%%%%%%%%%%%%%%%%%%%%%%%%%%%%%%%%%%%%%%%%%%%%%%%%%%%%%%%%%%%%%%%%%%%%%%%%%%%%%%%%%%%%%%%

Let $F$ be a face of $A$.
Let $\Sat(\calF)$ be the saturation of $\ZZ\calF$ in $\ZZ\calA$.
The composition of the inclusion $M_\calF\subset\ZZ\calA$ with the projection
$\ZZ\calA\twoheadrightarrow\ZZ\calA/\Sat(\calF)$
induces a map of monoids,
\[
   M_\calF\ \xrightarrow{\ \pi_\calF\ }\ \ZZ\calA/\Sat(\calF)\ \simeq\ \ZZ^{d+1-\dim(F)}
   \ =\ \ZZ^{\codim(F)}\,.
\]
Write $\defcolor{\tau_\calF}\subset\ZZ^{\codim(F)}$ for the image monoid, 
$\pi_\calF(M_\calF)$.
Let $\defcolor{\conv(\tau_\calF)}\subset\RR^{\codim(F)}$ be the convex cone it generates and
$\defcolor{\conv(\tau^\circ_\calF)}\subset\conv(\tau_\calF)$ be the convex hull of $\tau_\calF\smallsetminus\{\bfzero\}$.
Finally, define $\defcolor{Q_\calF}\vcentcolon=\conv(\tau_\calF)\smallsetminus \conv(\tau^\circ_\calF)$ to be their difference,
which is a non-convex lattice polytope.
Define $\defcolor{\mu_\calF}\vcentcolon= \nvol(Q_\calF)$.

Consider these for $\tau_\calF=\NN\{(1,0),(2,1),(3,2),(2,-1), (3,-2)\}$ and the cone $\tau_\calE$ of
Figure~\ref{F:doubleCone},
\[
   \raisebox{-50.7pt}{\begin{picture}(114,102)(-12,0)
      \put(0,0){\includegraphics{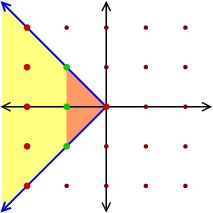}}
      \put(-12,70){$\tau_\calE$}  \put(54,54){$\bfzero$}
      \put(34,79){$Q_\calE$} \put(39,75){\vector(0,-1){20}}
    \end{picture}
    \qquad\quad
    \begin{picture}(120,102)
      \put(0,0){\includegraphics{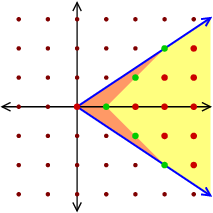}}
      \put(104,70){$\tau_\calF$}  \put(28,54){$\bfzero$}
      \put(53,85){$Q_\calF$} \put(57,81){\vector(0,-1){22}}
    \end{picture}}\ .
\]
Thus, $\mu_\calF=\nvol(Q_\calF)=4$, while $\mu_\calE=\nvol(Q_\calE)=2$, the multiplicity of the origin in $Y_\calE$.
This quantity $\mu_\calF$ is called a subdiagram volume and written $u(S/\Gamma)$ in~\cite[\S 5.3]{GKZ},
where $S=M_\calF$ and $\Gamma=\Sat(\calF)$.
%
%  The first is checked in the file QF.sing
%

%%%%%%%%%%%%%%%%%%%%%%%%%%%%%%%%%%%%%%%%%%%%%%%%%%%%%%%%%%%%%%%%%%%%%%%%%%%%%%%%%%%%%%%%%%%%%%%%%%%%
Suppose that a linear section $X_\calA\cap L$ of the projective toric variety $X_\calA$
contains a point $x$ that lies in an orbit $\calO_\calF$ in $\partial X_\calA$.
We determine the intersection multiplicity $m(x, X_\calA, L)$ of the linear section at $x$
when the intersection $X_\calF \cap (\PP^\calF\cap L)$ in $\PP^\calF$ is also proper.
Define $\defcolor{L_\calF}\vcentcolon= \PP^\calF\cap L$.
As  $X_\calF \cap L_\calF$ is proper we have $\dim(L_\calF)+\dim(X_\calF)=\dim(\PP^\calF)$.

%%%%%%%%%%%%%%%%%%%%%%%%%%%%%%%%%%%%%%%%%%%%%%%%%%%%%%%%%%%%%%%%%%%%%%%%%%%%%%%%%%%%%%%%%%%%%%%%%%%%
\begin{Theorem}\label{Th:multiplicity_at_infinity}
  Let $L$ be a linear subspace of\/ $\PP^\calA$ with $X_\calA\cap L$ and $X_\calF\cap L_\calF$ proper.
  Then 
  \begin{enumerate}[$(i)$]
  \item ${\displaystyle \deg(X_\calF)\ =\ \sum_{x\in X_\calF\cap L} m(x; X_\calF, L_\calF)\,,}$ and 

  \item for $x\in \calO_\calF\cap L$, 
    \[
       m(x; X_\calA, L)\ \geq\  \mu_\calF \cdot [ \Sat(\calF) : \ZZ\calF] \cdot m(x; X_\calF, L_\calF)\,,
    \]
    with equality if $L$ is general given that $\PP^\calF\cap L=L_\calF$.
  \end{enumerate}
\end{Theorem}
%%%%%%%%%%%%%%%%%%%%%%%%%%%%%%%%%%%%%%%%%%%%%%%%%%%%%%%%%%%%%%%%%%%%%%%%%%%%%%%%%%%%%%%%%%%%%%%%%%%%

%%%%%%%%%%%%%%%%%%%%%%%%%%%%%%%%%%%%%%%%%%%%%%%%%%%%%%%%%%%%%%%%%%%%%%%%%%%%%%%%%%%%%%%%%%%%%%%%%%%%
\begin{proof}
  Statement $(i)$ is Proposition~\ref{P:Degree_Section} for $X_\calF\cap L_\calF$.
  For $(ii)$,  $\mu_\calF \cdot [ \Sat(\calF) : \ZZ\calF]$ is the multiplicity of the
  toric variety $X_\calA$ along points of the orbit $\calO_\calF$~\cite[Ch.~5, Thm.\ 3.16]{GKZ}.
  This implies the result when  $m(x; X_\calF, L_\calF)=1$.
  The result follows by  perturbing $L_\calF$ and the dynamic interpretation of
  intersection multiplicities from Remark~\ref{Rem:intMult}.
\end{proof}
%%%%%%%%%%%%%%%%%%%%%%%%%%%%%%%%%%%%%%%%%%%%%%%%%%%%%%%%%%%%%%%%%%%%%%%%%%%%%%%%%%%%%%%%%%%%%%%%%%%%

%%%%%%%%%%%%%%%%%%%%%%%%%%%%%%%%%%%%%%%%%%%%%%%%%%%%%%%%%%%%%%%%%%%%%%%%%%%%%%%%%%%%%%%%%%%%%%%%%%%%
%
%
\section{Vertical faces}\label{S:vertical}

When the Newton polytope $A$ of a periodic graph has a vertical face $F$, there will be $\nvol(F)$ asymptotic critical
points on $X_\calF$.
If $F$ has a vertical face $E$, then at least $\nvol(E)$ of these will lie on $X_\calE$.
As observed in Example~\ref{Ex:TV_singularHouse}, $X_\calA$ may be singular along $X_\calE$, leading to additional
contributions.
Complicating this, $E$ may lie on more than one vertical face.
Before we do the accounting, consider the following example.

%%%%%%%%%%%%%%%%%%%%%%%%%%%%%%%%%%%%%%%%%%%%%%%%%%%%%%%%%%%%%%%%%%%%%%%%%%%%%%%%%%%%%%%%%%%%%%%%%%%%
\begin{Example}\label{Ex:singularHouseCP}
 Recall the graph $\Gamma$ and polytope of Example~\ref{Ex:meet_singularHouse} with dispersion
 polynomial~\eqref{Eq:singHouseDispersonPolynomial}.
 The polytope has normalized volume $16$, and a symbolic computation with parameters shows that the 
 Bloch variety has eight critical points, which is the lower bound of Corollary~\ref{C:Corner_CPdeg}.
 These occur at the corner points $(\pm 1,\pm 1)$, for general values of the parameters.
 The corresponding critical values $\lambda$ are
\[
   \pm b\pm c +\frac{V(u)+V(v)}{2}\ \pm\
   \sqrt{ (a+d)^2\; +\; \bigl(\pm b \pm c + \tfrac{V(u)-V(v)}{2}\bigr)^2 }\ ,
\]
where the  $\pm$  first coordinate of the corner point
is the sign of $b$, and  the second  is the sign of $c$.
%
%  Frank verified this on May 8 2025
%
Thus, $\cpdeg(\Gamma)=8$ and the vertical faces contribute eight to the count of 16 critical points.

We verify that.
The Newton polytope has four vertical facets and four vertical edges.
The vertical facet $G$ exposed by $(-1,-1,0)$ has initial form 
\[
\lambda( bx + cy)\ + adx - V(v)(bx + cy)\ =\
 x\bigl( \lambda ( b+ cyx^{-1})\ + ad -V(v)(b +cyx^{-1})\bigr)\,,
\]
which has no critical points on the orbit $\calO_\calG\simeq(\CC^\times)^2$ corresponding to this facet.
(The coordinates along $\calO_\calG$ are $\lambda$ and $yx^{-1}$.)
Thus, the critical points on $X_\calG$ occur along its boundary.
The same holds for the other vertical facets.

There are four vertical edges, each of length 1.
Removing units, their facial forms are
\[
\lambda b + ad - bV(v) \qquad\mbox{or}\qquad  \lambda-V(v)\,.
\]
Each has one solution on the orbit corresponding to its edge ($\CC^\times$ with coordinate $\lambda$).

As each vertical facet $G$ has two vertical edges, the toric subvariety $X_\calG$ corresponding to $G$ 
has two solutions to the critical point equations, one for each edge.
This gives four solutions to the critical point equations along vertical faces, one-half of the expected
eight.

To account for this discrepancy, recall that the toric variety $X_\calA$ is singular of multiplicity $\mu_\calE=2$ along
each orbit $\calO_\calE$ corresponding to a vertical edge $E$.
Thus, the critical point equations have $8=4\cdot 2$ solutions counted with multiplicity on $\partial X_\calA$.
\end{Example}
%%%%%%%%%%%%%%%%%%%%%%%%%%%%%%%%%%%%%%%%%%%%%%%%%%%%%%%%%%%%%%%%%%%%%%%%%%%%%%%%%%%%%%%%%%%%%%%%%%%%

Let $\defcolor{L}\subset\PP^\calA$ be the linear space corresponding to the critical point equations~\eqref{Eq:CPE} and
assume that $X_\calA\cap L$ is proper so that it consists of $\deg(X_\calA)$ points.

%%%%%%%%%%%%%%%%%%%%%%%%%%%%%%%%%%%%%%%%%%%%%%%%%%%%%%%%%%%%%%%%%%%%%%%%%%%%%%%%%%%%%%%%%%%%%%%%%%%%
\begin{Lemma}\label{Lem:Vertical}
  For each vertical face $F$ of $A=\conv(\calA)$, the hypotheses of Theorem~\ref{Th:multiplicity_at_infinity}
  hold and $X_\calF\cap L$ consists of $\deg(X_\calF)$ points.
\end{Lemma}
%%%%%%%%%%%%%%%%%%%%%%%%%%%%%%%%%%%%%%%%%%%%%%%%%%%%%%%%%%%%%%%%%%%%%%%%%%%%%%%%%%%%%%%%%%%%%%%%%%%%
\begin{proof}
  Let $\defcolor{L_\calF}\vcentcolon=\PP^\calF\cap L$, which is defined by the linear forms corresponding to members of
  the critical point equations and $\{x_a=0\mid a\in\calA\smallsetminus \calF\}$.
  By Corollary~\ref{Cor:VerticalSubSystems}, those linear forms satisfy $\codim(F)$-many linear equations.
  Consequently, $L_\calF$ has codimension in $\PP^\calF$ at most $d{+}1{-}\codim(F)=\dim(F)$.
  As this is the dimension of $X_\calF$, $\emptyset\neq X_\calF\cap L_\calF$.
  Since
  \[
    X_\calF\cap L_\calF\ \subset\ X_\calF\cap L\ \subset\ X_\calA\cap L\,,
  \] 
  $X_\calF\cap L_\calF$ has dimension zero and is therefore a proper intersection.
\end{proof}
%%%%%%%%%%%%%%%%%%%%%%%%%%%%%%%%%%%%%%%%%%%%%%%%%%%%%%%%%%%%%%%%%%%%%%%%%%%%%%%%%%%%%%%%%%%%%%%%%%%%

%%%%%%%%%%%%%%%%%%%%%%%%%%%%%%%%%%%%%%%%%%%%%%%%%%%%%%%%%%%%%%%%%%%%%%%%%%%%%%%%%%%%%%%%%%%%%%%%%%%%
\begin{Definition}\label{Def:Nvert}
  If the only vertical faces of $A$ are facets (see, e.g.~Figure~\ref{F:Oblique}), then we set
  \[
   \defcolor{\Nvert}\ \vcentcolon=\
   \sum_{F\mbox{\scriptsize\ vertical}} \ 
   \mu_\calF \cdot [ \Sat(\calF) : \ZZ\calF] \cdot \deg(X_\calF)\ ,
 \]
 which is combinatorial.
 For such a face $F$, $\calO_\calF\subsetneq X_\calF$, and this is the only orbit corresponding to a vertical face in $X_\calF$.
 Then Lemma~\ref{Lem:Vertical} implies that 
  \begin{equation}\label{Eq:Vert_Ineq}
   \Nvert\ \geq\
   \sum_{F\mbox{\scriptsize\ vertical}} \ \sum_{x\in\calO_\calF\cap L}\quad
   \mu_\calF \cdot [ \Sat(\calF) : \ZZ\calF] \cdot m(x; X_\calF, L_\calF)\ .
 \end{equation}

 If there exist vertical faces of $A$ of codimension more than 1, then we set
 \[
   \defcolor{\Nvert}\ \vcentcolon=\
   \sum_{F\mbox{\scriptsize\ vertical}} \ \sum_{x\in\calO_\calF\cap L}\quad
   \mu_\calF \cdot [ \Sat(\calF) : \ZZ\calF] \cdot m(x; X_\calF, L_\calF)\ .
 \]
 Notice that inequality~\ref{Eq:Vert_Ineq} holds in either case.
\end{Definition}
%%%%%%%%%%%%%%%%%%%%%%%%%%%%%%%%%%%%%%%%%%%%%%%%%%%%%%%%%%%%%%%%%%%%%%%%%%%%%%%%%%%%%%%%%%%%%%%%%%%%

Let \defcolor{$\partial_{{\rm vert}}X_\calA$}  be the union of all toric subvarieties
 $X_\calF$ for $F$ a proper vertical face of $A$.

%%%%%%%%%%%%%%%%%%%%%%%%%%%%%%%%%%%%%%%%%%%%%%%%%%%%%%%%%%%%%%%%%%%%%%%%%%%%%%%%%%%%%%%%%%%%%%%%%%%%
\begin{Corollary}\label{C:verticalContribution}
  We have the following inequality, 
  \[
    \sum_{x\in \partial_{{\rm vert}}X_\calA}   m(x; X_\calA, L)\ \geq\ \Nvert\,.  
  \]
\end{Corollary}
%%%%%%%%%%%%%%%%%%%%%%%%%%%%%%%%%%%%%%%%%%%%%%%%%%%%%%%%%%%%%%%%%%%%%%%%%%%%%%%%%%%%%%%%%%%%%%%%%%%%
\begin{proof}
  Collecting the terms in the sum by the orbit they lie on gives
  \[
     \sum_{x\in \partial_{{\rm vert}}X_\calA}  m(x; X_\calA, L)  \ =\ \sum_{F\mbox{\scriptsize\ vertical}}\ 
    \sum_{x\in\calO_\calF\cap L}  m(x; X_\calA, L)\,.
  \]
  The lemma follows from the inequality of Theorem~\ref{Th:multiplicity_at_infinity}$(ii)$
  and definition~\eqref{Def:Nvert}.
\end{proof}
%%%%%%%%%%%%%%%%%%%%%%%%%%%%%%%%%%%%%%%%%%%%%%%%%%%%%%%%%%%%%%%%%%%%%%%%%%%%%%%%%%%%%%%%%%%%%%%%%%%%

%%%%%%%%%%%%%%%%%%%%%%%%%%%%%%%%%%%%%%%%%%%%%%%%%%%%%%%%%%%%%%%%%%%%%%%%%%%%%%%%%%%%%%%%%%%%%%%%%%%%
\begin{Remark}
  The definition of $\Nvert$ is not entirely combinatorial, as it is challenging to determine the sum in~\eqref{Eq:Vert_Ineq}.
  Nevertheless, this may be done in many cases, such as in Examples~\ref{Ex:singularHouseCP} and~\ref{Ex:singularity}.
\end{Remark}
%%%%%%%%%%%%%%%%%%%%%%%%%%%%%%%%%%%%%%%%%%%%%%%%%%%%%%%%%%%%%%%%%%%%%%%%%%%%%%%%%%%%%%%%%%%%%%%%%%%%

%%%%%%%%%%%%%%%%%%%%%%%%%%%%%%%%%%%%%%%%%%%%%%%%%%%%%%%%%%%%%%%%%%%%%%%%%%%%%%%%%%%%%%%%%%%%%%%%%%%%
%
%
\section{Initial graph disconnected}\label{S:Structural}

In Section~\ref{S:initGraph} we studied initial graphs and in Definition~\ref{Def:asymptoticallyDisconnected} we explained what
it means for a periodic graph $\Gamma$ to be asymptotically disconnected.
Lemma~\ref{Lem:asymptoticallydisconnected} noted the consequence of this for the initial form $\ini_\eta\Phi$ of the
dispersion polynomial $\Phi$.
Here, we quantify how this contributes to the enumeration of asymptotic critical points.
We begin with an example.

%%%%%%%%%%%%%%%%%%%%%%%%%%%%%%%%%%%%%%%%%%%%%%%%%%%%%%%%%%%%%%%%%%%%%%%%%%%%%%%%%%%%%%%%%%%%%%%%%%%%
\begin{Example}\label{Ex:hexPlus}
 
The periodic graph on the left in Figure~\ref{F:Oblique}
%%%%%%%%%%%%%%%%%%%%%%%%%%%%%%%%%%%%%%%%%%%%%%%%%%%%%%%%%%%%%%%%%%%%%%%%%%%%%%%%%%%%%%%%%%%%%%%%%%%%
\begin{figure}[htb]
  \centering
   \includegraphics{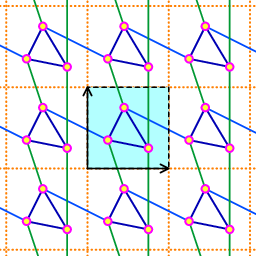}
    \qquad
   \begin{picture}(129,123)(-3,-3)
    \put(0,-2){\includegraphics{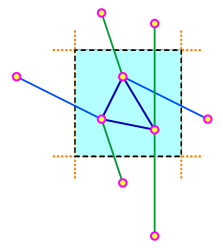}}
    \put(36,30){{\color{white}\circle*{9}}}
    \put(41,53){\small$u$}    \put(61.5, 81.5){\small$v$}   \put(77,51){\small$w$}
    \put(-2,84){\small$v{-}x$} \put(34,26){\small$v{-}y$}
    \put(25,109){\small$u{+}y$}   \put(105, 63){\small$u{+}x$}
    \put(78, 2){\small$w{-}y$}    \put( 78,106){\small$w{+}y$}
    \put(47, 71){\small$a$}
    \put(49,42){{\color{white}\circle*{9}}}    \put(49,94){{\color{white}\circle*{9}}}
    \put(25, 74){\small$b$}   \put(90, 68){\small$b$}
    \put(46.5, 92){\small$c$}   \put(47, 39){\small$c$}
    \put(60, 48){\small$d$}   \put(60, 65){\small$e$}
    \put(76, 84){\small$f$}   \put(76, 25){\small$f$}
   \end{picture}
   \quad
   \begin{picture}(142,123)
       \put(0,0){\includegraphics{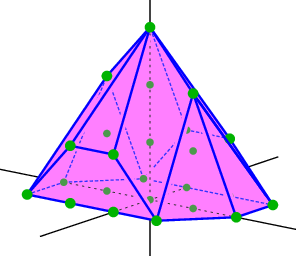}}
       \put(92,92){$A$}
    \put(25,4){$x$} \put(130,6){$y$} 
   \end{picture}
   \caption{A $\ZZ^2$-periodic graph, its edge labels, and Newton polytope.}
   \label{F:Oblique}
\end{figure}
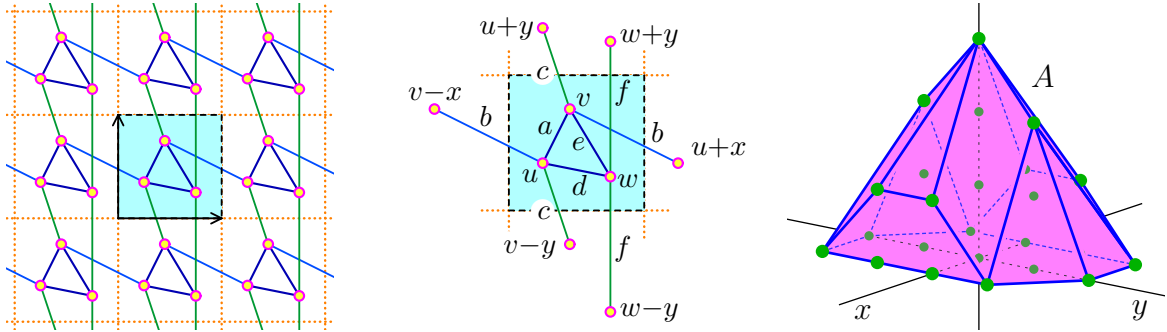
%%%%%%%%%%%%%%%%%%%%%%%%%%%%%%%%%%%%%%%%%%%%%%%%%%%%%%%%%%%%%%%%%%%%%%%%%%%%%%%%%%%%%%%%%%%%%%%%%%%%
has six (orbits of) edges and three vertices in the fundamental domain.
Let $u,v,w$ be the values of the potential at the eponymous vertices, and $a,b,c,d,e,f$ be edge weights, as indicated
in the middle diagram in Figure~\ref{F:Oblique}.
We have also simplified notation, writing $x$ for $(1,0)$ and $y$ for $(0,1)$.

Here is its characteristic matrix (in the ordered basis corresponding to $u,v,w$),
\begin{equation}\label{Eq:hexPlusMatrix}
  M\ \vcentcolon=\ 
  \left(
  \begin{array}{ccc}
     \underline{\lambda} - u & \underline{a}+bx^{-1}+cy^{-1} & d \\
        a+\underline{bx}+cy  &  \underline{\lambda} - v     & e \\
           d     &       e          & \underline{\lambda}-w + \underline{fy}+ fy^{-1})   
  \end{array}
   \right)\;.
\end{equation}
%(Ignore the underlined entries for now.)
The Newton polytope has two vertical faces (the trapezoids in Figure~\ref{F:Oblique_Facets}) and
four faces whose initial graphs are disconnected.
%%%%%%%%%%%%%%%%%%%%%%%%%%%%%%%%%%%%%%%%%%%%%%%%%%%%%%%%%%%%%%%%%%%%%%%%%%%%%%%%%%%%%%%%%%%%%%%%%%%%
\begin{figure}[htb]
\centering
  \begin{picture}(130,115)(0,0)
    \put(4,0){\includegraphics{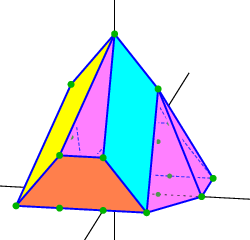}}
    \thicklines
     {\color{white}\put(10,55.7){\vector(1,0){27.8}}\put(10,56.3){\vector(1,0){27.8}}}
    \thinlines
    \put(0,53){$G$}\put(10,56){\vector(1,0){26.5}}
    \put(39,22){$V$}  
    \put(61,53){$F$}  \put(70,88){$A$}
    \put(36,1){$x$} \put(118,11){$y$} \put(61,106){$\lambda$}
  \end{picture}

  \caption{Facets $V$, $F$, and $G$ with asymptotic critical points.}
  \label{F:Oblique_Facets}
\end{figure}
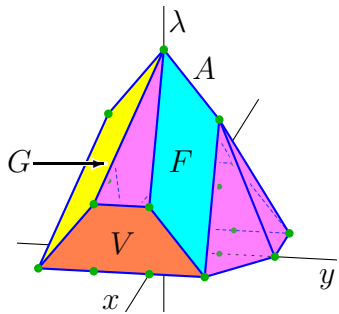
%%%%%%%%%%%%%%%%%%%%%%%%%%%%%%%%%%%%%%%%%%%%%%%%%%%%%%%%%%%%%%%%%%%%%%%%%%%%%%%%%%%%%%%%%%%%%%%%%%%%
All have asymptotic critical points.
The face $F$ is exposed by the vector $\eta=(-2,-1,-1)$.
It consists of those points of the Newton polytope with minimum weight
$-3$ with respect to $\eta$.
The underlined entries in the characteristic matrix~\eqref{Eq:hexPlusMatrix} form the initial matrix.
The edge weights $d$ and $e$ are the initial forms of their entries, but do not contribute to the initial matrix,
which has block form,
 \begin{equation}\label{Eq:hexPlusInitial_F}
  \ini_\eta M\ =\ \ \left(
  \begin{array}{ccc}
    \lambda &     a   & 0\\
      bx    & \lambda & 0 \\
      0     &     0   & \lambda + fy
  \end{array}
  \right)\; ,
 \end{equation}
with no block a monomial.
This is apparent from the initial graph $\ini_\eta\hGamma$,
\[
  \begin{picture}(158,76)(0,-6)
   \put(0,0){\includegraphics{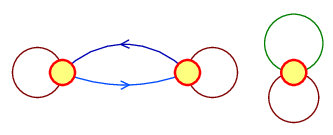}}
   \put( 27,  27){$u$}        \put(87, 27){$v$}   \put(137, 27){$w$}
   \put( 14,  44){$\lambda$}  \put( 47,  44){$a$}     
   \put(100,  44){$\lambda$}  \put( 47,  14){$bx$}
   \put(137,  -5){$\lambda$}  \put(135,62){$fy$}
  \end{picture}\;  \raisebox{35pt}{.}
\]

By Theorem~\ref{Th:InitialGraphForm},
 \begin{eqnarray*}
   \ini_\eta\Phi\ =\ \det \ini_\eta M& = &(\lambda^2-abx)(\lambda + fy)\ =\ \lambda^3+fy\lambda^2-abx\lambda-abfxy\\
   &=& \lambda^3(1 - abx\lambda^{-2})(1 + fy\lambda^{-1})\,.
 \end{eqnarray*}
The four monomials $\{\lambda^3, y\lambda^2, x\lambda, xy\}$ correspond to the vertices of the facet $F$ and form the set
$\calF$. 
The span $\ZZ\calF$ of their differences is
\[
\ZZ\calF\ =\ \ZZ(1,0,-2)\oplus\ZZ(0,1,-1)\,.
\]
Thus 
$\calO_\calF\simeq (\CC^\times)^2$ with coordinates $x\lambda^{-2}$ and $y\lambda^{-1}$.
Also, $X_\calF\simeq\PP^1\times\PP^1$, reflecting that $F$ is a parallelogram that is the Minkowski sum of two primitive
line segments. 
Thus, the initial Bloch variety $\ini_\eta\BV=\Var(\ini_\eta\Phi)$ is the union of two coordinate lines on $X_\calF$ that intersect, and has one
singular point, $p_F$. 
As $\ZZ^3=\ZZ\calF + \ZZ(0,0,-1)$, $\ZZ\calF$ equals its saturation and $X_\calA$ is smooth along $\calO_\calF$,
which implies that the multiplicities $\mu_\calF$ and  $[ \Sat(\calF) : \ZZ\calF]$ in Definition~\ref{D:N_Disc} are both 1.
\end{Example}
%%%%%%%%%%%%%%%%%%%%%%%%%%%%%%%%%%%%%%%%%%%%%%%%%%%%%%%%%%%%%%%%%%%%%%%%%%%%%%%%%%%%%%%%%%%%%%%%%%%%

Let us now consider the general case.
Suppose that $\Gamma$ is asymptotically disconnected in the direction of $\eta\in\ZZ^{d{+}1}$.
Let $F$ be the face of $\calN(\Phi)$ exposed by $\eta$ and set $\calF\vcentcolon=F\cap A$.
Let
\[
\ini_\eta\Phi\ =\  \mu g_1\cdot g_2\dotsb g_r
\]
be the factorization of the initial form given by Lemma~\ref{Lem:asymptoticallydisconnected}.
Then
\[
  \ini_\eta\BV\ =\ \Var(g_1)\cup \Var(g_2)\cup \dotsb \cup\Var(g_r)\,.
\]
As the factors $g_i$ arise from disjoint components of $\ini_\eta\hGamma$, the parameters $E,V$ in one factor $g_i$ are
distinct from the parameters in other factors.

Before we explain the consequence for asymptotic critical points, we need a definition.
Suppose that $P$ and $Q$ are lattice polytopes in $\ZZ^2$---polygons or degenerate line segments.
Their \demph{mixed area}, $\MA(P,Q)$, is
\[
   \defcolor{\MA(P,Q)}\ \vcentcolon=\
       \area(P+Q) - \area(P) - \area(Q)\,,
\]
where $\area$ is the usual area in $\RR^2$ in which the unit square has unit area.
Figure~\ref{F:MinkowskiSum}
%%%%%%%%%%%%%%%%%%%%%%%%%%%%%%%%%%%%%%%%%%%%%%%%%%%%%%%%%%%%%%%%%%%%%%%%%%%%%%%%%
 \begin{figure}[htb]
 \[
   P\ =\ \raisebox{-20pt}{\includegraphics{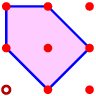}}
   \qquad\quad
   Q\ =\ \raisebox{-20pt}{\includegraphics{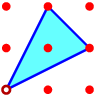}}
   \qquad\quad
   P+Q\ =\ \raisebox{-40pt}{\includegraphics{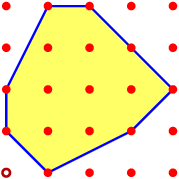}}
 \]
 \caption{As $\area(P)=5/2$, $\area(Q)=3/2$, and $\area(P+Q)=10$, the mixed area is
 $\MA(P,Q)=10-5/2-3/2=6$.}
 \label{F:MinkowskiSum}
\end{figure}
%%%%%%%%%%%%%%%%%%%%%%%%%%%%%%%%%%%%%%%%%%%%%%%%%%%%%%%%%%%%%%%%%%%%%%%%%%%%%%%%%
shows an example.
A one-dimensional lattice polytope $F$ is a line segment and its \demph{lattice length} $\ell(F)$ satisfies
\[
\ell(F)+1\ =\ \#( F\cap \ZZ\calF)\,,
\]
the 1-dimensional volume of $F$ measured with respect to the 1-dimensional lattice $\ZZ\calF$.

%%%%%%%%%%%%%%%%%%%%%%%%%%%%%%%%%%%%%%%%%%%%%%%%%%%%%%%%%%%%%%%%%%%%%%%%%%%%%%%%%%%%%%%%%%%%%%%%%%%%
\begin{Lemma}\label{L:asymptoticSing}
  Suppose that the  $V,E$ are general and $\Gamma$ is asymptotically disconnected in the direction of 
  $\eta\in\ZZ^{d+1}$ with $F$ the face of $\calN(\Phi)$ exposed by $\eta$.
  We have the trichotomy.
\begin{enumerate}[$(i)$]
  \item If $\dim F=1$, then $\ini_\eta\BV$ consists of $\ell(F)$ points.

  \item If $\dim F > 2$, then for any $1\leq i<j\leq r$ either $\Var(g_i)\cap\Var(g_j)=\emptyset$ or
    the intersection $\Var(g_i)\cap\Var(g_j)$ is a positive-dimensional set of singular points on $\ini_\eta\BV$,
    which are all asymptotic critical points.

  \item If $\dim F=2$, then the initial Bloch variety $\ini_\eta\BV$ has
    \begin{equation}\label{Eq:MA_sum}
    \sum_{1\leq i<j\leq r} \MA(\calN_i,\calN_j)
    \end{equation}
    singular points, counted with multiplicity as a subset of $\calO_\calF$.
    All are asymptotic critical points of $\BV$ in the direction $\eta$.
    Here, $\calN_i\subset\ZZ\calF$ is the Newton polytope of the factor $g_i$.
\end{enumerate}
\end{Lemma}
%%%%%%%%%%%%%%%%%%%%%%%%%%%%%%%%%%%%%%%%%%%%%%%%%%%%%%%%%%%%%%%%%%%%%%%%%%%%%%%%%%%%%%%%%%%%%%%%%%%%
\begin{proof}
  Statement $(i)$ holds as $\ini_\eta\Phi$ is a univariate polynomial of degree $\ell(F)$, in the coordinates of $\calO_\calF$.
%  As each factor $g_i$ of $\ini_\eta\Phi$ has independent, general parameters, no two have a common zero.

  For $(ii)$, $\calO_\calF$ is at least 3-dimensional.
  As the factors $g_i$ are irreducible and distinct, the intersection $\Var(g_i)\cap\Var(g_j)$ is either empty or has dimension $n-2\geq 1$. 
  If non-empty, then the intersection consists of singular points of $\ini_\eta\BV$.
  Statement $(ii)$ holds by Proposition~\ref{P:SolutionsAtInfinity}$(ii)$.

  For $(iii)$, if $i\neq j$, then  $\Var(g_i)\cap\Var(g_j)$ consists of singular points, each is an asymptotic singular point
  of $\BV$ in the direction of $\eta$  and hence is an asymptotic critical point.
  If $g_i,g_j$ are general, then this consists of $\MA(\calN_i,\calN_j)$ points,  counted with multiplicity.
  Generality is ensured, as they depend on independent parameters.
  Generality also implies that all triple intersections are empty.
  This proves~\eqref{Eq:MA_sum}.
\end{proof}
%%%%%%%%%%%%%%%%%%%%%%%%%%%%%%%%%%%%%%%%%%%%%%%%%%%%%%%%%%%%%%%%%%%%%%%%%%%%%%%%%%%%%%%%%%%%%%%%%%%%
 
%%%%%%%%%%%%%%%%%%%%%%%%%%%%%%%%%%%%%%%%%%%%%%%%%%%%%%%%%%%%%%%%%%%%%%%%%%%%%%%%%%%%%%%%%%%%%%%%%%%%
\begin{Definition}\label{D:N_Disc}
  Let $F$ be a two-dimensional oblique face of $\calN(\Phi)$ such that $\Gamma$ is asymptotically disconnected
  in the direction of $\eta$, where $\eta$ exposes $F$.
  Define
  \[
     \defcolor{\Ndisc(F)}\ \vcentcolon=\ \mu_\calF [ \Sat(\calF) : \ZZ\calF]  \sum_{1\leq i<j\leq r} \MA(\calN_i,\calN_j)\,,
  \]
  the product of the sum~\eqref{Eq:MA_sum} and the multiplicity of the orbit $\calO_\calF$ in $X_\calA$.

  For all other faces $F$ set $\Ndisc(F)=0$ and let \defcolor{$\Ndisc$} be the sum of $\Ndisc(F)$ over all faces.
\end{Definition}  
%%%%%%%%%%%%%%%%%%%%%%%%%%%%%%%%%%%%%%%%%%%%%%%%%%%%%%%%%%%%%%%%%%%%%%%%%%%%%%%%%%%%%%%%%%%%%%%%%%%%

%%%%%%%%%%%%%%%%%%%%%%%%%%%%%%%%%%%%%%%%%%%%%%%%%%%%%%%%%%%%%%%%%%%%%%%%%%%%%%%%%%%%%%%%%%%%%%%%%%%%
\begin{Lemma}\label{L:N_Disc}
  With these definitions, and where $L_\Psi$ is the linear space corresponding to the critical point equations $\Psi$, we have
  \begin{equation}
    \label{Eq:atInfinity}
  \sum_{x\in{\textrm sing}(\ini_\eta\BV)}  m(x; X_\calA, L_\Psi)
  \ \geq\ \Ndisc(F)\,.
 \end{equation}
\end{Lemma}  
%%%%%%%%%%%%%%%%%%%%%%%%%%%%%%%%%%%%%%%%%%%%%%%%%%%%%%%%%%%%%%%%%%%%%%%%%%%%%%%%%%%%%%%%%%%%%%%%%%%%
\begin{proof}
  A singular point $x\in\Var(g_i)\cap\Var(g_j)$ has multiplicity $m(x,X_\calA,L_\Psi)$ as a critical point that is at least the product of
  its multiplicity as a point of $\Var(g_i)\cap\Var(g_j)$ and the multiplicity $\mu_\calF [ \Sat(\calF) : \ZZ\calF]$ of the
  orbit $\calO_\calF$ in $X_\calA$ (This is Theorem~\ref{Th:multiplicity_at_infinity}$(ii)$).
  Since such points are a subset of the singular points of $\ini_\eta\BV$, we deduce~\eqref{Eq:atInfinity}.
\end{proof}
%%%%%%%%%%%%%%%%%%%%%%%%%%%%%%%%%%%%%%%%%%%%%%%%%%%%%%%%%%%%%%%%%%%%%%%%%%%%%%%%%%%%%%%%%%%%%%%%%%%%

Let $L_\Psi$ be the linear space corresponding to the critical point equations $\Psi$ and
define \defcolor{$\partial_{{\rm obl}}X_\calA$} to be the union of all toric subvarieties
$X_\calF$ for $F$ an oblique face of $A$.

%%%%%%%%%%%%%%%%%%%%%%%%%%%%%%%%%%%%%%%%%%%%%%%%%%%%%%%%%%%%%%%%%%%%%%%%%%%%%%%%%%%%%%%%%%%%%%%%%%%%
\begin{Corollary}\label{C:obliqueContribution}
  We have the following inequality, 
  \[
    \sum_{x\in \partial_{{\rm obl}}X_\calA}   m(x; X_\calA, L_\Psi)\ \geq\ \Ndisc\,.  
  \]
\end{Corollary}
%%%%%%%%%%%%%%%%%%%%%%%%%%%%%%%%%%%%%%%%%%%%%%%%%%%%%%%%%%%%%%%%%%%%%%%%%%%%%%%%%%%%%%%%%%%%%%%%%%%%
\begin{proof}
  Collecting the terms in the sum by the orbit they lie on gives
  \[
     \sum_{x\in \partial_{{\rm obl}}X_\calA}  m(x; X_\calA, L_\Psi)  \ =\ \sum_{F\mbox{\scriptsize\ oblique}}\ 
    \sum_{x\in\calO_\calF\cap L_\Psi}  M(x; X_\calA, L_\Psi)\,.
  \]
  The lemma follows from the inequality of Theorem~\ref{Th:multiplicity_at_infinity}$(ii)$
  and Lemma~\ref{L:N_Disc}.
\end{proof}
%%%%%%%%%%%%%%%%%%%%%%%%%%%%%%%%%%%%%%%%%%%%%%%%%%%%%%%%%%%%%%%%%%%%%%%%%%%%%%%%%%%%%%%%%%%%%%%%%%%%

%%%%%%%%%%%%%%%%%%%%%%%%%%%%%%%%%%%%%%%%%%%%%%%%%%%%%%%%%%%%%%%%%%%%%%%%%%%%%%%%%%%%%%%%%%%%%%%%%%%%
\begin{Example}\label{Ex:singularity}
  Let us now consider the facet $G$ of the polytope in Figures~\ref{F:Oblique} and~\ref{F:Oblique_Facets}.
  This is exposed by the vector $\eta=\langle -1,1,-1\rangle$, which takes value $-3$ on $G$.
  The initial graph has the same structure as in Example~\ref{Ex:hexPlus}, but with a different labeling,
\[
  \begin{picture}(158,76)(0,-6)
   \put(0,0){\includegraphics{images/iniHat_hexPlusF}}
   \put( 27,  27){$u$}        \put(87, 27){$v$}   \put(137, 27){$w$}
   \put( 15,  44){$\lambda$}  \put( 67,  46){$cy^{-1}$}   
   \put(100,  44){$\lambda$}  \put( 43,  14){$bx$}
   \put(137,  -5){$\lambda$}  \put(135,62){$fy^{-1}$}
  \end{picture}\;  \raisebox{35pt}{.}
\]
  Here is the facial form,
 \begin{equation}\label{Eq:FacialG}
  (\lambda^2\ -\ bc xy^{-1})(\lambda + f y^{-1})\,.
 \end{equation}
 As  Example~\ref{Ex:hexPlus}, the monomials $\{\lambda^3, y^{-1}\lambda^2, xy^{-1}\lambda, xy^{-2}\}$
 correspond to the vertices of the facet $G$ and form the set $\calG$. 
The span $\ZZ\calG$ of their differences is
\[
\ZZ\calG\ =\ \ZZ(1,-1,-2)\oplus\ZZ(0,-1,-1)\ =\ \ZZ(1,0,-1)\oplus\ZZ(0,-1,-1)\,.
\]
Thus 
$\calO_\calG\simeq (\CC^\times)^2$ with coordinates $x\lambda^{-1}$ and $y^{-1} \lambda^{-1}$.
Also, $X_\calG\simeq\PP^1\times\PP^1$, reflecting that $G$ is a parallelogram that is the Minkowski sum of two primitive
line segments. 
Thus the initial Bloch variety $\ini_\eta\BV=\Var(\ini_\eta\Phi)$ is the union of two coordinate lines on $X_\calG$ that intersect, and has one
singular point, $p_G$. 
As $\ZZ^3=\ZZ\calG + \ZZ(0,0,-1)$, $\ZZ\calG$ equals its saturation and $X_\calA$ is smooth along $\calO_\calG$,
which implies that the multiplicities $\mu_\calG$ and  $[\Sat(\calG) : \ZZ\calF]$ in Definition~\ref{D:N_Disc} are both 1.

The point $p_G$ has multiplicity two on the compactified Bloch variety.
We study it in $V_\calG$.
If we multiply each entry of the characteristic matrix by $\lambda^{-1}$, then each entry lies in
\[
\CC[M_\calG]\ =\ \CC[x\lambda^{-1}, x^{-1}\lambda, y^{-1}\lambda^{-1}, y\lambda][\lambda^{-1}]
           \ =\ \CC[\zeta,\zeta^{-1}, \xi, \xi^{-1}][\rho]\,,
\]
where $\defcolor{\zeta}=x\lambda^{-1}$, $\defcolor{\xi}=y^{-1}\lambda^{-1}$, and $\defcolor{\rho}=\lambda^{-1}$.
This ring is graded by the degree of the indeterminate $\rho$.
It is the coordinate ring of $V_\calG$ and $\rho=0$ defines the orbit $\calO_\calG$.
This shows that $V_\calG\simeq(\CC^\times)^2\times\CC$
Here is the characteristic matrix (in the coordinates of $\CC[M_\calG]$),
 \begin{equation}\label{Eq:CharMatrixG}
  \left(
  \begin{array}{ccc}
    1 -u \rho      & a\rho + b\zeta^{-1}\rho^2 + c\xi &    d\rho\\
     a\rho + b\zeta + c\xi^{-1}\rho^2&  1 -v \rho   & e\rho \\
      d\rho   &  e\rho    &  1-w \rho +f\xi^{-1}\rho^2 + f \xi
  \end{array}
  \right)\ .
 \end{equation}
Its determinant $\Phi$ defines the Bloch variety in $V_\calG$.
We write $\Phi$ as a polynomial in $\rho$ (suppressing terms of degree greater than 2),
 \begin{multline}\label{Eq:expandedPhi}
   (1-b\zeta\xi)(1+f\xi)\ \ -\ \ \rho\bigl(w(1-b\zeta\xi) + (1+f\xi)(u+v+ac\xi +ab\zeta) \bigr) \\
  \ +\  \rho^2(uw + vw -e^2-d^2 +b\zeta(de-cf + aw) + c\xi(aw + de) + f\xi^{-1}) \\
  \ +\     \rho^2 (1+f\xi)(uv-a^2-b^2-c^2)\ +\ \rho^3( \dotsb \qquad
 \end{multline}
The terms of degree zero in $\rho$, $(1+f\xi)(1-b\zeta\xi)$, give the initial form in these coordinates.

We use this expression~\eqref{Eq:expandedPhi} to study the point $p_G$.
Observe that in $V_\calG$, the point $p_G$ is defined by the vanishing of $\rho$, $\defcolor{\tau}\vcentcolon= 1-b\zeta\xi$,
and $\defcolor{\sigma}\vcentcolon=1+f\xi$.
These provide local coordinates for $V_\calG$ at the point $p_G$.
Similarly, the terms $w$,  $\defcolor{f}\vcentcolon=(u+v+ac\xi +ab\zeta)$, the coefficient \defcolor{$g$} of $\rho^2$ in the second line,
and $\defcolor{h}=uv-a^2-b^2-c^2$  are all non-zero at $p_G$. 
(This uses the assumption that the graph parameters $a,\dotsc,f,u,v,w$ are general.)
Let us rewrite the expression~\eqref{Eq:expandedPhi} for $\Phi$ in these local coordinates
 \begin{equation}
   \label{Eq:simplifiedPhi}
  \Phi\ =\ \tau\sigma\ -\ w\tau\rho\ -\ f\sigma\rho \ +\ g\rho^2\ + h\sigma\rho^2\ +\ \rho^3(\dotsb
 \end{equation}

This vanishes at $p_G$, recovering that $p_G$ lies on the compactified Bloch variety , $\overline{\BV}$.
As there are no terms in $\Phi$ that are linear in the local coordinates $\rho,\sigma,\tau$ for $p_G$, we see that $p_G$ is a
singular point on $\BV$, and thus its multiplicity as a critical point is at least two.
The Hessian matrix of $\Phi$ at the point $p_G$ (with coordinates in the order $(\tau,\sigma,\rho)$) is
\[
\left(\begin{array}{ccc}
  0  &   1   & -w \\
  1  &   1   & -f(p_G)\\
  -w &-f(p_G)&  g(p_G)\end{array}\right)\ .
\]
This has full rank, so the multiplicity of $p_G$ is exactly two.

For the face $F$, a similar computation shows that the critical point $p_F$ studied in Example~\ref{Ex:hexPlus} is a smooth point
of the Bloch variety, and hence its multiplicity as a critical point is 1.
Indeed, the expression in local coordinates at $p_F$ similar to~\eqref{Eq:simplifiedPhi} has the term $bdex\, \rho$, which is
linear in $\rho$ as $bdex$ is nonzero at $p_F$.
The difference in the cases for the faces $F$ and $G$ comes from the structure of the Floquet matrix and ultimately of the graph.
Unlike the analysis leading to the (dis)connectivity of the initial graph, we do not yet understand this singularity/smoothness in terms
of the graph.

Nevertheless, we do understand the critical point degree of the graph of Figure~\ref{F:Oblique}.
The Newton polytope has  normalized volume 54, and a computation at numerical values of the parameters shows that the Bloch variety has 40
critical points.
We show that the contribution of asymptotic critical points is at least 14.

There are asymptotic critical points arising from the two vertical faces, and each of the four parallelogram faces, $F$ and $G$, and
their reflections in the $\lambda$-axis.
 \begin{enumerate}
   \item Each vertical face has normalized area four, so $\Nvert\geq 8$.
   \item The face $F$ and its reflection each have $\Ndisc=1$, which accounts for at least two more.
   \item The face $G$ and its reflection each contribute two, for at least four more.
 \end{enumerate}
As $8+2+4=14$, we see that the critical point degree is at most $54-14=40$.
With the computation, this shows that the critical point degree of the graph is exactly 40.
\end{Example}
%%%%%%%%%%%%%%%%%%%%%%%%%%%%%%%%%%%%%%%%%%%%%%%%%%%%%%%%%%%%%%%%%%%%%%%%%%%%%%%%%%%%%%%%%%%%%%%%%%%%

%%%%%%%%%%%%%%%%%%%%%%%%%%%%%%%%%%%%%%%%%%%%%%%%%%%%%%%%%%%%%%%%%%%%%%%%%%%%%%%%%%%%%%%%%%%%%%%%%%%%
\def\cprime{$'$}

\providecommand{\bysame}{\leavevmode\hbox to3em{\hrulefill}\thinspace}
\providecommand{\MR}{\relax\ifhmode\unskip\space\fi MR }
\providecommand{\MRhref}[2]{%
  \href{http://www.ams.org/mathscinet-getitem?mr=#1}{#2}
}
\providecommand{\href}[2]{#2}

\end{document}